\documentclass[12pt,twoside]{article}
\usepackage[margin=1in]{geometry}
\usepackage{fancyhdr}
\usepackage{graphicx}
\usepackage{amssymb}
\usepackage{amsmath}
\usepackage{amsthm}
\usepackage{amsfonts}
\usepackage{mathrsfs}
\usepackage{mathtools}
\usepackage{bm}
\usepackage{color}
\usepackage{setspace}
\usepackage{exscale}
\usepackage{relsize}
\usepackage{float}
\usepackage{cite}
\usepackage{hyperref}
\usepackage{cleveref}
\usepackage{nicefrac}
\usepackage{ulem}
\usepackage{caption}
\usepackage{soul}
\usepackage{booktabs,multirow} 
\usepackage{array} 
\usepackage{paralist} 
\usepackage{fancyhdr} 
\theoremstyle{plain}                    
\newtheorem{thm}{Theorem}[section]

\newtheorem{rmk}[thm]{Remark}
\newenvironment{acknowledgment}{{\flushleft \bf Acknowledgment:}}{}

\usepackage{tikz}
\usetikzlibrary{positioning}
\usepackage{xcolor}
\allowdisplaybreaks[1]
\numberwithin{equation}{section}
\numberwithin{figure}{section}
\numberwithin{table}{section}
\newcommand\eref[1]{(\ref{#1})}
\newcommand\fref[1]{Figure~\ref{#1}}
\newcommand*\xbar[1]{%
  \hbox{%
    \vbox{%
      \hrule height 0.5pt 
      \kern0.4ex
      \hbox{%
        \kern-0.05em
        \ensuremath{#1}%
        \kern-0.00em
      }%
    }%
  }%
}

\newcommand{\dt}{\Delta t}
\newcommand{\dx}{\Delta x}
\newcommand{\dxi}{\Delta\xi}
\newcommand{\deta}{\Delta\eta}
\newcommand{\dy}{\Delta y}
\newcommand{\hf}{{\frac{1}{2}}}
\newcommand{\kph}{{k+\frac{1}{2}}}
\newcommand{\kmh}{{k-\frac{1}{2}}}
\newcommand{\kpmh}{{k\pm\frac{1}{2}}}
\newcommand{\mph}{{m+\frac{1}{2}}}
\newcommand{\mmh}{{m-\frac{1}{2}}}
\newcommand{\jph}{{j+\frac{1}{2}}}
\newcommand{\jmh}{{j-\frac{1}{2}}}
\newcommand{\jpmh}{{j\pm\frac{1}{2}}}
\newcommand{\lph}{{\ell+\frac{1}{2}}}
\newcommand{\lmh}{{\ell-\frac{1}{2}}}
\newcommand{\mU}{\bm{U}}
\newcommand{\mF}{\bm{F}}
\newcommand{\mS}{\bm{S}}
\newcommand{\softd}{d\hspace{-0.2mm}'}
\graphicspath{{Figures/}}

\hypersetup{colorlinks=true,
            breaklinks=true,
            urlcolor=blue,
            linkcolor=black,
            bookmarksopen=false,
            filecolor=black,
            citecolor=black,
            linkbordercolor=blue
}

\title{New High-Order Numerical Methods for Hyperbolic Systems of Nonlinear PDEs with Uncertainties}
\author{Alina Chertock\thanks{Department of Mathematics, North Carolina State University, Raleigh, NC, USA;
{\href{mailto:chertock@math.ncsu.edu}{chertock@math.ncsu.edu}}}, Michael Herty\thanks{Department of Mathematics, RWTH Aachen University,
Aachen, Germany; {\href{mailto:herty@igpm.rwth-aachen.de}{herty@igpm.rwth-aachen.de}}}, Arsen S. Iskhakov\thanks{Department of Mathematics,
North Carolina State University, Raleigh, NC, USA; {\href{mailto:aiskhak@ncsu.edu}{aiskhak@ncsu.edu}}}, Safa Janajra\thanks{Department of
Mathematics, North Carolina State University, Raleigh, NC, USA; {\href{mailto:sjanajr@ncsu.edu}{sjanajr@ncsu.edu}}},\\
Alexander Kurganov\thanks{Department of Mathematics, Shenzhen International Center for Mathematics, and Guangdong Provincial Key Laboratory
of Computational Science and Material Design, Southern University of Science and Technology, Shenzhen 518055, China;
{\href{mailto:alexander@sustech.edu.cn}{alexander@sustech.edu.cn}}}, and M\'{a}ria Luk\'a\v{c}ov\'a-Medvi{\softd}ov\'a\thanks{Institute of
Mathematics, Johannes Gutenberg University Mainz, Mainz, Germany; {\href{mailto:lukacova@uni-mainz.de}{lukacova@uni-mainz.de}}}}
\date{}

\begin{document}
\maketitle
\begin{abstract}
In this paper, we develop new high-order numerical methods for hyperbolic systems of nonlinear partial differential equations (PDEs) with
uncertainties. The new approach is realized in the semi-discrete finite-volume framework and is based on fifth-order weighted essentially
non-oscillatory (WENO) interpolations in (multidimensional) random space combined with second-order piecewise linear reconstruction in
physical space. Compared with spectral approximations in the random space, the presented methods are essentially non-oscillatory as they do
not suffer from the Gibbs phenomenon while still achieving high-order accuracy. The new methods are tested on a number of numerical examples
for both the Euler equations of gas dynamics and the Saint-Venant system of shallow-water equations. In the latter case, the methods are
also proven to be well-balanced and positivity-preserving.
\end{abstract}

\smallskip
\noindent
{\bf Keywords:} Hyperbolic conservation and balance laws with uncertainties, finite-volume methods, central-upwind schemes, weighted
essentially non-oscillatory (WENO) interpolations.

\medskip
\noindent
{\bf AMS subject classification:} 65M08, 76M12, 35L65, 35R60.

\section{Introduction}\label{sec1}
Many important scientific problems contain sources of uncertainties in parameters, initial and boundary conditions (ICs and BCs), etc. In
partial differential equations (PDEs), uncertainties may be described with the help of random variables. In this paper, the focus is placed
on hyperbolic systems of conservation and balance laws with uncertainties. Such systems read as
\begin{equation}
\bm U_t+\nabla_{\bm x}\cdot\bm F=\bm S,
\label{1.1}
\end{equation}
where $\bm U(\bm x,t,\bm\xi)\in\mathbb R^N$ is an unknown vector-function, $\bm F(\bm U):\mathbb R^N\to\mathbb R^N$ are the flux functions,
and $\bm S(\bm U,\bm x,\bm\xi)$ is a source term. Furthermore, $t$ is time and $\bm x\in\mathbb R^d$ are spatial variables. Without loss of
generality, we assume that $\bm\xi\in\Xi\subset\mathbb R^s$ are real-valued random variables. We denote by $(\Xi,{\cal F},\nu)$ the
underlying probability space, where $\Xi$ is a set of events, ${\cal F}(\Xi)$ is the $\sigma$-algebra of Borel measurable sets, and
$\nu(\bm\xi):\Xi\to\mathbb R_+$ is the probability density function (PDF), $\nu\in L^1(\Xi)$.

The system \eref{1.1} arises in many applications, including fluid dynamics, geophysics, electromagnetism, meteorology, and astrophysics, to
name a few. Quantifying uncertainties that appear as input quantities, as well as in the ICs and BCs due to empirical approximations or
measuring errors, is essential as it helps to perform sensitivity analysis and provides guidance to improve models.

We are interested in developing highly accurate and robust numerical methods for \eref{1.1}. Several methods have already been developed.
Monte Carlo-type methods (see, e.g., \cite{AM17,MSS13,MSS12a,MSS12b,MS12,MS17}) are reliable but not computationally efficient due to the
large number of realizations required to achieve accurate approximation of statistical moments. Another widely used approach for solving
\eref{1.1} is the generalized polynomial chaos (gPC), where the solution is sought in terms of a series of orthogonal polynomials with
respect to the probability density in $\bm\xi$; see, e.g., \cite{LK10,PIN-book,Xiu09,Xiu10}. There are two types of gPC methods: intrusive
and non-intrusive ones. In non-intrusive algorithms, such as stochastic collocation (gPC-SC) methods, one seeks to satisfy the governing
equations at a discrete set of nodes in the random space, employing the same numerical solver as for the deterministic problem and then
using interpolation and quadrature rules to evaluate statistical moments numerically; see, e.g.,
\cite{Xiu09,SMS13,Xiu07,Xiu10,PIN-book,TZ10}. In intrusive approaches, such as stochastic Galerkin (gPC-SG) methods, the gPC expansions are
substituted into the governing equations and projected by a Galerkin approximation to obtain deterministic equations for the expansion
coefficients; see, e.g., \cite{TLNE10a,TLNE10,DPL13,DKMSS20,DEN21a,DEN21b,GHI21,GH20,GHS19,JS19,PIN-book,PIN14,SS18,WTX17,WXZ21,ZS22}.
Solving the coefficient equations yields the statistical moments of the original uncertain problem solution.

Several challenges are associated with applying the gPC-SC and gPC-SG methods to nonlinear hyperbolic systems \eref{1.1}. It is well-known
that spectral-type gPC-based methods exhibit fast convergence when the solution depends smoothly on the random parameters. However, one of
the main problems in using these methods is related to the lack of smoothness of their solutions, which may break down in finite time as a
result of the development of shock and contact waves. Although these discontinuities appear in spatial variables, their propagation speed
can be affected by uncertainty, causing discontinuities in random variables and triggering aliasing errors \cite{Cho74} and Gibbs-type
phenomena \cite{LKNG04} (see also \cite{WanK-SISC06,JAX11,AC13,Barth}). Another open question in gPC-based stochastic methods is the
representation of strictly positive quantities such as gas density or water depth and/or the enforcement of discrete bound-preserving
constraints; see, e.g., \cite{DEN21a,DEN21b,GHS19,LK10,PIN14}. It can also be shown that the gPC-based methods, which are highly accurate for
moment estimation, might fail to approximate PDF, even for one-dimensional (1-D) noise \cite{DFS19}. An additional numerical
difficulty associated with implementing the gPC-SG methods is posed by the fact that the deterministic Jacobian of the projected system
differs from the random Jacobian of the original system. Therefore, when applied to general nonlinear hyperbolic systems, the gPC-SG method
results in a system for expansion coefficients, which is not necessarily globally hyperbolic \cite{DPL13}. Consequently, additional effort
is needed to develop stable numerical methods for (potentially ill-posed) systems for the gPC coefficients; see, e.g.,
\cite{DKMSS20,LK10,PDL,PIN14,PIN-book,TLNE10,DEN21a,DEN21b,GHI21,GH20,GHS19,SS18,WTX17,WXZ21}.

In this paper, we propose a new approach in which we conduct the approximation in the random space using weighted essentially
non-oscillatory (WENO) interpolants rather than orthogonal polynomial expansions, which are highly oscillatory in the case of nonsmooth
solutions. The new approach is realized in the semi-discrete finite-volume framework. The system \eref{1.1} is integrated over the
$(\bm x,\bm\xi)$-cells and the solution is obtained in terms of cell averages, which are evolved in time using numerical fluxes in the
$\bm x$-directions. These fluxes are evaluated with the help of a second-order piecewise linear reconstruction in $\bm x$ and a
Gauss-Legendre quadrature in $\bm\xi$ implemented using high-order WENO interpolants. This allows one to achieve high accuracy in $\bm\xi$
even in the generic case of discontinuous solutions. We refer the reader to related work in \cite{AM17,AbgTok,Abgrall_2,PTT,Tokareva}, where
a similar idea has been used in the framework of stochastic finite-volume methods. In this work, we implement a second-order finite-volume
method in physical space. It is combined with a high-order WENO interpolation in the random space. Notice that by using a high-order WENO
interpolation, we not only keep a high-order approximation in the random space but also properly (without the Gibbs phenomenon) resolve
discontinuities that may propagate into the random directions.

The new family of methods may use different numerical fluxes, piecewise-linear reconstructions, and high-order interpolations. A particular
choice made in this work is the following. We use the Riemann-problem-solver-free central-upwind (CU) fluxes introduced in \cite{KNP,KPshw};
the generalized minmod reconstruction in $\bm x$ (see, e.g., \cite{LN,NT,Swe}); the fifth-order Gauss-Legendre quadrature; and the recently
proposed fifth-order affine-invariant WENO-Z (Ai-WENO-Z) interpolation in $\bm\xi$ (see \cite{DLWW22,LLWDW23,WD22}), which is a stabilized
version of the original WENO-Z interpolation (see, e.g., \cite{DLGW,JSZ,Liu17,WLGD}). It should be observed that the lack of BCs in the
random space imposes an additional difficulty, as it requires the development of a special high-order interpolation technique near the
boundary. In this paper, we overcome this difficulty by designing a new one-sided Ai-WENO-Z interpolation. In addition, we restrict our
consideration to the simplest 1-D case ($d=s=1$) and two higher-dimensional extensions ($d=1$, $s=2$ and $d=2$, $s=1$). We test the
resulting scheme on several numerical examples for {the Euler equations of gas dynamics and Saint-Venant system of
shallow water equations, both considered with different probability densities in random variables.} For the shallow water application, we
use the same technique as in \cite{KPshw} to enforce the well-balanced (WB) property, that is, to make the scheme capable of exactly
preserving ``lake-at-rest''/still-water steady-state solutions.

{
The paper is organized as follows. In \S\ref{sec2}, we describe the proposed 1-D semi-discrete scheme with a detailed explanation of
numerical fluxes, reconstruction in the physical space and interpolation in the random space. \S\ref{sec3} extends the proposed methodology
for the cases of higher dimensions: $d=1$, $s=2$ and $d=2$, $s=1$. In \S\ref{sec4}, the evolution in time is discussed. In \S\ref{sec5}, we
illustrate the methodology using several numerical examples for the Euler equations of gas dynamics. In \S\ref{sec6}, we develop
well-balanced and positivity-preserving vesions of the proposed schemes for the Saint-Venant system with uncertainties and present
additional numerical examples. Finally, \S\ref{sec7} contains concluding remarks and ideas for future works.
}

\section{One-Dimensional Semi-Discrete Scheme}\label{sec2}
In the 1-D case ($d=s=1$), the system \eref{1.1} reduces to
\begin{equation}
\bm U_t+\bm F(\bm U)_x=\bm S(\bm U,x,\xi),
\label{2.1}
\end{equation}
{which we multiply by the PDF $\nu(\xi)$ and rewrite in terms of $\nu(\xi)\bm U(x,\xi,t)$:
\begin{equation}
(\nu\bm U)_t+\left[\nu\bm F(\bm U)\right]_x=\nu\bm S(\bm U,x,\xi).
\label{2.2f}
\end{equation}
We discretize \eref{2.2f}} on a finite-volume mesh with cells $C_{j,\ell}=[x_\jmh,x_\jph]\times[\xi_\lmh,\xi_\lph]$, $j=1,\ldots,N_x$,
$\ell=1,\ldots,N_\xi$, taken to be uniform both in $x$ and $\xi$. Here, $x_\jph-x_\jmh\equiv\dx$ and $\xi_\lph-\xi_\lmh\equiv\dxi$, and the
cells are centered at $(x_j,\xi_\ell)=\Big(\frac{x_\jmh+x_\jph}{2},\frac{\xi_\lmh+\xi_\lph}{2}\Big)$.

The computed {weighted cell averages of $\bm U$} are assumed to be available at a certain time~$t$:
\begin{equation}
\xbar{\bm U}_{j,\ell}(t)\approx\frac{1}{\dx\dxi}\iint\limits_{C_{j,\ell}}\bm U(x,\xi,t)\,\nu(\xi)\,{\rm d}x\,{\rm d}\xi.
\label{ubar}
\end{equation}
These quantities are evolved in time using a semi-discretization of \eref{2.2f}:
\begin{equation}
\frac{\rm d}{{\rm d}t}\,\xbar{\bm U}_{j,\ell}(t)=-\frac{\bm{{\cal F}}_{\jph,\ell}(t)-\bm{{\cal F}}_{\jmh,\ell}(t)}{\dx}+
\xbar{\bm S}_{j,\ell}(t),
\label{2.2}
\end{equation}
where $\bm{{\cal F}}_{\jpmh,\ell}$ are numerical fluxes (see \S\ref{sec21}) and $\xbar{\bm S}_{j,\ell}$ are approximated cell averages of
the source term:
\begin{equation}
\xbar{\bm S}_{j,\ell}(t)\approx\frac{1}{\dx\dxi}\iint\limits_{C_{j,\ell}}\bm S(\bm U,x,\xi)\,\nu(\xi)\,{\rm d}x\,{\rm d}\xi.
\label{2.3}
\end{equation}
Note that all indexed quantities that depend on $\bm U$ are time-dependent, but from now on, this dependence will be omitted for the sake of
brevity.

\subsection{Numerical Fluxes and Source Terms}\label{sec21}
The numerical fluxes in \eref{2.2} are obtained by averaging the fluxes in the $\xi$-direction:
\begin{equation}
\bm{{\cal F}}_{\jph,\ell}=\sum_{i=1}^M\mu_i\,\nu(\xi_{\ell_i})\,\bm{{\cal F}}\big(\bm U^-_{\jph,\ell_i},\bm U^+_{\jph,\ell_i}\big),
\label{2.4}
\end{equation}
where the integration in the $\xi$-direction is performed using a Gauss-Legendre quadrature on the cell $[\xi_\lmh,\xi_\lph]$ with the
weights $\mu_i$ and nodes $\xi_{\ell_i}\in(\xi_\lmh,\xi_\lph)$, $i=1,\ldots,M$. Furthermore, $\bm{{\cal F}}$ is computed using the
reconstructed one-sided point values of $\bm U$ at $(x_\jph\pm0,\xi_{\ell_i})$, which are denoted by $\bm U^\pm_{\jph,\ell_i}$. We use the
CU numerical flux from \cite{KNP}:
\begin{equation}
\begin{aligned}
\bm{{\cal F}}\big(\bm U^-_{\jph,\ell_i},\bm U^+_{\jph,\ell_i}\big)&=\frac{a_{\jph,\ell}^+\bm F\big(\bm U^-_{\jph,\ell_i}\big)-
a_{\jph,\ell}^-\bm F\big(\bm U^+_{\jph,\ell_i}\big)}{a_{\jph,\ell}^+-a_{\jph,\ell}^-}\\
&+\frac{a_{\jph,\ell}^+a_{\jph,\ell}^-}{a_{\jph,\ell}^+-a_{\jph,\ell}^-}\left(\bm U^+_{\jph,\ell_i}-\bm U^-_{\jph,\ell_i}\right),
\end{aligned}
\label{2.5}
\end{equation}
where $a_{\jph,\ell}^\pm$ are the one-sided local speeds of propagation in the $x$-direction at $(x_\jph,\xi_\ell)$ estimated using the
smallest ($\lambda_1$) and largest ($\lambda_N$) eigenvalues of the Jacobian~$\frac{\partial\bm F}{\partial\bm U}$:
\begin{equation}
\begin{aligned}
&a_{\jph,\ell}^-=\min\left\{\lambda_1\left(\frac{\partial\bm F}{\partial\bm U}\left(\bm U^-_{\jph,\ell}\right)\right),\,
\lambda_1\left(\frac{\partial\bm F}{\partial\bm U}\left(\bm U^+_{\jph,\ell}\right)\right),\,0\right\},\\
&a_{\jph,\ell}^+=\max\left\{\lambda_N\left(\frac{\partial\bm F}{\partial\bm U}\left(\bm U^-_{\jph,\ell}\right)\right),\,
\lambda_N\left(\frac{\partial\bm F}{\partial\bm U}\left(\bm U^+_{\jph,\ell}\right)\right),\,0\right\}.
\end{aligned}
\label{2.6}
\end{equation}

Finally, the integral on the right-hand side (RHS) of \eref{2.3} is discretized in two steps. First, we use the same Gauss-Legendre
quadrature as in \eref{2.4} to numerically integrate with respect to $\xi$ and obtain
\begin{equation}
\xbar{\bm S}_{j,\ell}\approx\frac{1}{\dx}\sum_{i=1}^M\mu_i\,\nu(\xi_{\ell_i})
\int\limits_{x_\jmh}^{x_\jph}\bm S(\bm U(x,\xi_{\ell_i},t),x,\xi_{\ell_i})\,{\rm d}x.
\label{2.7}
\end{equation}
Second, depending on the system at hand, we choose a proper quadrature for the integrals on the RHS of \eref{2.7}. In the case of the
Saint-Venant system considered in \S\ref{sec6}, a special quadrature for the geometric source term will be introduced to enforce a WB
property of the resulting scheme.

\subsection{Reconstruction of the Point Values}\label{sec22}
This section describes reconstruction procedures in both physical and random space. We use a piecewise linear reconstruction in $x$ and
a high-order interpolation in $\xi$. The order of the resulting scheme in $\xi$ depends on the accuracy of the quadratures used to evaluate
the numerical flux and source terms in \eref{2.4} and \eref{2.7}. In this paper, we use the fifth-order Gauss-Legendre quadrature with $M=3$
and the corresponding weights
\begin{equation}
\mu_1=\mu_3=\frac{5}{18},\quad\mu_2=\frac{4}{9},
\label{2.11f}
\end{equation}
and nodes $\xi_{\ell_i}$, $i=1,2,3$ with
\begin{equation}
\xi_{\ell_1}=\xi_{\ell-\kappa},\quad\xi_{\ell_2}=\xi_\ell,\quad\xi_{\ell_3}=\xi_{\ell+\kappa},\quad\kappa=\frac{1}{2}\sqrt{\frac{3}{5}}.
\label{2.12f}
\end{equation}
We will therefore first apply a piecewise linear reconstruction in the $x$-direction to obtain the one-sided point values
$\bm U_{\jph,\ell}^\pm$ at the midpoints of the cell interfaces $(x_\jph,\xi)$, $\xi\in(\xi_\lmh,\xi_\lph)$, and then use these values to
interpolate the one-sided point values $\bm U_{\jph,\ell\pm\kappa}^\pm$ at the Gauss-Legendre points $(x_\jph,\xi_{\ell\pm\kappa})$ along
the corresponding cell interfaces. The latter is done with the help of a fifth-order Ai-WENO-Z interpolation \cite{DLWW22,LLWDW23,WD22} in
the $\xi$-directions.

\subsubsection{Reconstruction in Physical Space}\label{sec221}
In order to reconstruct the one-sided point values $\bm U_{\jph,\ell}^\pm$, we first compute the slopes
{$(\,\xbar{\bm U}_x)_{j,\ell}\approx\nu(\xi_\ell)\,\bm U_x(x_j,\xi_\ell,t)$} and then use them to obtain
{
\begin{equation}
\bm U_{\jmh,\ell}^+=\frac{1}{\nu(\xi_\ell)}\Big[\,\xbar{\bm U}_{j,\ell}-\frac{\dx}{2}(\,\xbar{\bm U}_x)_{j,\ell}\Big]\quad\mbox{and}\quad
\bm U_{\jph,\ell}^-=\frac{1}{\nu(\xi_\ell)}\Big[\,\xbar{\bm U}_{j,\ell}+\frac{\dx}{2}(\,\xbar{\bm U}_x)_{j,\ell}\Big].
\label{2.8}
\end{equation}
}
These values are second-order accurate and non-oscillatory provided the slopes are approximated using an appropriate nonlinear limiter with
at least the first order accuracy. In this paper, we use the generalized minmod limiter (see, e.g., \cite{LN,NT,Swe}):
\begin{equation}
{(\,\xbar{\bm U}_x)_{j,\ell}}={\rm minmod}\left(\theta\,\frac{\xbar{\bm U}_{j,\ell}-\xbar{\bm U}_{j-1,\ell}}{\dx},
\frac{\xbar{\bm U}_{j+1,\ell}-\xbar{\bm U}_{j-1,\ell}}{2\dx},\theta\,\frac{\xbar{\bm U}_{j+1,\ell}-\xbar{\bm U}_{j,\ell}}{\dx}\right),
\label{2.9}
\end{equation}
where the minmod function
\begin{equation}
{\rm{minmod}}\left(c_1,c_2,\ldots\right) = \left\{\begin{aligned}
&\min(c_1,c_2,\ldots)&&\mbox{if}~c_i>0,~\forall i,\\
&\max(c_1,c_2,\ldots)&&\mbox{if}~c_i<0,~\forall i,\\
&0&&\mbox{otherwise},
\end{aligned}\right.
\label{2.10}
\end{equation}
is applied componentwise. The parameter $\theta\in[1,2]$ controls the amount of numerical dissipation: larger values yield less dissipative
schemes, but the computed quantities might be oscillatory.
{
\begin{rmk}
It should be observed that formula \eref{2.8} is well defined even when $\nu(\xi_\ell)\simeq0$ due to the same scaling with respect to $\nu$
in \eref{2.8} and \eref{ubar}, \eref{2.4}, and \eref{2.7}.
\end{rmk}
}

\subsubsection{Interpolation in Random Space}\label{sec222}
Equipped with $\bm U_{\jph,\ell}^\pm$, we proceed with the calculation of the remaining point values $\bm U_{\jph,\ell\pm\kappa}^\pm$. For
the left-sided values $\bm U_{\jph,\ell\pm\kappa}^-$, the interpolation is based on $\bm U_{\jph,\ell}^-$ and its neighboring (in $\xi$)
values $\bm U_{\jph,\ell\pm1}^-$ and $\bm U_{\jph,\ell\pm2}^-$. Similarly, for the right-sided values $\bm U_{\jph,\ell\pm\kappa}^+$, the
interpolation is based on $\bm U_{\jph,\ell}^+$, $\bm U_{\jph,\ell\pm1}^+$, and $\bm U_{\jph,\ell\pm2}^+$.

The Ai-WENO-Z interpolation is applied in a componentwise manner. We denote by $U$ a certain component of $\bm U$ throughout the remainder
of this section. We describe the algorithm to calculate the left values $U_{\jph,\ell\pm\kappa}^-$ (the right values
$U_{\jph,\ell\pm\kappa}^+$ can be obtained in a similar way).

\paragraph{Internal points.} Away from the $\xi$-boundaries, the point values $U_{\jph,\ell\pm\kappa}^-$ are obtained using the ``standard''
fifth-order Ai-WENO-Z interpolation, which is based on the five point values $U^-_{\jph,\ell}$, $U^-_{\jph,\ell\pm1}$, and
$U^-_{\jph,\ell\pm2}$ available from \eref{2.8}:
\begin{equation}
U_{\jph,\ell\pm\kappa}^-=\sum_{i=0}^2\omega_{i,\pm}{\cal P}_i(\xi_{\ell\pm\kappa}).
\label{2.11}
\end{equation}
Here, ${\cal P}_0(\xi)$, ${\cal P}_1(\xi)$, and ${\cal P}_2(\xi)$ are the parabolic interpolants obtained on the 3-point stencils
$\big\{(x_\jph,\xi_{\ell-2+i}),(x_\jph,\xi_{\ell-1+i}),(x_\jph,\xi_{\ell+i})\big\}$ for $i=0$, 1, and 2, respectively, their corresponding
values are
\begin{equation*}
\begin{aligned}
&{\cal P}_0(\xi_{\ell\pm\kappa})=\frac{3\pm2\sqrt{15}}{40}\,U_{\jph,\ell-2}^--\frac{3\pm4\sqrt{15}}{20}\,U_{\jph,\ell-1}^-
+\frac{43\pm6\sqrt{15}}{40}\,U_{\jph,\ell}^-,\\[0.2ex]
&{\cal P}_1\left(\xi_{\ell\pm\kappa}\right)=\frac{3\mp2\sqrt{15}}{40}\,U_{\jph,\ell-1}^-+\frac{17}{20}\,U_{\jph,\ell}^-
+\frac{3\pm2\sqrt{15}}{40}\,U_{\jph,\ell+1}^-,\\[0.2ex]
&{\cal P}_2\left(\xi_{\ell\pm\kappa}\right)=\frac{43\mp6\sqrt{15}}{40}\,U_{\jph,\ell}^--\frac{3\mp4\sqrt{15}}{20}\,U_{\jph,\ell+1}^-
+\frac{3\mp2\sqrt{15}}{40}\,U_{\jph,\ell+2}^-,
\end{aligned}
\end{equation*}
and the weights $\omega_{i,\pm}$ $(i=0,1,2)$ are given by
\begin{equation}
\omega_{i,\pm}=\frac{\alpha_{i,\pm}}{\alpha_{0,\pm}+\alpha_{1,\pm}+\alpha_{2,\pm}},\quad\alpha_{i,\pm}=d_{i,\pm}
\left[1+\left(\frac{\tau}{\beta_i+\varepsilon_1\varphi_i^2+\varepsilon_2}\right)^2\,\right],\quad\tau=|\beta_2-\beta_0|,
\label{2.12}
\end{equation}
with the so-called linear weights
\begin{equation*}
d_{0,\pm}=\frac{43\mp6\sqrt{15}}{240},\quad d_{1,\pm}=\frac{77}{120},\quad d_{2,\pm}=\frac{43\pm6\sqrt{15}}{240},
\end{equation*}
selected in a unique way that guarantees the fifth-order accuracy of \eref{2.11} when $\tau=0$. In \eref{2.12}, $\varphi_i=
\frac{1}{5}\sum_{m=i-2}^{i+2}|U_{\jph,m}^--\psi_i|$ with $\psi_i=\frac{1}{5}\sum_{m=i-2}^{i+2}U_{\jph,m}^-$ and
$1\gg\varepsilon_1\gg\varepsilon_2$ being two small positive numbers taken $\varepsilon_1=10^{-12}$ and $\varepsilon_2=10^{-40}$ in all of
the numerical examples presented in \S\ref{sec5} and \S\ref{sec6}. Finally, $\beta_i$ in \eref{2.12} are the smoothness indicators for the
corresponding interpolants ${\cal P}_i(\xi)$:
\begin{equation}
\beta_i=\sum_{r=1}^2(\dxi)^{2r-1}\int\limits_{\xi_\lmh}^{\xi_\lph}\left(\frac{{\rm d}^r{\cal P}_i}{{\rm d}\xi^r}\right)^2{\rm d}\xi,
\quad i=0,1,2.
\label{2.13}
\end{equation}
Evaluating the integrals in \eref{2.13} results in
\begin{equation*}
\begin{aligned}
&\beta_0=\frac{13}{12}\Big(U_{\jph,\ell-2}^--2U_{\jph,\ell-1}^-+U_{\jph,\ell}^-\Big)^2+\frac{1}{4}\Big(U_{\jph,\ell-2}^--4U_{\jph,\ell-1}^-+
3U_{\jph,\ell}^-\Big)^2,\\[0.4ex]
&\beta_1=\frac{13}{12}\Big(U_{\jph,\ell-1}^--2U_{\jph,\ell}^-+U_{\jph,\ell+1}^-\Big)^2+\frac{1}{4}\Big(U_{\jph,\ell-1}^--U_{\jph,\ell+1}^-
\Big)^2,\\[0.4ex]
&\beta_2=\frac{13}{12}\Big(U_{\jph,\ell}^--2U_{\jph,\ell+1}^-+U_{\jph,\ell+2}^-\Big)^{2}+\frac{1}{4}\Big(3U_{\jph,\ell}^--4U_{\jph,\ell+1}^-
+U_{\jph,\ell+2}^-\Big)^{2}.
\end{aligned}
\end{equation*}

\paragraph{Near the $\xi$-boundary points.} It is important to point out that no BCs are imposed for the solution of \eref{2.1} in the
$\xi$-direction. Therefore, the computation of the point values in cells located within the distance of $2\dxi$ from the $\xi$-boundaries
should be carried out differently. We use one-sided stencils to construct the three parabolic interpolants in \eref{2.11}.

Assume that $\xi\in[a,b]$: the cells $C_{j,1}$ and $C_{j,2}$ for all $j$ are located near the $\xi=\xi_\hf=a$ boundary. Once again, we only
present the algorithm for computing the point values $U_{\jph,1\pm\kappa}^-$ and $U_{\jph,2\pm\kappa}^-$, where $\kappa$ is given in
\eref{2.12f}, and note that the point values $U_{\jph,1\pm\kappa}^+$ and $U_{\jph,2\pm\kappa}^+$ can be obtained similarly. The
corresponding point values in the cells near the $\xi=\xi_{N_\xi+\hf}=b$ boundary are computed mirror-symmetrically.

We obtain the point values $U_{\jph,1\pm\kappa}^-$ and $U_{\jph,2\pm\kappa}^-$ using the fifth-order Ai-WENO-Z interpolation based on the
five point values $U^-_{\jph,\ell}$, $\ell=1,\ldots,5$, available from \eref{2.8}, and the parabolic interpolants obtained on the 3-point
stencils $\big\{(x_\jph,\xi_{1+i}),(x_\jph,\xi_{2+i}),(x_\jph,\xi_{3+i})\big\}$ for $i=0$, 1, and 2, respectively. The desired values of $U$
are then computed by
\begin{equation*}
U_{\jph,\ell\pm\kappa}^-=\sum_{i=0}^2\omega_{i,\ell\pm\kappa}{\cal P}_i(\xi_{\ell\pm\kappa}),\quad\ell=1,2,
\end{equation*}
with
\begin{align*}
&{\cal P}_0(\xi_{1\pm\kappa})=\frac{43\mp6\sqrt{15}}{40}\,U_{\jph,1}^--\frac{3\mp4\sqrt{15}}{20}\,U_{\jph,2}^-+
\frac{3\mp2\sqrt{15}}{40}\,U_{\jph,3}^-,\\[0.2ex]
&{\cal P}_1(\xi_{1\pm\kappa})=\frac{123\mp10\sqrt{15}}{40}\,U_{\jph,2}^--\frac{63\mp8\sqrt{15}}{20}\,U_{\jph,3}^-+
\frac{43\mp6\sqrt{15}}{40}\,U_{\jph,4}^-,\\[0.2ex]
&{\cal P}_2(\xi_{1\pm\kappa})=\frac{243\mp14\sqrt{15}}{40}\,U_{\jph,3}^--\frac{163\mp12\sqrt{15}}{20}\,U_{\jph,4}^-+
\frac{123\mp10\sqrt{15}}{40}\,U_{\jph,5}^-,\\[0.5ex]
&{\cal P}_0(\xi_{2\pm\kappa})=\frac{3\mp2\sqrt{15}}{40}\,U_{\jph,1}^-+\frac{17}{20}\,U_{\jph,2}^-+\frac{3\pm2\sqrt{15}}{40}\,U_{\jph,3}^-,
\\[0.2ex]
&{\cal P}_1(\xi_{2\pm\kappa})=\frac{43\mp6\sqrt{15}}{40}\,U_{\jph,2}^--\frac{3\mp4\sqrt{15}}{20}\,U_{\jph,3}^-+
\frac{3\mp2\sqrt{15}}{40}\,U_{\jph,4}^-,\\[0.2ex]
&{\cal P}_2(\xi_{2\pm\kappa})=\frac{123\mp10\sqrt{15}}{40}\,U_{\jph,3}^--\frac{63\mp8\sqrt{15}}{20}\,U_{\jph,4}^-+
\frac{43\mp6\sqrt{15}}{40}\,U_{\jph,5}^-.
\end{align*}
As in the case of internal points, the nonlinear weights $\omega_{i,\ell\pm\kappa}$ $(i=0,1,2,\,\,\ell=1,2)$ can be computed from the
corresponding linear weights $d_{i,\ell\pm\kappa}$ and the smoothness indicators $\beta_{i,\ell}$. However, in this case, some of the linear
weights are negative, which may lead to the appearance of nonphysical reconstructed values. In order to avoid such a situation, we implement
the technique from \cite{SHS} and compute the nonlinear weights in \eref{2.11} by
\begin{equation*}
\omega_{i,\ell\pm\kappa}=\widetilde\delta_{\ell\pm\kappa}\,\widetilde\omega_{i,\ell\pm\kappa}-\widehat\delta_{\ell\pm\kappa}\,
\widehat\omega_{i,\ell\pm\kappa},
\end{equation*}
where
\begin{equation}
\begin{aligned}
&\widetilde\omega_{i,\ell\pm\kappa}=\frac{\widetilde\alpha_{i,\ell\pm\kappa}}{\widetilde\alpha_{0,\ell\pm\kappa}
+\widetilde\alpha_{1,\ell\pm\kappa}+\widetilde\alpha_{2,\ell\pm\kappa}},~~\widetilde\alpha_{i,\ell\pm\kappa}=\widetilde d_{i,\ell\pm\kappa}
\left[1+\left(\frac{\beta_{2,\ell}-\beta_{0,\ell}}{\beta_{i,\ell}+\varepsilon_1\varphi_i^2+\varepsilon_2}\right)^2\,\right],\\[0.2ex]
&\widehat\omega_{i,\ell\pm\kappa}=\frac{\widehat\alpha_{i,\ell\pm\kappa}}{\widehat\alpha_{0,\ell\pm\kappa}+\widehat\alpha_{1,\ell\pm\kappa}
+\widehat\alpha_{2,\ell\pm\kappa}},~~\widehat\alpha_{i,\ell\pm\kappa}=\widehat d_{i,\ell\pm\kappa}
\left[1+\left(\frac{\beta_{2,\ell}-\beta_{0,\ell}}{\beta_{i,\ell}+\varepsilon_1\varphi_i^2+\varepsilon_2}\right)^2\,\right],
\end{aligned}
\label{2.14}
\end{equation}
and
\begin{align*}
&\widetilde d_{0,1-\kappa}=\frac{18489-782\sqrt{15}}{17787},&&\widetilde d_{1,1-\kappa}=\frac{-711+640\sqrt{15}}{17787},
&&\widetilde d_{2,1-\kappa}=\frac{9+142\sqrt{15}}{17787},\\
&\widetilde d_{0,1+\kappa}=\frac{77142-5368\sqrt{15}}{73767},&&\widetilde d_{1,1+\kappa}=\frac{-2844+5048\sqrt{15}}{73767},
&&\widetilde d_{2,1+\kappa}=\frac{-531+320\sqrt{15}}{73767},\\
&\widetilde d_{0,2-\kappa}=\frac{38022+2648\sqrt{15}}{73767},&&\widetilde d_{1,2-\kappa}=\frac{36276-2968\sqrt{15}}{73767},
&&\widetilde d_{2,2-\kappa}=\frac{-531+320\sqrt{15}}{73767},\\
&\widetilde d_{0,2+\kappa}=\frac{123-10\sqrt{15}}{240},&&\widetilde d_{1,2+\kappa}=\frac{57+4\sqrt{15}}{120},
&&\widetilde d_{2,2+\kappa}=\frac{3+2\sqrt{15}}{240},\\[0.5ex]
&\widehat d_{0,1-\kappa}=\frac{8769-1342\sqrt{15}}{5334},&&\widehat d_{1,1-\kappa}=\frac{-1662+640\sqrt{15}}{2667},
&&\widehat d_{2,1-\kappa}=\frac{-111+62\sqrt{15}}{5334},\\
&\widehat d_{0,1+\kappa}=\frac{19131-1564\sqrt{15}}{17607},&&\widehat d_{1,1+\kappa}=\frac{-942+1244\sqrt{15}}{17607},
&&\widehat d_{2,1+\kappa}=\frac{-582+320\sqrt{15}}{17607},\\
&\widehat d_{0,2-\kappa}=\frac{9171+524\sqrt{15}}{17607},&&\widehat d_{1,2-\kappa}=\frac{9018-844\sqrt{15}}{17607},
&&\widehat d_{2,2-\kappa}=\frac{-582+320\sqrt{15}}{17607},\\
&\widehat d_{0,2+\kappa}=0,&&\widehat d_{1,2+\kappa}=0,&&\widehat d_{2,2+\kappa}=0,
\end{align*}
with
\begin{align*}
&\widetilde\delta_{1-\kappa}=\frac{83+8\sqrt{15}}{40},&&\widetilde\delta_{1+\kappa}=\frac{157+2\sqrt{15}}{80},
&&\widetilde\delta_{2-\kappa}=\widetilde\delta_{1+\kappa},&&\widetilde\delta_{2+\kappa}=1,\\[0.2ex]
&\widehat\delta_{1-\kappa}=\frac{43+8\sqrt{15}}{40}&&\widehat\delta_{1+\kappa}=\frac{77+2\sqrt{15}}{80},
&&\widehat\delta_{2-\kappa}=\widehat\delta_{1+\kappa},&&\widehat\delta_{2+\kappa}=0.
\end{align*}
Finally, $\beta_{i,1}$ and $\beta_{i,2}$ in \eref{2.14} are the smoothness indicators for the corresponding parabolic interpolants
${\cal P}_i(\xi)$:
\begin{equation}
\beta_{i,1}=\sum_{r=1}^2(\dxi)^{2r-1}\int\limits_{\xi_\hf}^{\xi_\frac{3}{2}}\left(\frac{{\rm d}^r{\cal P}_i}{{\rm d}\xi^r}\right)^2
{\rm d}\xi,\quad\beta_{i,2}=\sum_{r=1}^2(\dxi)^{2r-1}\int\limits_{\xi_\frac{3}{2}}^{\xi_\frac{5}{2}}
\left(\frac{{\rm d}^r{\cal P}_i}{{\rm d}\xi^r}\right)^2{\rm d}\xi.
\label{2.15}
\end{equation}
Evaluating the integrals in \eref{2.15} results in
\begin{align*}
\beta_{0,1}=\frac{1}{3}&\Big[10\big(U_{\jph,1}^-\big)^2-31U_{\jph,1}^-U_{\jph,2}^-+25\big(U_{\jph,2}^-\big)^2\\
&+11U_{\jph,1}^-U_{\jph,3}^--19U_{\jph,2}^-U_{\jph,3}^-+4\big(U_{\jph,3}^-\big)^2\Big],\\
\beta_{1,1}=\frac{1}{3}&\Big[22\big(U_{\jph,2}^-\big)^2-73U_{\jph,2}^-U_{\jph,3}^-+61\big(U_{\jph,3}^-\big)^2\\
&+29U_{\jph,2}^-U_{\jph,4}^--49U_{\jph,3}^-U_{\jph,4}^-+10\big(U_{\jph,4}^-\big)^2\Big],\\
\beta_{2,1}=\frac{1}{3}&\Big[40\big(U_{\jph,3}^-\big)^2-139U_{\jph,3}^-U_{\jph,4}^-+121\big(U_{\jph,4}^-\big)^2\\
&+59U_{\jph,3}^-U_{\jph,5}^--103U_{\jph,4}^-U_{\jph,5}^-+22\big(U_{\jph,5}^-\big)^2\Big],\\
\beta_{0,2}=\frac{1}{3}&\Big[4\big(U_{\jph,1}^-\big)^2-13U_{\jph,1}^-U_{\jph,2}^-+13\big(U_{\jph,2}^-\big)^2\\
&+5U_{\jph,1}^-U_{\jph,3}^--13U_{\jph,2}^-U_{\jph,3}^-+4\big(U_{\jph,3}^-\big)^2\Big],\\
\beta_{1,2}=\frac{1}{3}&\Big[10\big(U_{\jph,2}^-\big)^2-31U_{\jph,2}^-U_{\jph,3}^-+25\big(U_{\jph,3}^-\big)^2\\
&+11U_{\jph,2}^-U_{\jph,4}^--19U_{\jph,3}^-U_{\jph,4}^-+4\big(U_{\jph,4}^-\big)^2\Big],\\
\beta_{2,2}=\frac{1}{3}&\Big[22\big(U_{\jph,3}^-\big)^2-73U_{\jph,3}^-U_{\jph,4}^-+61\big(U_{\jph,4}^-\big)^2\\
&+29U_{\jph,3}^-U_{\jph,5}^--49U_{\jph,4}^-U_{\jph,5}^-+10\big(U_{\jph,5}^-\big)^2\Big].
\end{align*}
In \eref{2.14}, $\varphi_i=\frac{1}{5}\sum_{m=1}^5|U^-_{\jph,m}-\psi_i|$, $\psi_i=\frac{1}{5}\sum_{m=1}^5U^-_{\jph,m}$, and the small
numbers $1\gg\varepsilon_1\gg\varepsilon_2>0$ are taken $\varepsilon_1=10^{-12}$ and $\varepsilon_2=10^{-40}$ in all of the numerical
examples presented in \S\ref{sec5} and \S\ref{sec6}.

\section{Higher-Dimensional Extensions}\label{sec3}
This section discusses extensions of the numerical method introduced in \S\ref{sec2} to higher-dimensional cases. For simplicity of
presentation, we will consider two particular multidimensional cases with $d=1$, $s=2$ (\S\ref{sec31}) and $d=2$, $s=1$ (\S\ref{sec32}).

\subsection{Case $d=1$ and $s=2$}\label{sec31}
Consider a different version of the studied system \eref{1.1}:
\begin{equation*}
\mU_t+\mF(\mU)_x=\mS(\mU,x,\xi,\eta),
\end{equation*}
{
or, equivalently,
\begin{equation}
(\nu\mU)_t+\left[\nu\mF(\mU)\right]_x=\nu\mS(\mU,x,\xi,\eta),
\label{3.1}
\end{equation}
}
where $\mU=\mU(x,\xi,\eta,t)$ with $\xi$ and $\eta$ being random variables with joint PDF $\nu(\xi,\eta)$.

In order to develop a finite-volume method for \eref{3.1}, we introduce the cells
$C_{j,\ell,m}:=[x_\jmh,x_\jph]\times[\xi_\lmh,\xi_\lph]\times[\eta_\mmh,\eta_\mph]$, $j=1,\ldots,N_x$, $\ell=1,\ldots,N_\xi$,
$m=1,\ldots,N_\eta$, and evolve the corresponding {weighted cell averages},
\begin{equation*}
\xbar{\bm U}_{j,\ell,m}\approx
\frac{1}{\dx\dxi\deta}\iiint\limits_{C_{j,\ell,m}}\bm U(x,\xi,\eta,t)\nu(\xi,\eta)\,{\rm d}x\,{\rm d}\xi\,{\rm d}\eta,
\end{equation*}
in time according to the following semi-discretization of \eref{3.1}:
\begin{equation}
\frac{{\rm d}}{{\rm d}t}\,\xbar{\mU}_{j,\ell,m}=
-\frac{\bm{{\cal F}}_{\jph,\ell,m}-\bm{{\cal F}}_{\jmh,\ell,m}}{\dx}+\xbar{\bm S}_{j,\ell,m},
\label{3.2}
\end{equation}
where $\bm{{\cal F}}_{\jpmh,\ell,m}$ are numerical fluxes and $\xbar{\bm S}_{j,\ell,m}$ are approximated cell averages of the source term.

\subsubsection{Numerical Fluxes and Source Terms}
Similarly to the 1-D case ($d=s=1$) considered in \S\ref{sec21}, the numerical fluxes in \eref{3.2} approximate the average of the
corresponding fluxes, but now in both $\xi$ and $\eta$ directions:
\begin{equation}
\bm{{\cal F}}_{\jph,\ell,m}=\sum_{i,r=1}^M\mu_i\,\mu_r\,\nu(\xi_{\ell_i},\eta_{m_r})\,
\bm{{\cal F}}\big(\bm U^-_{\jph,\ell_i,m_r},\bm U^+_{\jph,\ell_i,m_r}\big),
\label{3.3}
\end{equation}
where $\mu_i$, $i=1,\ldots,M$ are the Gauss-Legendre weights and $(\xi_{\ell_i},\eta_{m_r})$, $i,r=1,\ldots,M$ are the corresponding nodes
in the cell $[\xi_\lmh,\xi_\lph]\times[\eta_\mmh,\eta_\mph]$. In \eref{3.3},
\begin{equation}
\begin{aligned}
\bm{{\cal F}}\big(\bm U^-_{\jph,\ell_i,m_r},\bm U^+_{\jph,\ell_i,m_r}\big)&=\frac{a_{\jph,\ell,m}^+\bm F\big(\bm U^-_{\jph,\ell_i,m_r}\big)
-a_{\jph,\ell,m}^-\bm F\big(\bm U^+_{\jph,\ell_i,m_r}\big)}{a_{\jph,\ell,m}^+-a_{\jph,\ell,m}^-}\\
&+\frac{a_{\jph,\ell,m}^+a_{\jph,\ell,m}^-}{a_{\jph,\ell,m}^+-a_{\jph,\ell,m}^-}\left(\bm U^+_{\jph,\ell_i,m_r}-\bm U^-_{\jph,\ell_i,m_r}
\right)
\end{aligned}
\label{3.4}
\end{equation}
are the CU numerical fluxes with $\bm U^\pm_{\jph,\ell_i,m_r}$ being the reconstructed one-sided point values of $\bm U$ at
$(x_\jph\pm0,\xi_{\ell_i},\eta_{m_r})$ and $a_{\jph,\ell,m}^\pm$ being one-sided local speeds of propagation in the $x$-direction at
$(x_\jph,\xi_\ell,m_r)$. The latter can be estimated as in \eref{2.6}:
\begin{equation}
\begin{aligned}
&a_{\jph,\ell,m}^-=\min\left\{\lambda_1\left(\frac{\partial\bm F}{\partial\bm U}\left(\bm U^-_{\jph,\ell,m}\right)\right),\,
\lambda_1\left(\frac{\partial\bm F}{\partial\bm U}\left(\bm U^+_{\jph,\ell,m}\right)\right),\,0\right\},\\
&a_{\jph,\ell,m}^+=\max\left\{\lambda_N\left(\frac{\partial\bm F}{\partial\bm U}\left(\bm U^-_{\jph,\ell,m}\right)\right),\,
\lambda_N\left(\frac{\partial\bm F}{\partial\bm U}\left(\bm U^+_{\jph,\ell,m}\right)\right),\,0\right\}.
\end{aligned}
\label{3.5}
\end{equation}
The source term on the RHS of \eref{3.2} is discretized similarly to \eref{2.7}:
\begin{equation}
\xbar{\bm S}_{j,\ell,m}=\frac{1}{\dx}\sum_{i,r=1}^M\mu_i\,\mu_r\,\nu(\xi_{\ell_i},\eta_{m_r})
\int\limits_{x_\jmh}^{x_\jph}\bm S\big(\bm U(x,\xi_{\ell_i},\eta_{m_r},t),x,\xi_{\ell_i},\eta_{m_r}\big)\,{\rm d}x.
\label{3.6}
\end{equation}
A particular WB quadrature for the integral in \eref{3.6} in the case of the Saint-Venant system will be specified in \S\ref{sec6}.

\subsubsection{Reconstruction of the Point Values}
The point values $\bm U^\pm_{\jph,\ell_i,m_r}$ used in \eref{3.3}--\eref{3.5} are obtained in two steps, as in \S\ref{sec221}, using a
second-order generalized minimod reconstruction in $x$ and a fifth-order Ai-WENO-Z interpolation in random space. In order to achieve the
fifth order of accuracy, we use the fifth-order Gauss-Legendre quadrature in \eref{3.3} and \eref{3.6} with $M=3$, the corresponding weights
\eref{2.11f}, nodes in $\xi$ given by \eref{2.12f}, and nodes $\eta_{m_r}$, $r=1,2,3$ with
\begin{equation*}
\eta_{m_1}=\eta_{m-\kappa},\quad\eta_{m_2}=\eta_m,\quad\eta_{m_3}=\eta_{m+\kappa}.
\end{equation*}

\paragraph{Reconstruction in Physical Space.} We first evaluate {the values
\begin{equation}
\begin{aligned}
&\bm U_{\jmh,\ell,m}^+=\frac{1}{\nu(\xi_\ell,\eta_m)}\Big[\,\xbar{\bm U}_{j,\ell,m}-\frac{\dx}{2}(\,\xbar{\bm U}_x)_{j,\ell,m}\Big],\\
&\bm U_{\jph,\ell,m}^-=\frac{1}{\nu(\xi_\ell,\eta_m)}\Big[\,\xbar{\bm U}_{j,\ell,m}+\frac{\dx}{2}(\,\xbar{\bm U}_x)_{j,\ell,m}\Big],
\end{aligned}
\label{3.7}
\end{equation}
where $(\,\xbar{\bm U}_x)_{j,\ell,m}\approx\nu(\xi_\ell,\eta_m)\,\bm U_x(x_j,\xi_\ell,\eta_m,t)$ computed using} the generalized minmod
limiter:
\begin{equation*}
{(\,\xbar{\bm U}_x)_{j,\ell,m}}={\rm minmod}\left(\theta\,\frac{\,\xbar{\bm U}_{j,\ell,m}-\,\xbar{\bm U}_{j-1,\ell,m}}{\dx},\,
\frac{\,\xbar{\bm U}_{j+1,\ell,m}-\,\xbar{\bm U}_{j-1,\ell,m}}{2\dx},\,\theta\,
\frac{\,\xbar{\bm U}_{j+1,\ell,m}-\,\xbar{\bm U}_{j,\ell,m}}{\dx}\right),
\end{equation*}
where the minmod function, defined in \eref{2.10}, is applied in a componentwise manner.

\paragraph{Interpolation in Random Space.} Equipped with $\bm U_{\jph,\ell,m}^\pm$, we proceed with the calculation of the remaining point
values $\bm U_{\jph,\ell\pm\kappa,m}^\pm$, $\bm U_{\jph,\ell,m\pm\kappa}^\pm$, and $\bm U_{\jph,\ell\pm\kappa,m\pm\kappa}^\pm$ using two
consecutive 1-D fifth-order Ai-WENO-Z interpolations. We will describe how to obtain the left-sided values
$\bm U_{\jph,\ell\pm\kappa,m}^-$, $\bm U_{\jph,\ell,m\pm\kappa}^-$, and $\bm U_{\jph,\ell\pm\kappa,m\pm\kappa}^-$ using
$\bm U_{\jph,\ell,m}^-$ and its neighboring values in both the $\xi$- and $\eta$-directions. The right-sided values
$\bm U_{\jph,\ell\pm\kappa,m}^+$, $\bm U_{\jph,\ell,m\pm\kappa}^+$, and $\bm U_{\jph,\ell\pm\kappa,m\pm\kappa}^+$ can be obtained similarly.

We proceed in the following two steps. (Recall that $U$ denotes any component of $\bm U$.)
\begin{asparaitem}
\addtolength{\itemsep}{0.25\baselineskip}
\item[\bf Step 1.] We apply the 1-D fifth-order Ai-WENO-Z interpolations described in \S\ref{sec222} in the $\xi$- and $\eta$-directions to
compute the values $\bm U_{\jph,\ell\pm\kappa,m}^-$ and $\bm U_{\jph,\ell,m\pm\kappa}^-$, respectively.
\item[\bf Step 2.] We compute the values $\bm U_{\jph,\ell-\kappa,m\pm\kappa}^-$ and $\bm U_{\jph,\ell+\kappa,m\pm\kappa}^-$ by applying the
1-D fifth-order Ai-WENO-Z interpolation in the $\eta$-direction to the sets of point values obtained in {\bf Step 1}:
\begin{equation*}
\left\{\bm U_{\jph,\ell-\kappa,m-2}^-,\,\bm U_{\jph,\ell-\kappa,m-1}^-,\,\bm U_{\jph,\ell-\kappa,m}^-,\,\bm U_{\jph,\ell-\kappa,m+1}^-,\,
\bm U_{\jph,\ell-\kappa,m+2}^-\right\}
\end{equation*}
and
\begin{equation*}
\left\{\bm U_{\jph,\ell+\kappa,m-2}^-,\,\bm U_{\jph,\ell+\kappa,m-1}^-,\,\bm U_{\jph,\ell+\kappa,m}^-,\,\bm U_{\jph,\ell+\kappa,m+1}^-,\,
\bm U_{\jph,\ell+\kappa,m+2}^-\right\},
\end{equation*}
respectively.
\end{asparaitem}
\begin{rmk}
Alternatively, in {\bf Step 2}, one can apply the interpolation in the $\xi$-direction to the sets of point values
\begin{equation*}
\left\{\bm U_{\jph,\ell-2,m-\kappa}^-,\,\bm U_{\jph,\ell-1,m-\kappa}^-,\,\bm U_{\jph,\ell,m-\kappa}^-,\,\bm U_{\jph,\ell+1,m-\kappa}^-,\,
\bm U_{\jph,\ell+2,m-\kappa}^-\right\}
\end{equation*}
and
\begin{equation*}
\left\{\bm U_{\jph,\ell-2,m+\kappa}^-,\,\bm U_{\jph,\ell-1,m+\kappa}^-,\,\bm U_{\jph,\ell,m+\kappa}^-,\,\bm U_{\jph,\ell+1,m+\kappa}^-,\,
\bm U_{\jph,\ell+2,m+\kappa}^-\right\},
\end{equation*}
to evaluate $\bm U_{\jph,\ell\pm\kappa,m-\kappa}^-$ and $\bm U_{\jph,\ell\pm\kappa,m+\kappa}^-$, respectively. We note that both
alternatives for {\bf Step 2} are equally accurate.
\end{rmk}
\begin{rmk}
Note that near the $\xi$- and $\eta$-boundaries, one needs to use one-sided Ai-WENO-Z interpolation in the corresponding direction, as
described in \S\ref{sec222}.
\end{rmk}

\subsection{Case $d=2$ and $s=1$}\label{sec32}
In this section, we consider the following version of the system \eref{1.1}:
\begin{equation}
\mU_t+\mF(\mU)_x+\bm G(\mU)_y=\mS(\mU,x,y,\xi),
\label{3.8}
\end{equation}
{
or, equivalently,
\begin{equation}
(\nu\mU)_t+\left[\nu\mF(\mU)\right]_x+\left[\nu\bm G(\mU)\right]_y=\nu\mS(\mU,x,y,\xi),
\label{3.9f}
\end{equation}
}
where $\mU=\mU(x,y,\xi,t)$ with $x$ and $y$ being the spatial variables.

In order to develop a finite-volume method for \eref{3.9f}, we introduce the cells
$C_{j,k,\ell}:=[x_\jmh,x_\jph]\times[y_\kmh,y_\kph]\times[\xi_\lmh,\xi_\lph]$, $j=1,\ldots,N_x$, $k=1,\ldots,N_y$, $\ell=1,\ldots,N_\xi$,
and evolve the {weighted cell averages}
\begin{equation*}
\xbar{\bm U}_{j,k,\ell}\approx\frac{1}{\dx\dy\dxi}\iiint\limits_{C_{j,k,\ell}}\bm U(x,y,\xi,t)\,\nu(\xi)\,{\rm d}x\,{\rm d}y\,{\rm d}\xi
\end{equation*}
in time according to the following semi-discretization of \eref{3.8}:
\begin{equation}
\frac{{\rm d}}{{\rm d}t}\,\xbar{\mU}_{j,k,\ell}=-\frac{\bm{{\cal F}}_{\jph,k,\ell}-\bm{{\cal F}}_{\jmh,k,\ell}}{\dx}
-\frac{\bm{{\cal G}}_{j,\kph,\ell}-\bm{{\cal G}}_{j,\kmh,\ell}}{\dy}+\xbar{\bm S}_{j,k,\ell},
\label{3.9}
\end{equation}
where $\bm{{\cal F}}_{\jpmh,k,\ell}$ and $\bm{{\cal G}}_{j,\kpmh,\ell}$ are the $x$- and $y$-directional numerical fluxes and
$\xbar{\bm S}_{j,k,\ell}$ are approximated cell averages of the source term.

\subsubsection{Numerical Fluxes and Source Terms}
Similarly to the $d=s=1$ case, the numerical fluxes in \eref{3.9} approximate the average of the corresponding fluxes in the
$\xi$-direction:
\begin{equation}
\begin{aligned}
&\bm{{\cal F}}_{\jph,k,\ell}=\sum_{i=1}^M\mu_i\,\nu(\xi_{\ell_i})\,\bm{{\cal F}}\big(\bm U^-_{\jph,k,\ell_i},\bm U^+_{\jph,k,\ell_i}\big),\\
&\bm{{\cal G}}_{j,\kph,\ell}=\sum_{i=1}^M\mu_i\,\nu(\xi_{\ell_i})\,\bm{{\cal G}}\big(\bm U^-_{j,\kph,\ell_i},\bm U^+_{j,\kph,\ell_i}\big).
\end{aligned}
\label{3.10}
\end{equation}
Here, $\bm{{\cal F}}$ and $\bm{{\cal G}}$ are the CU numerical fluxes from \cite{KTrp}:
\begin{equation}
\begin{aligned}
\bm{{\cal F}}\big(\bm U^-_{\jph,k,\ell_i},\bm U^+_{\jph,k,\ell_i}\big)&=\frac{a_{\jph,k,\ell}^
+\bm F\big(\bm U^-_{\jph,k,\ell_i}\big)-a_{\jph,k,\ell}^-\bm F\big(\bm U^+_{\jph,k,\ell_i}\big)}{a_{\jph,k,\ell}^+-a_{\jph,k,\ell}^-}\\
&+\frac{a_{\jph,k,\ell}^+a_{\jph,k,\ell}^-}{a_{\jph,k,\ell}^+-a_{\jph,k,\ell}^-}
\left(\bm U^+_{\jph,k,\ell_i}-\bm U^-_{\jph,k,\ell_i}\right),\\[0.5ex]
\bm{{\cal G}}\big(\bm U^-_{j,\kph,\ell_i},\bm U^+_{j,\kph,\ell_i}\big)&=\frac{a_{j,\kph,\ell}^
+\bm G\big(\bm U^-_{j,\kph,\ell_i}\big)-a_{j,\kph,\ell}^-\bm G\big(\bm U^+_{j,\kph,\ell_i}\big)}{a_{j,\kph,\ell}^+-a_{j,\kph,\ell}^-}\\
&+\frac{a_{j,\kph,\ell}^+a_{j,\kph,\ell}^-}{a_{j,\kph,\ell}^+-a_{j,\kph,\ell}^-}
\left(\bm U^+_{j,\kph,\ell_i}-\bm U^-_{j,\kph,\ell_i}\right),
\end{aligned}
\label{3.11}
\end{equation}
where $\bm U^\pm_{\jph,k,\ell_i}\approx\bm U(x_\jph\pm0,y_k,\xi_{\ell_i})$ and $\bm U^\pm_{j,\kph,\ell_i}\approx
\bm U(x_j,y_\kph\pm0,\xi_{\ell_i})$, and $a_{\jph,k,\ell}^\pm$ and $a_{j,\kph,\ell}^\pm$ are the $x$- and $y$-directional one-sided local
speeds of propagation at $(x_\jph,y_k,\xi_\ell)$ and $(x_j,k_\kph,\xi_\ell)$, respectively. They can be estimated using the smallest and
largest eigenvalues of the Jacobians $\frac{\partial\bm F}{\partial\bm U}$ and $\frac{\partial\bm G}{\partial\bm U}$:
\begin{equation}
\begin{aligned}
&a_{\jph,k,\ell}^-=\min\left\{\lambda_1\left(\frac{\partial\bm F}{\partial\bm U}\left(\bm U^-_{\jph,k,\ell}\right)\right),\,
\lambda_1\left(\frac{\partial\bm F}{\partial\bm U}\left(\bm U^+_{\jph,k,\ell}\right)\right),\,0\right\},\\
&a_{\jph,k,\ell}^+=\max\left\{\lambda_N\left(\frac{\partial\bm F}{\partial\bm U}\left(\bm U^-_{\jph,k,\ell}\right)\right),\,
\lambda_N\left(\frac{\partial\bm F}{\partial\bm U}\left(\bm U^+_{\jph,k,\ell}\right)\right),\,0\right\},\\
&a_{j,\kph,\ell}^-=\min\left\{\lambda_1\left(\frac{\partial\bm G}{\partial\bm U}\left(\bm U^-_{j,\kph,\ell}\right)\right),\,
\lambda_1\left(\frac{\partial\bm G}{\partial\bm U}\left(\bm U^+_{j,\kph,\ell}\right)\right),\,0\right\},\\
&a_{j,\kph,\ell}^+=\max\left\{\lambda_N\left(\frac{\partial\bm G}{\partial\bm U}\left(\bm U^-_{j,\kph,\ell}\right)\right),\,
\lambda_N\left(\frac{\partial\bm G}{\partial\bm U}\left(\bm U^+_{j,\kph,\ell}\right)\right),\,0\right\}.
\end{aligned}
\label{3.12}
\end{equation}
Finally, the source term on the RHS of \eref{3.9} is discretized as in \eref{2.7}:
\begin{equation}
\xbar{\bm S}_{j,k,\ell}=\frac{1}{\dx\dy}\sum_{i=1}^M\mu_i\,\nu(\xi_{\ell_i})\int\limits_{y_\kmh}^{y_\kph}\int\limits_{x_\jmh}^{x_\jph}
\bm S\big(\bm U(x,y,\xi_{\ell_i},t),x,y,\xi_{\ell_i}\big)\,{\rm d}x\,{\rm d}y.
\label{3.13}
\end{equation}
A particular quadrature for the integrals on the RHS of \eref{3.13} in the case of the Saint-Venant system will be specified in
\S\ref{sec6}.

\subsubsection{Reconstruction of the Point Values}
The point values $\bm U^\pm_{\jph,k,\ell_i}$ and $\bm U^\pm_{j,\kph,\ell_i}$ used in \eref{3.10}--\eref{3.12} are computed as follows.

\paragraph{Reconstruction in Physical Space.} We first use the generalized minmod limiter to evaluate
{
\begin{equation*}
\begin{aligned}
\bm U_{\jmh,k,\ell}^+=\frac{1}{\nu(\xi_\ell)}\Big[\,\xbar{\bm U}_{j,k,\ell}-\frac{\dx}{2}(\,\xbar{\bm U}_x)_{j,k,\ell}\Big],\quad
\bm U_{\jph,k,\ell}^-=\frac{1}{\nu(\xi_\ell)}\Big[\,\xbar{\bm U}_{j,k,\ell}+\frac{\dx}{2}(\,\xbar{\bm U}_x)_{j,k,\ell}\Big],\\
\bm U_{j,\kmh,\ell}^+=\frac{1}{\nu(\xi_\ell)}\Big[\,\xbar{\bm U}_{j,k,\ell}-\frac{\dy}{2}(\,\xbar{\bm U}_y)_{j,k,\ell}\Big],\quad
\bm U_{j,\kph,\ell}^-=\frac{1}{\nu(\xi_\ell)}\Big[\,\xbar{\bm U}_{j,k,\ell}+\frac{\dy}{2}(\,\xbar{\bm U}_y)_{j,k,\ell}\Big],
\end{aligned}
\end{equation*}
where $(\,\xbar{\bm U}_x)_{j,k,\ell}\approx\nu(\xi_\ell)\,\bm U_x(x_j,y_k,\xi_\ell,t)$ and
$(\,\xbar{\bm U}_y)_{j,k,\ell}\approx\nu(\xi_\ell)\,\bm U_y(x_j,y_k,\xi_\ell,t)$ computed using} the generalized minmod limiter:
\begin{equation*}
\begin{aligned}
&{(\,\xbar{\bm U}_x)_{j,k,\ell}}={\rm minmod}\left(\theta\,\frac{\,\xbar{\bm U}_{j,k,\ell}-\,\xbar{\bm U}_{j-1,k,\ell}}{\dx},\,
\frac{\,\xbar{\bm U}_{j+1,k,\ell}-\,\xbar{\bm U}_{j-1,k,\ell}}{2\dx},\, \theta\,\frac{\,\xbar{\bm U}_{j+1,k,\ell}-\,
\xbar{\bm U}_{j,k,\ell}}{\dx}\right),\\
&{(\,\xbar{\bm U}_y)_{j,k,\ell}}={\rm minmod}\left(\theta\,\frac{\,\xbar{\bm U}_{j,k,\ell}-\,\xbar{\bm U}_{j,k-1,\ell}}{\dy},\,
\frac{\,\xbar{\bm U}_{j,k+1,\ell}-\,\xbar{\bm U}_{j,k-1,\ell}}{2\dy},\, \theta\,\frac{\,\xbar{\bm U}_{j,k+1,\ell}-\,
\xbar{\bm U}_{j,k,\ell}}{\dy}\right),
\end{aligned}
\end{equation*}
where the minmod function, defined in \eref{2.10}, is applied in a componentwise manner.

\paragraph{Interpolation in Random Space.} Second, equipped with $\big\{\bm U_{\jph,k,\ell}^\pm\big\}$ and
$\big\{\bm U_{j,\kph,\ell}^\pm\big\}$, the solution values at the Gauss-Legendre points are obtained $(x_\jph,y_k,\xi_{\ell\pm\kappa})$ and
$(x_j,y_\kph,\xi_{\ell\pm\kappa})$ using the fifth-order Ai-WENO-Z interpolant in $\xi$ precisely as described in \S\ref{sec222}.

\section{Time Evolution}\label{sec4}
The systems of ordinary differential equations (ODEs) \eref{2.2}, \eref{3.2}, and \eref{3.9} have to be numerically integrated using an
appropriate ODE solver. In this paper, we use the three-stage third-order strong stability preserving (SSP) Runge-Kutta method; see, e.g.,
\cite{GKS,GST}. The method has been implemented with the CFL number $K=0.45$.

\subsection{Case $d=1$ and $s=1$}\label{sec41}
A fully discrete version of the CU scheme will be, in general, stable if the following CFL condition is satisfied:
\begin{equation}
\dt=\frac{K\dx}{\max\limits_{j,\ell}\left(\max\big\{a_{\jph,\ell}^+,-a_{\jph,\ell}^-\big\}\right)},\quad K\le\hf.
\label{4.1}
\end{equation}
The CFL condition \eref{4.1}, however, does not guarantee that components of the computed $\bm U$, which are supposed to be non-negative,
remain non-negative. As before, we denote one of those components of $\bm U$ by $U$, assume that the corresponding component of the source
term $S\equiv0$, and adopt the ``draining'' timestep technique originally introduced in \cite{BNL}.

We begin with a forward Euler discretization of this component of \eref{2.2},
\begin{equation}
\xbar U_{j,\ell}(t+\dt)=\xbar U_{j,\ell}(t)-\frac{\dt}{\dx}\left({\cal F}_{\jph,\ell}(t)-{\cal F}_{\jmh,\ell}(t)\right),
\label{4.2}
\end{equation}
and modify the numerical fluxes on the RHS of \eref{4.2} to ensure $\,\xbar U_{j,\ell}(t+\dt)\ge0$ as long as $\,\xbar U_{j,\ell}(t)\ge0$
for all $j,\ell$. To this end, we compute
\begin{equation*}
\dt^{\rm drain}_{j,\ell}=\frac{\dx\,\xbar U_{j,\ell}(t)}{\max\big\{0,{\cal F}_{\jph,\ell}(t)\big\}+\max\big\{0,-{\cal F}_{\jmh,\ell}(t)
\big\}},
\end{equation*}
introduce the following quantities:
\begin{equation*}
\dt_{\jph,\ell}=\min\big\{\dt,\dt_{j_0,\ell}^{\rm drain}\big\},\quad
{j_0=\left\{\begin{aligned}&j,&&{\cal F}_{\jph,\ell}(t)>0,\\&j+1,&&\mbox{otherwise},\end{aligned}\right.}
\end{equation*}
and replace \eref{4.2} with
\begin{equation}
\xbar U_{j,\ell}(t+\dt)=\,\xbar U_{j,\ell}(t)-\frac{\dt_{\jph,\ell}\,{\cal F}_{\jph,\ell}(t)-\dt_{\jmh,\ell}\,{\cal F}_{\jmh,\ell}(t)}{\dx}.
\label{4.3}
\end{equation}
Notice that the last equation can also be obtained by replacing the numerical fluxes in \eref{4.2} with
$\frac{\dt_{\jph,\ell}}{\dt}\,{\cal F}_{\jph,\ell}(t)$. One can show that using \eref{4.3} ensures the non-negativity of
$\,\xbar U_{j,\ell}(t+\dt)$. An extension to the three-stage third-order SSP Runge-Kutta method is straightforward: the numerical flux
modification is carried out at every stage.

\subsection{Case $d=1$ and $s=2$}\label{sec42}
In this case, a similar time evolution as in \S\ref{sec41} can be used, with $\ell$ now being the multi-index $(\ell,m)$.

\subsection{Case $d=2$ and $s=1$}\label{sec43}
A fully discrete version of the CU scheme will be, in general, stable if the following CFL condition is satisfied:
\begin{equation*}
\dt=K\min\Bigg\{\frac{\dx}{\max\limits_{j,k,\ell}\left(\max\big\{a_{\jph,k,\ell}^+,-a_{\jph,k,\ell}^-\big\}\right)},
\frac{\dy}{\max\limits_{j,k,\ell}\left(\max\big\{a_{j,\kph,\ell}^+,-a_{j,\kph,\ell}^-\big\}\right)}\Bigg\},~~K\le\hf.
\end{equation*}
As in the $d=s=1$ case, we first introduce the ``draining'' timestep, which is now
\begin{equation*}
\dt^{\rm drain}_{j,k,\ell}=\frac{\dx\,\dy\,\xbar U_{j,k,\ell}}{f^{\rm drain}_{j,k,\ell}\dy+g^{\rm drain}_{j,k,\ell}\dx},
\end{equation*}
where
\begin{equation*}
\begin{aligned}
&f^{\rm drain}_{j,k,\ell}=\max\big\{0,{\cal F}_{\jph,k,\ell}\big\}+\max\big\{0,-{\cal F}_{\jmh,k,\ell}\big\},\\
&g^{\rm drain}_{j,k,\ell}=\max\big\{0,{\cal G}_{j,\kph,\ell}\big\}+\max\big\{0,-{\cal G}_{j,\kmh,\ell}\big\}.
\end{aligned}
\end{equation*}
We then introduce the following quantities:
\begin{equation*}
\begin{aligned}
&\dt_{\jph,k,\ell}=\min\big\{\dt,\dt_{j_0,k,\ell}^{\rm drain}\big\},\quad
{j_0=\left\{
\begin{aligned}
&j,&&{\cal F}_{\jph,k,\ell}(t)>0,\\&j+1,&&\mbox{otherwise},
\end{aligned}
\right.}\\
&\dt_{j,\kph,\ell}=\min\big\{\dt,\dt_{j,k_0,\ell}^{\rm drain}\big\},\quad
{k_0=\left\{
\begin{aligned}
&k,&&{\cal G}_{j,\kph,\ell}(t)>0,\\&k+1,&&\mbox{otherwise},
\end{aligned}
\right.}
\end{aligned}
\end{equation*}
and use them to replace the corresponding component of the numerical fluxes in \eref{3.9} with
$\frac{\dt_{\jph,k,\ell}}{\dt}\,{\cal F}_{\jph,k,\ell}$ and $\frac{\dt_{j,\kph,\ell}}{\dt}\,{\cal G}_{j,\kph,\ell}$.
\begin{rmk}
We note that in this paper, the ``draining'' timestep technique has been applied to the water depth component of the Saint-Venant system in
the numerical examples reported in \S\ref{sec64}. This helps to prevent the appearance of discrete negative water
depth values.
\end{rmk}

\section{Euler Equations of Gas Dynamics}\label{sec5}
In this section, we illustrate the performance of the proposed numerical method on three examples for the Euler equations of gas dynamics.
In all of the numerical examples below, we take the minmod parameter $\theta=1.3$.

\paragraph{Example 1 ($d=s=1$).} We start with the 1-D case, in which the Euler equations of gas dynamics read as \eref{2.1} with
\begin{equation}
\bm U=(\rho,\rho u,E)^\top,\quad\bm F(\bm U)=(\rho u,\rho u^2+p,u(E+p))^\top,\quad\bm S\equiv\bm0.
\label{5.1}
\end{equation}
Here, $\rho(x,t,\xi)$ is the density, $u(x,t,\xi)$ is the velocity, $E(x,t,\xi)$ is the total energy, and $p(x,t,\xi)$ is the pressure,
which is related to the conservative quantities through the equation of state (EOS) for the ideal gas:
\begin{equation}
p=(\gamma-1)\left[E-\frac{\rho u^2}{2}\right],
\label{5.2}
\end{equation}
where $\gamma$ is the specific heat ratio.

We consider the Sod shock tube problem with $\gamma=1.4$, the Riemann ICs,
\begin{equation}
(\rho(x,0),u(x,0),p(x,0))=\left\{\begin{aligned}
&(1,0,1),&&x<0.5,\\
&(0.125,0,0.1),&&x>0.5,
\end{aligned}\right.
\label{5.3}
\end{equation}
and free BCs, prescribed in the spatial domain $x\in[0,1]$.

\smallskip
\noindent {\bf Test 1.} We first perturb the initial density and replace $\rho(x,0)$ in \eref{5.3} with
\begin{equation*}
\rho(x,0,\xi)=\left\{\begin{aligned}
&1+\sigma\xi,&&x<0.5,\\
&0.125,&&x>0.5.
\end{aligned}\right.
\end{equation*}
We compute the solution until time $t=0.1644$ on a uniform mesh with $\dx=1/200$ and $\dxi=1/50$. The mean, 95\%-quantile, and standard
deviation of $\rho$ obtained for $\sigma=0.1$, 0.2, 0.3, and 0.4 {and uniformly distributed $\xi\sim{\cal U}(-1,1)$, that is,
$\nu(\xi)\equiv\hf$}, are presented in \fref{fig51}. As one can see, a larger initial perturbation of $\rho$ yields a larger variance and
$95\%$-quantile.
\begin{figure}[ht!]
\centerline{\includegraphics[width=0.32\textwidth]{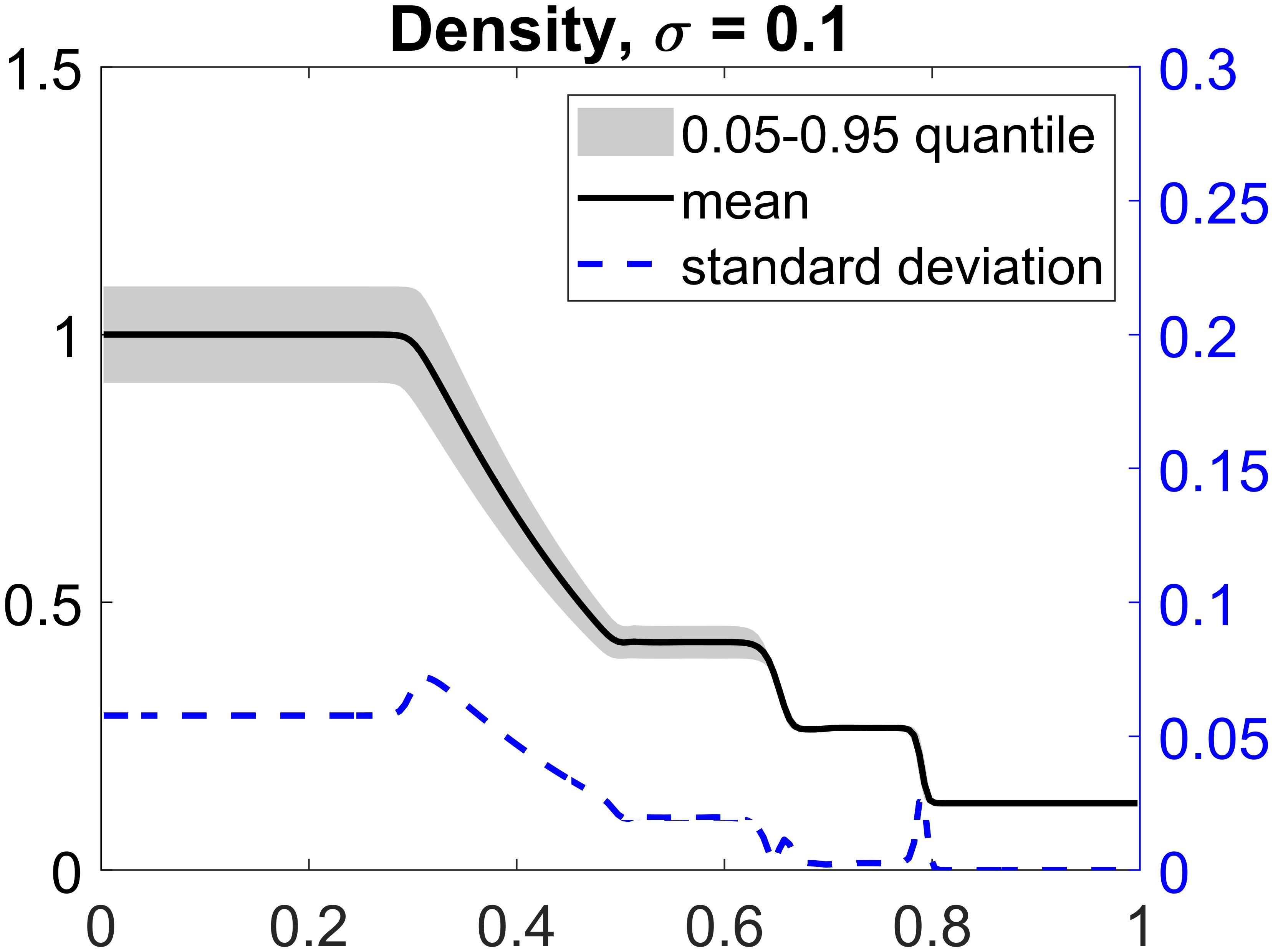}\hspace*{0.5cm}\includegraphics[width=0.32\textwidth]{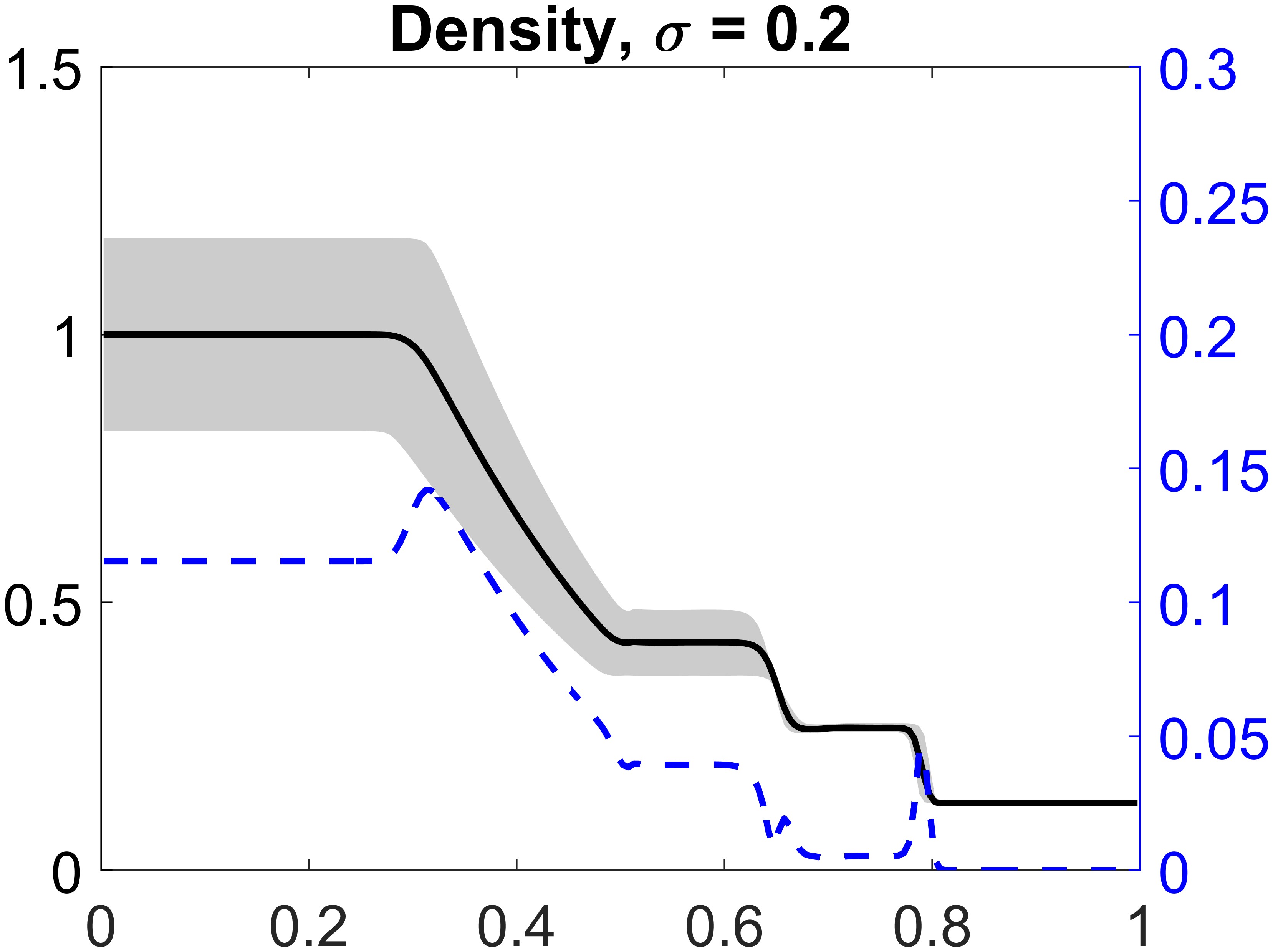}}
\vskip8pt
\centerline{\includegraphics[width=0.32\textwidth]{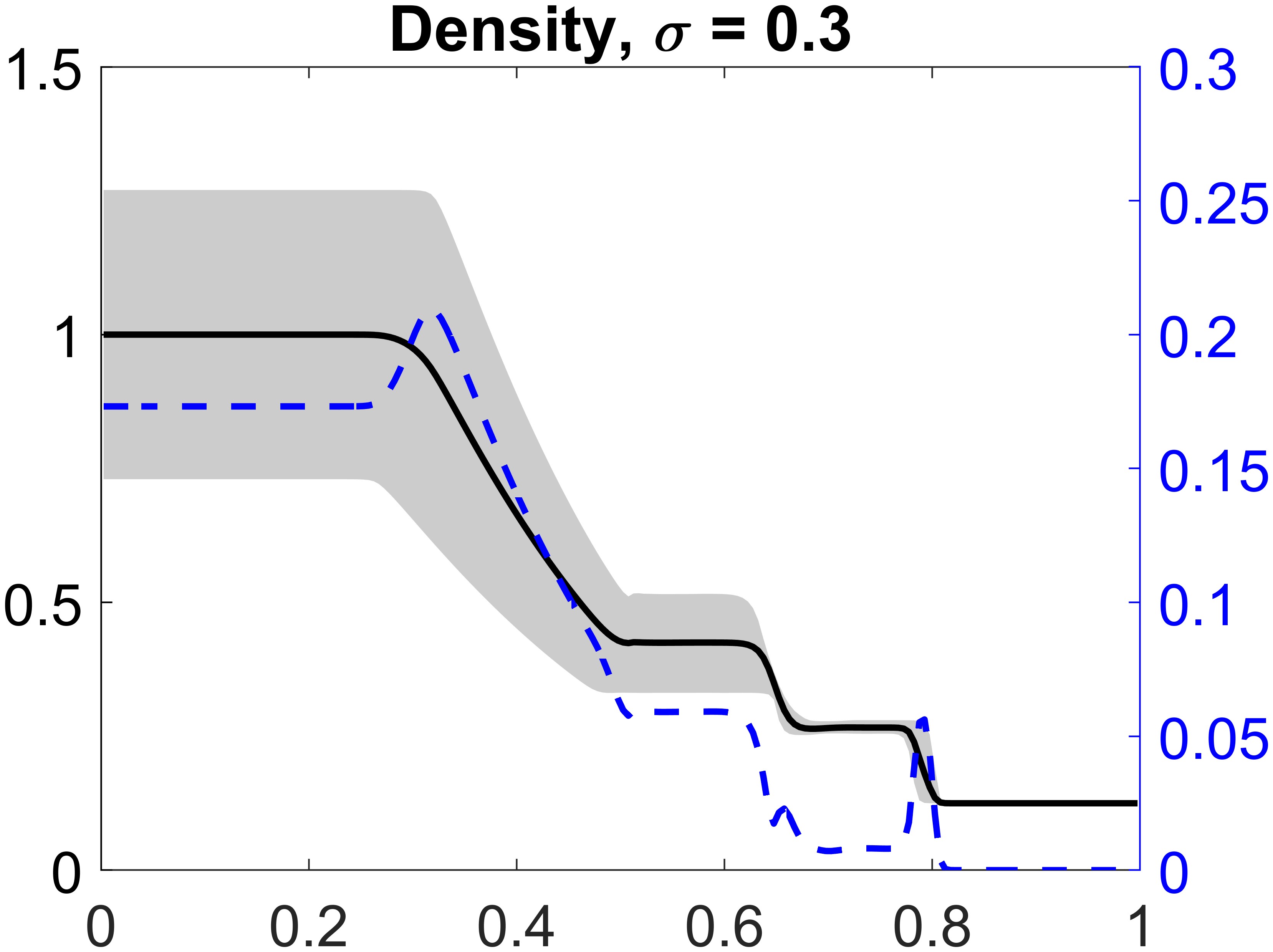}\hspace*{0.5cm}\includegraphics[width=0.32\textwidth]{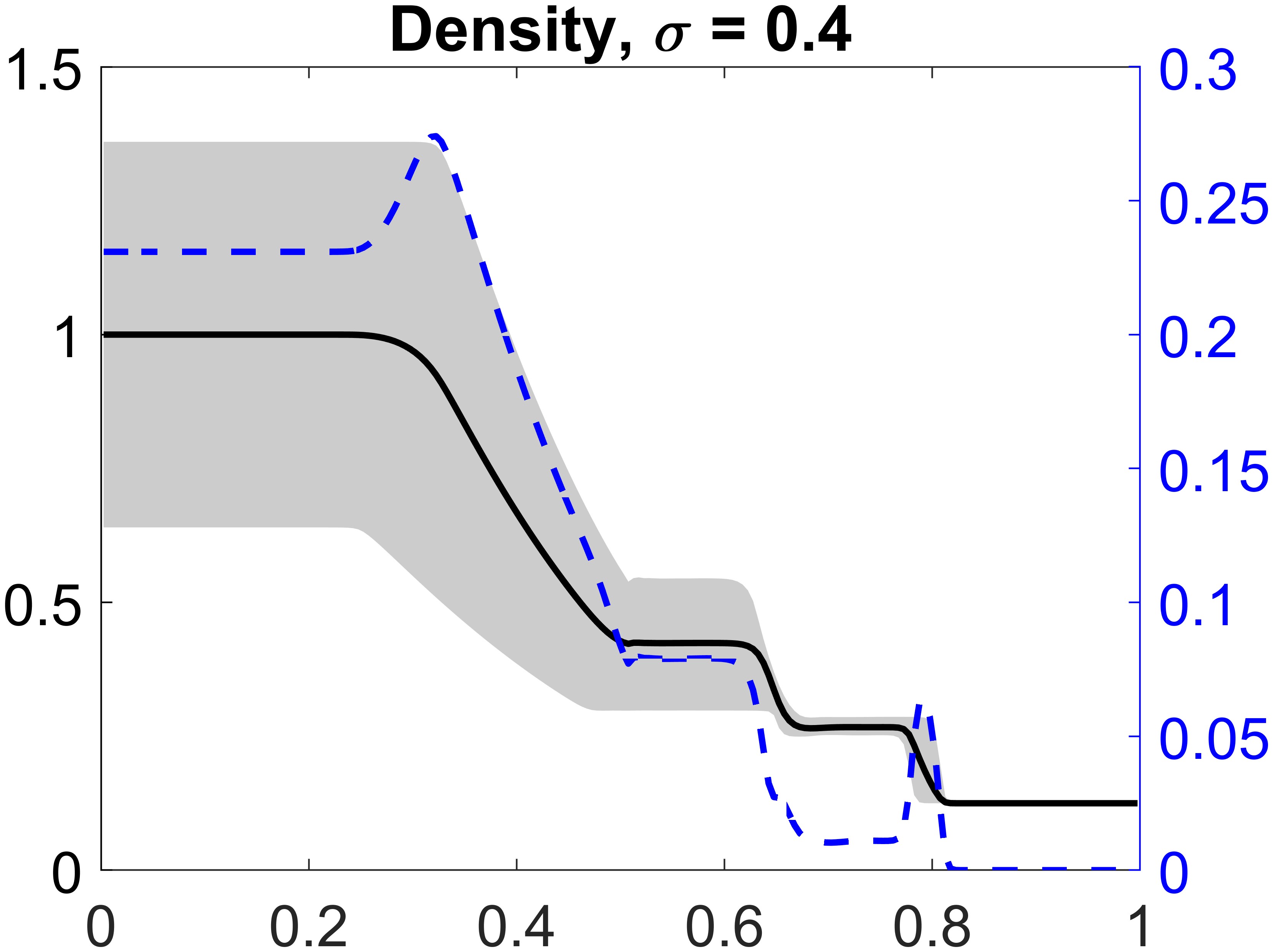}}
\caption{\sf Example 1, Test 1: Mean, 95\%-quantile, and standard deviation of $\rho$ for {$\xi\sim{\cal U}(-1,1)$ and} different
$\sigma$.\label{fig51}}
\end{figure}

{
We also perform analogous simulations assuming that the random variable $\xi$ is normally distributed, $\xi\sim{\cal N}(0,1/36)$, that is,
$\nu(\xi)=\sqrt{\frac{18}{\pi}}{\rm e}^{-18\xi^2}$, and plot
the results in \fref{fig51a}. As one can see from Figures \ref{fig51} and \ref{fig51a}, the mean values obtained with the uniform and normal
distributions are similar, while the standard deviation for the normally distributed random variable has smaller values since $\xi$ is more
concentrated around 0.
}
\begin{figure}[ht!]
\centerline{\includegraphics[width=0.32\textwidth]{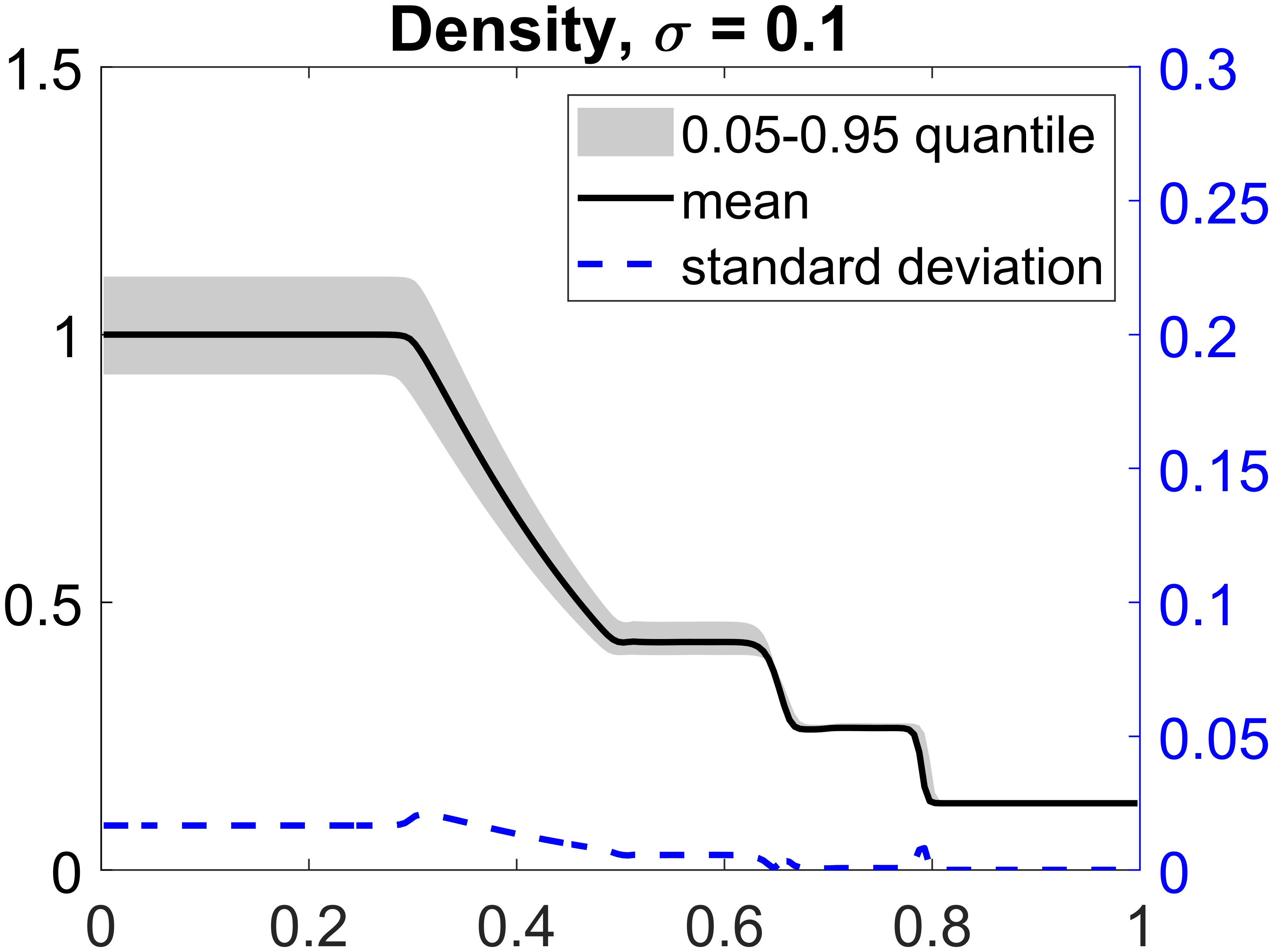}\hspace*{0.5cm}
\includegraphics[width=0.32\textwidth]{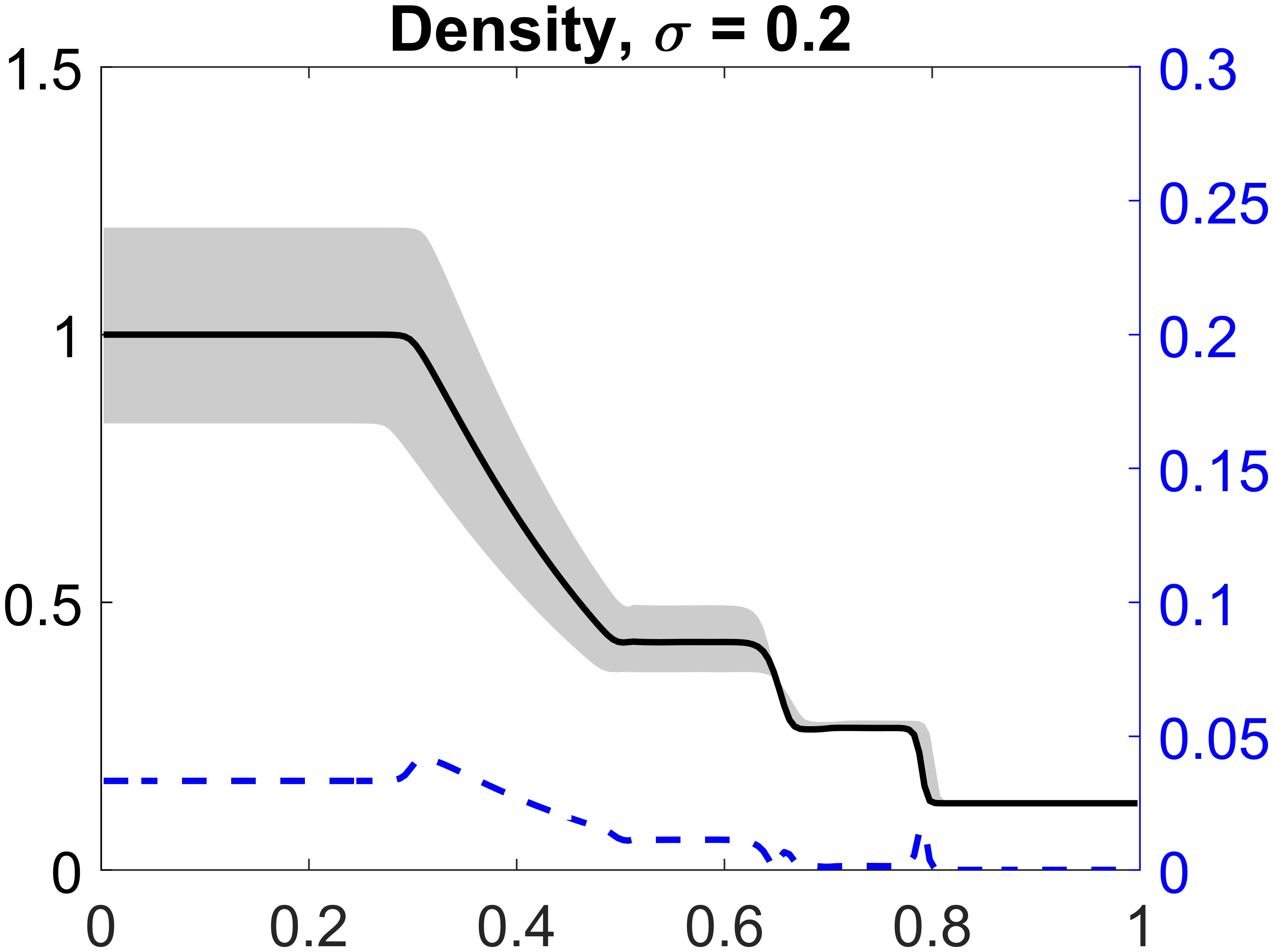}}
\vskip8pt
\centerline{\includegraphics[width=0.32\textwidth]{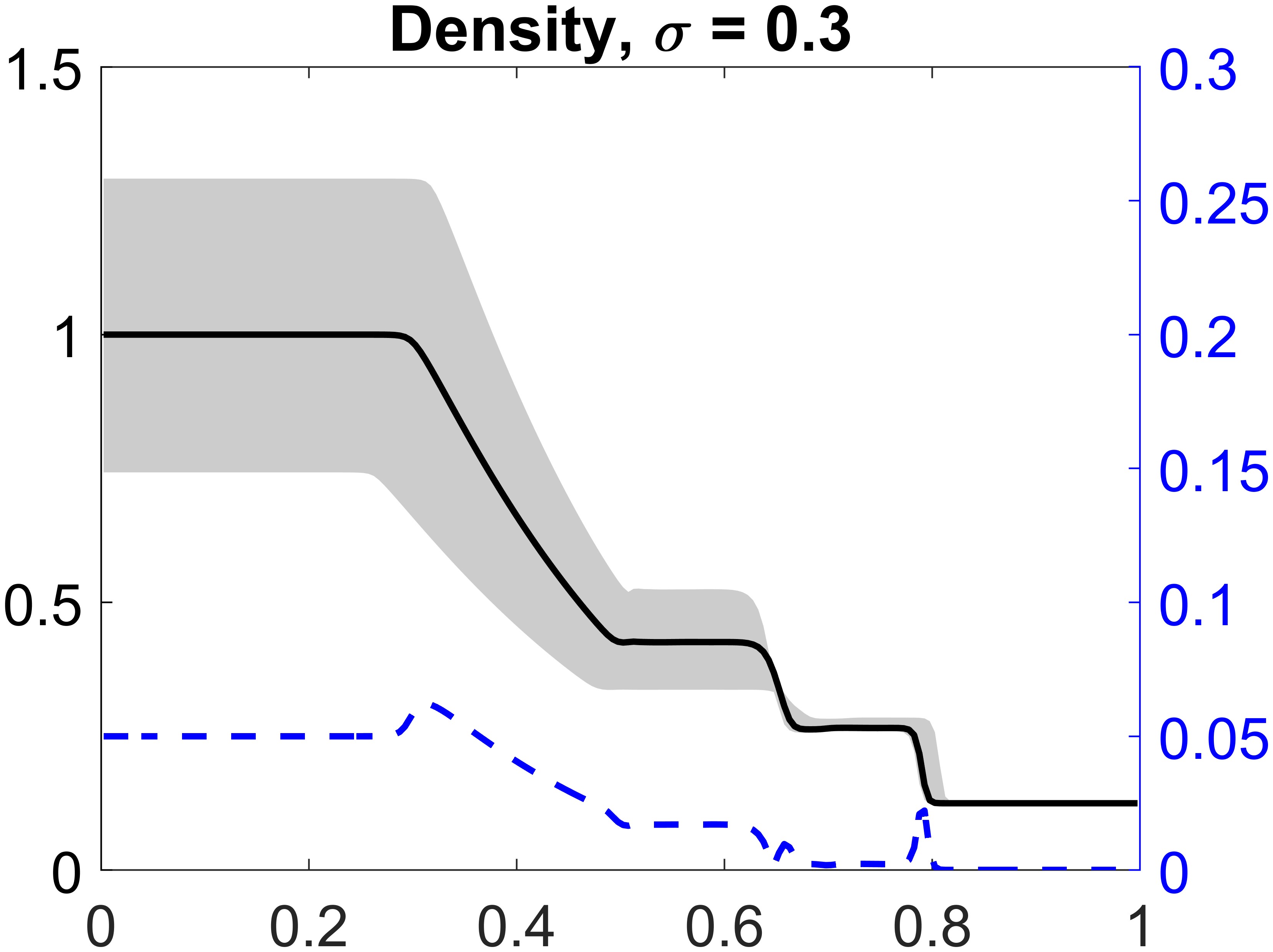}\hspace*{0.5cm}
\includegraphics[width=0.32\textwidth]{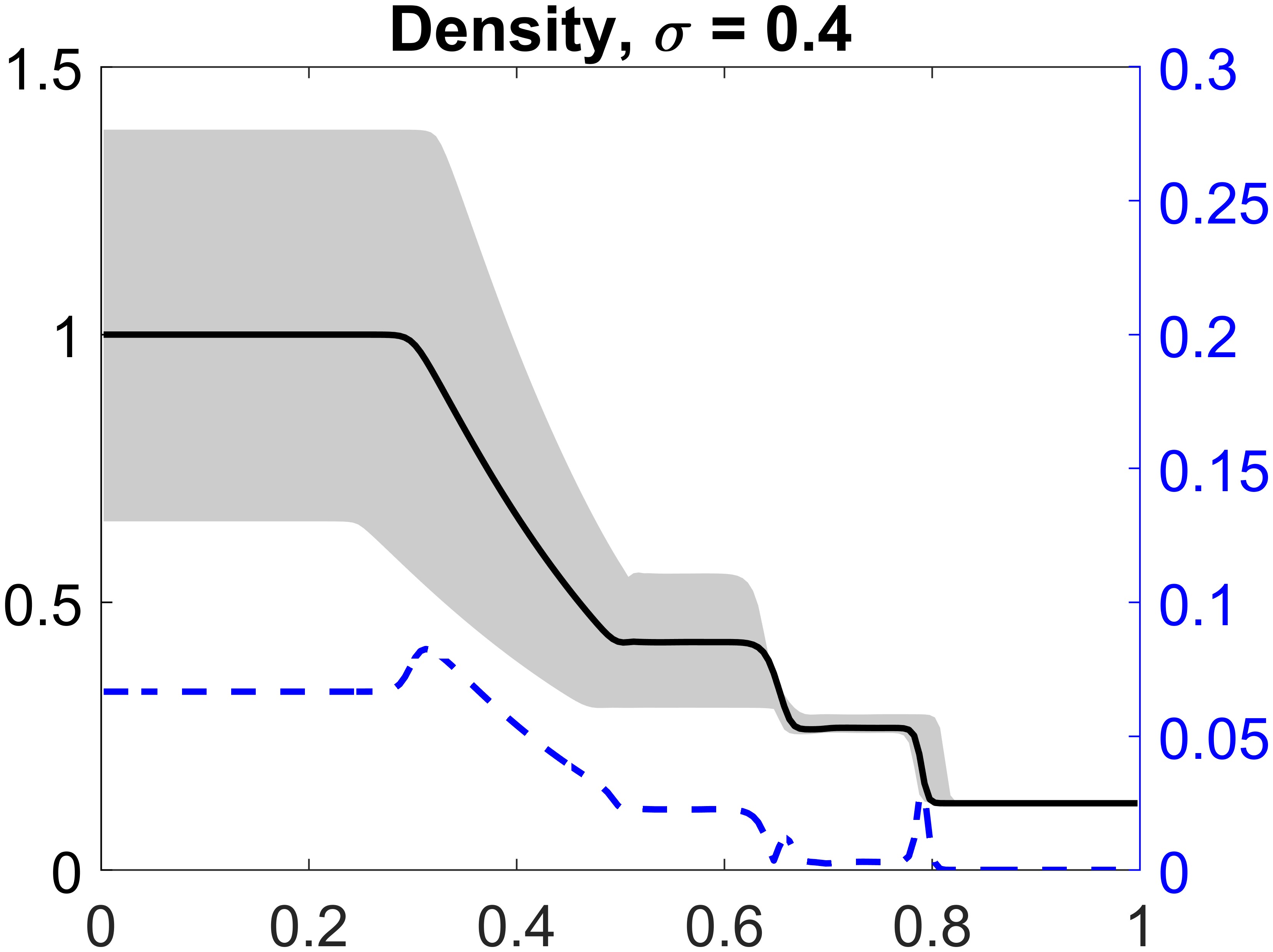}}
\caption{\sf Example 1, Test 1: {Same as in \fref{fig51}, but for $\xi\sim{\cal N}(0,1/36)$}.\label{fig51a}}
\end{figure}

It is also instructive to observe that the $95\%$-quantile in the momentum is distributed over a slightly different domain, while the
$95\%$-quantile is almost invisible in the pressure and total energy; {see Figures \ref{fig52} and \ref{fig52a}}, where the
corresponding results for $\sigma=0.1$ are plotted.
\begin{figure}[ht!]
\centerline{\hspace*{-0.1cm}\includegraphics[width=0.31\textwidth]{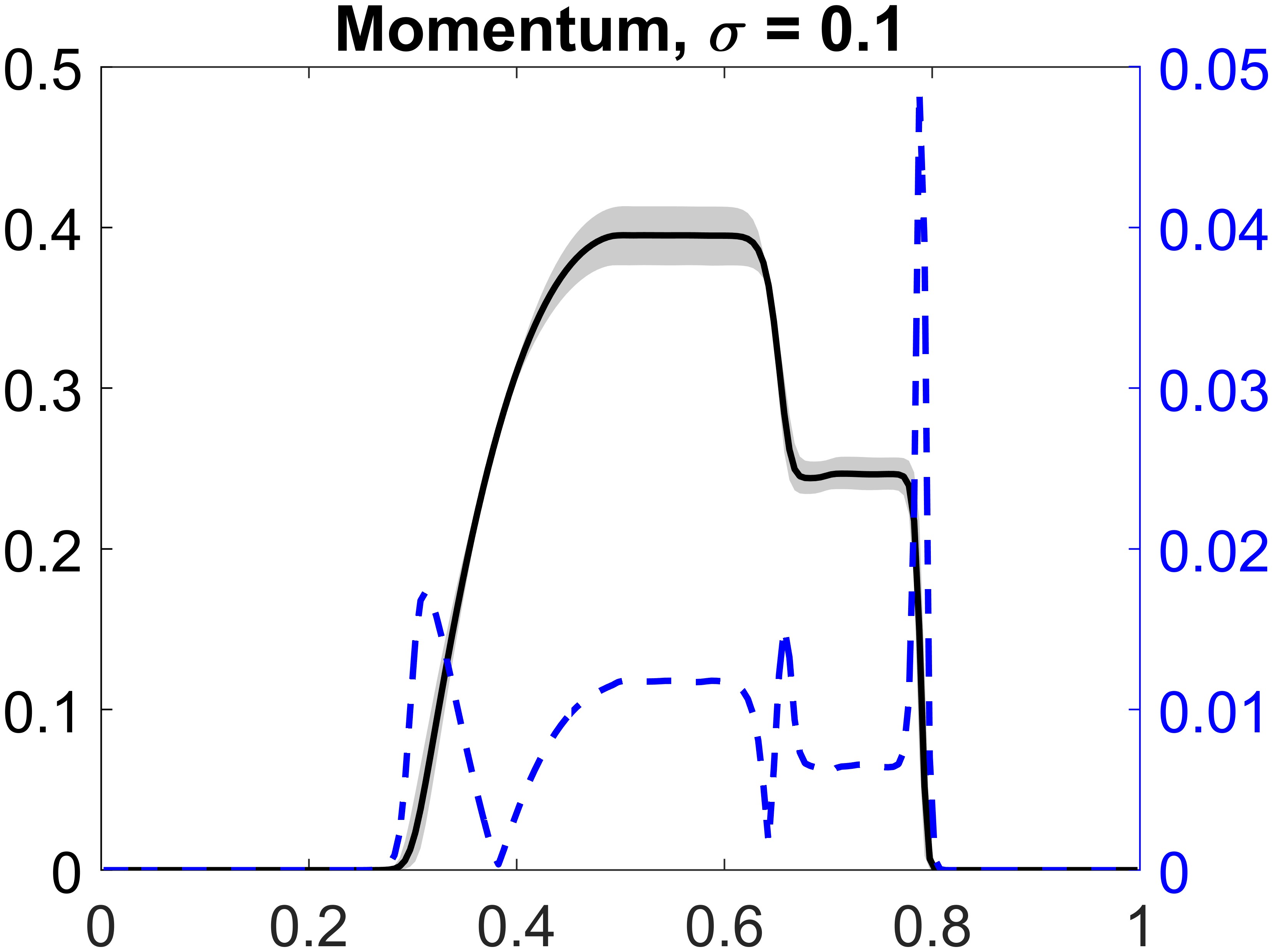}\hspace*{0.5cm}
\includegraphics[width=0.31\textwidth]{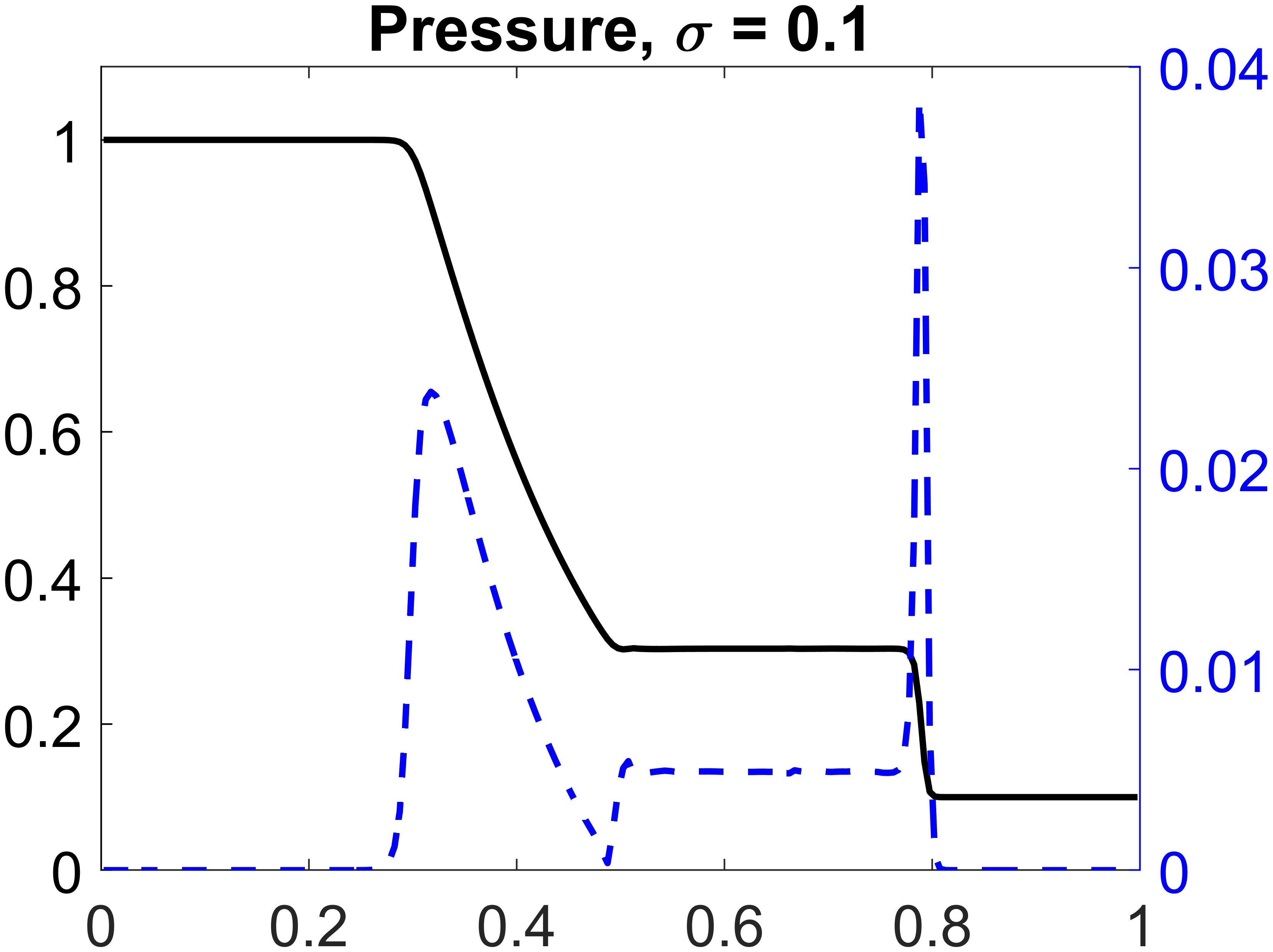}\hspace*{0.5cm}\includegraphics[width=0.31\textwidth]{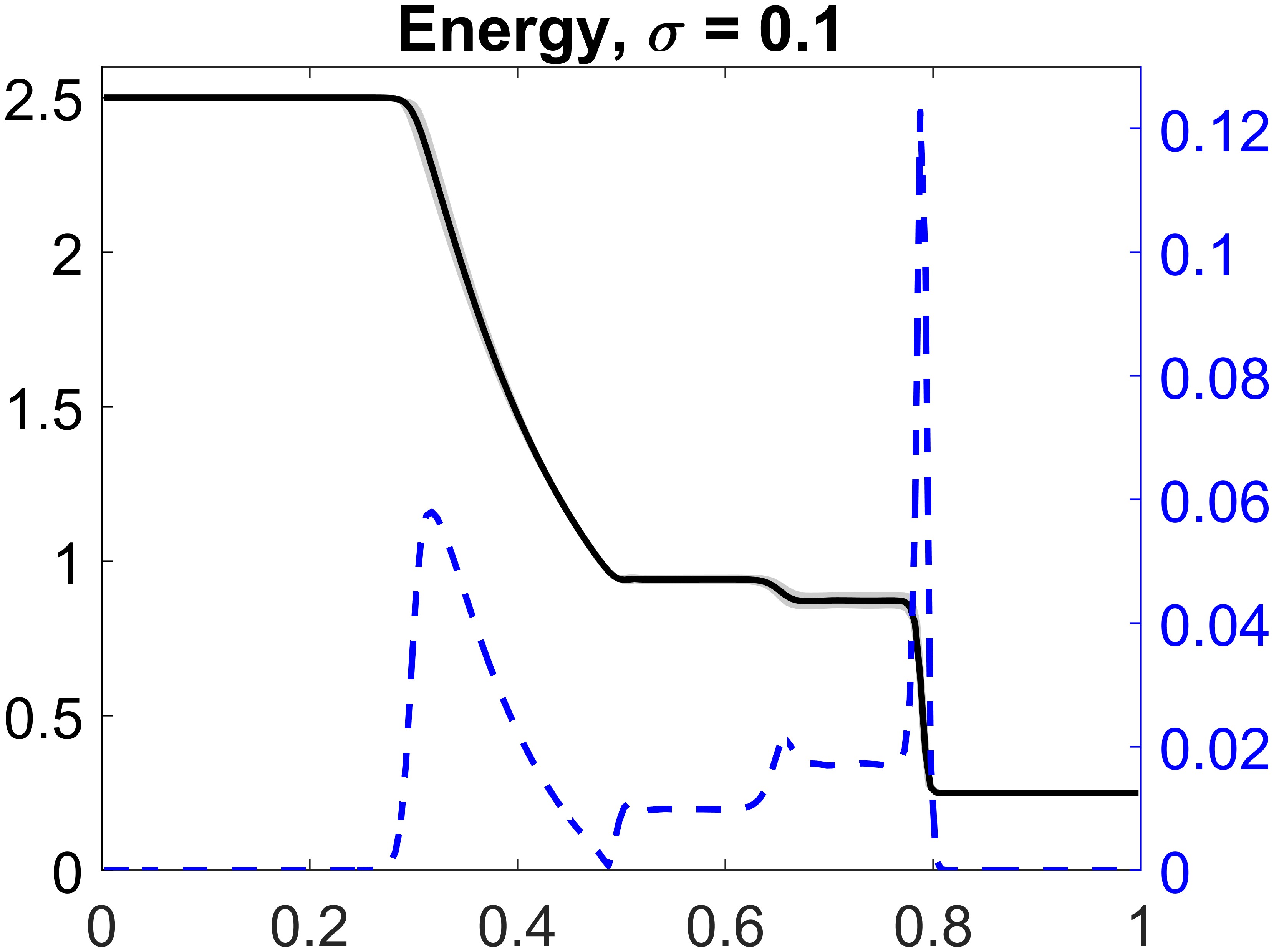}}
\caption{\sf Example 1, Test 1: Mean, 95\%-quantile, and standard deviation of $\rho u$, $p$, and $E$ for $\sigma=0.1$ {and
$\xi\sim{\cal U}(-1,1)$.}\label{fig52}}
\end{figure}
\begin{figure}[ht!]
\centerline{\hspace*{-0.1cm}\includegraphics[width=0.31\textwidth]{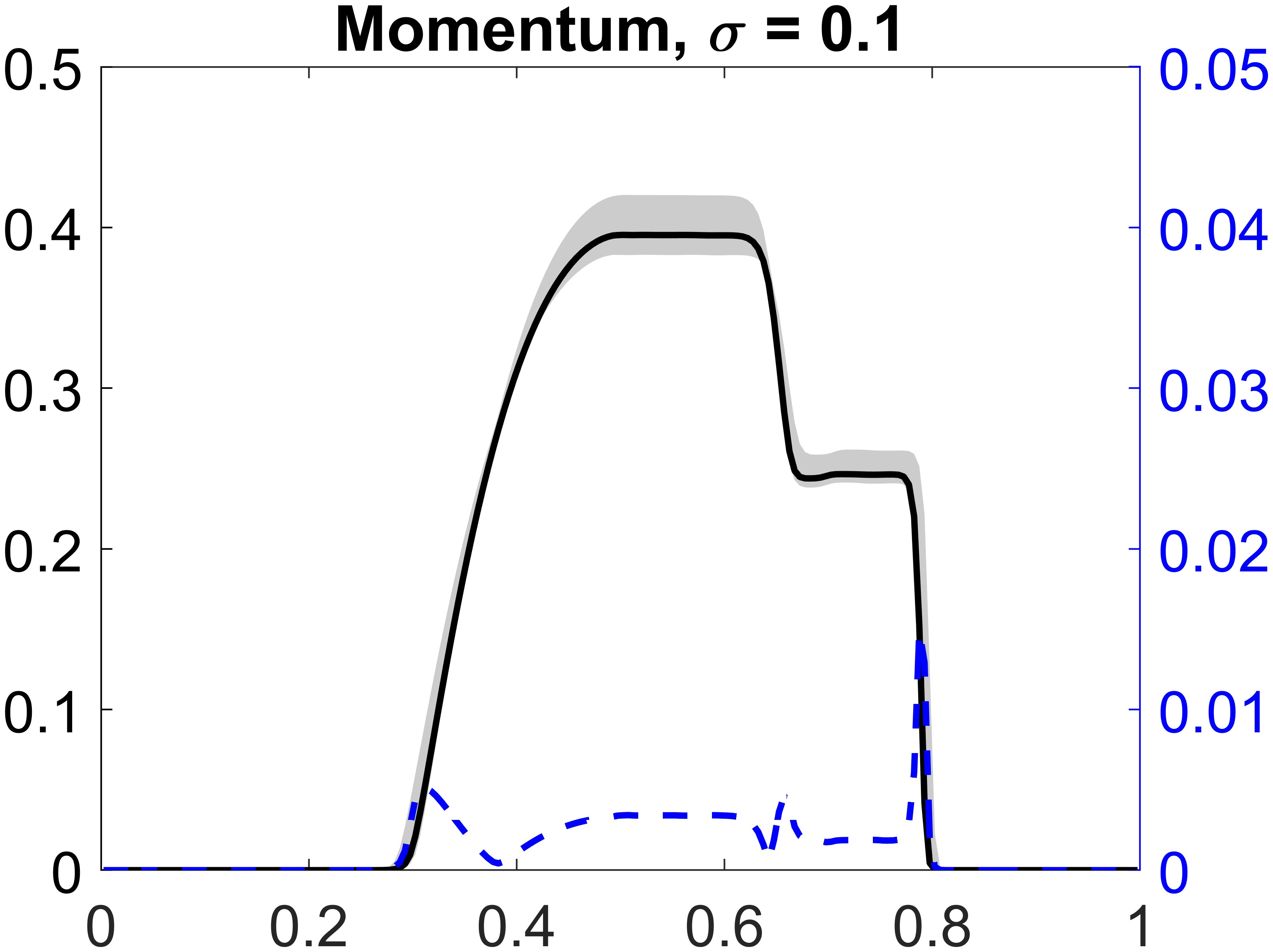}\hspace*{0.5cm}
\includegraphics[width=0.31\textwidth]{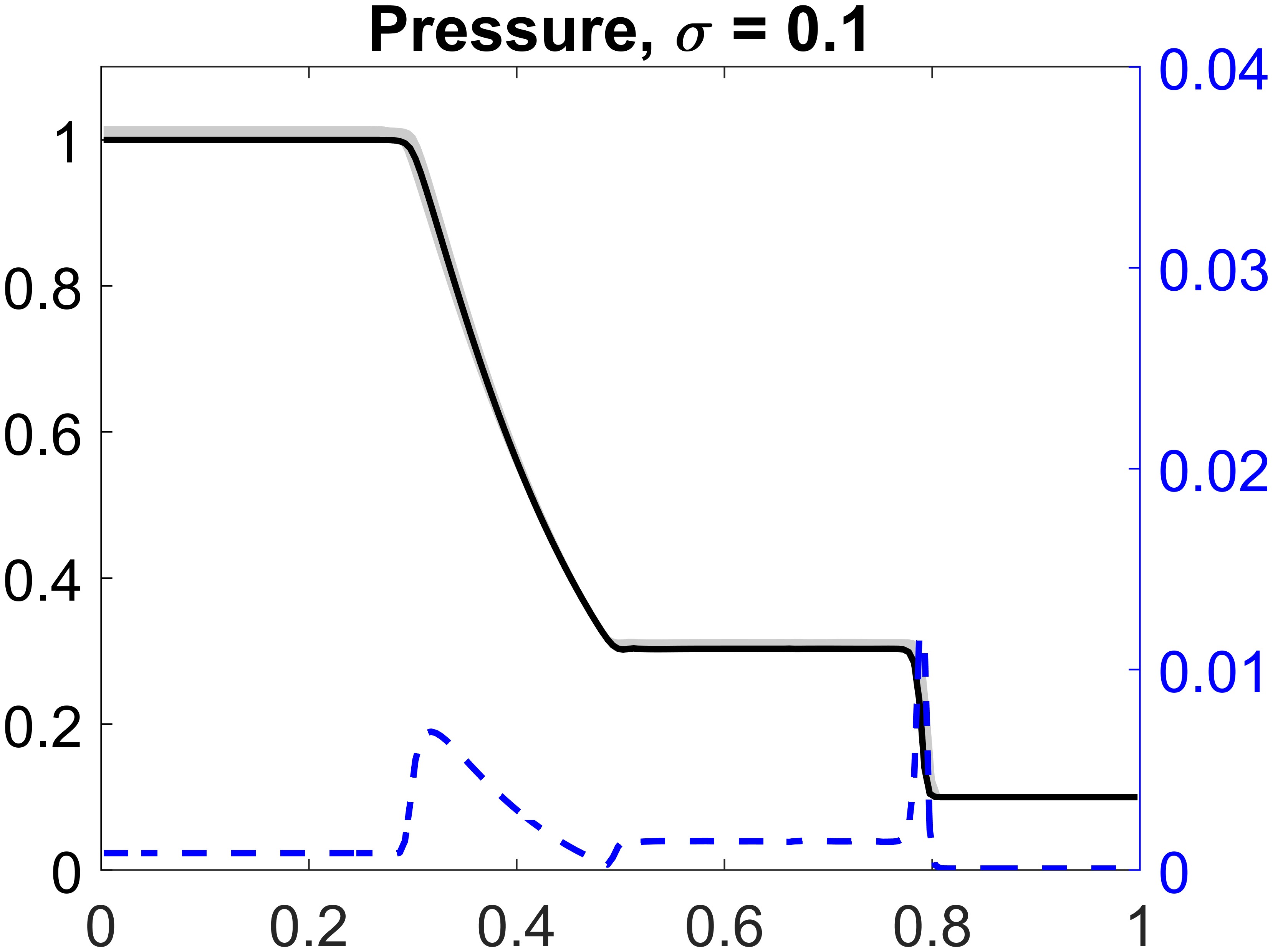}\hspace*{0.5cm}\includegraphics[width=0.31\textwidth]{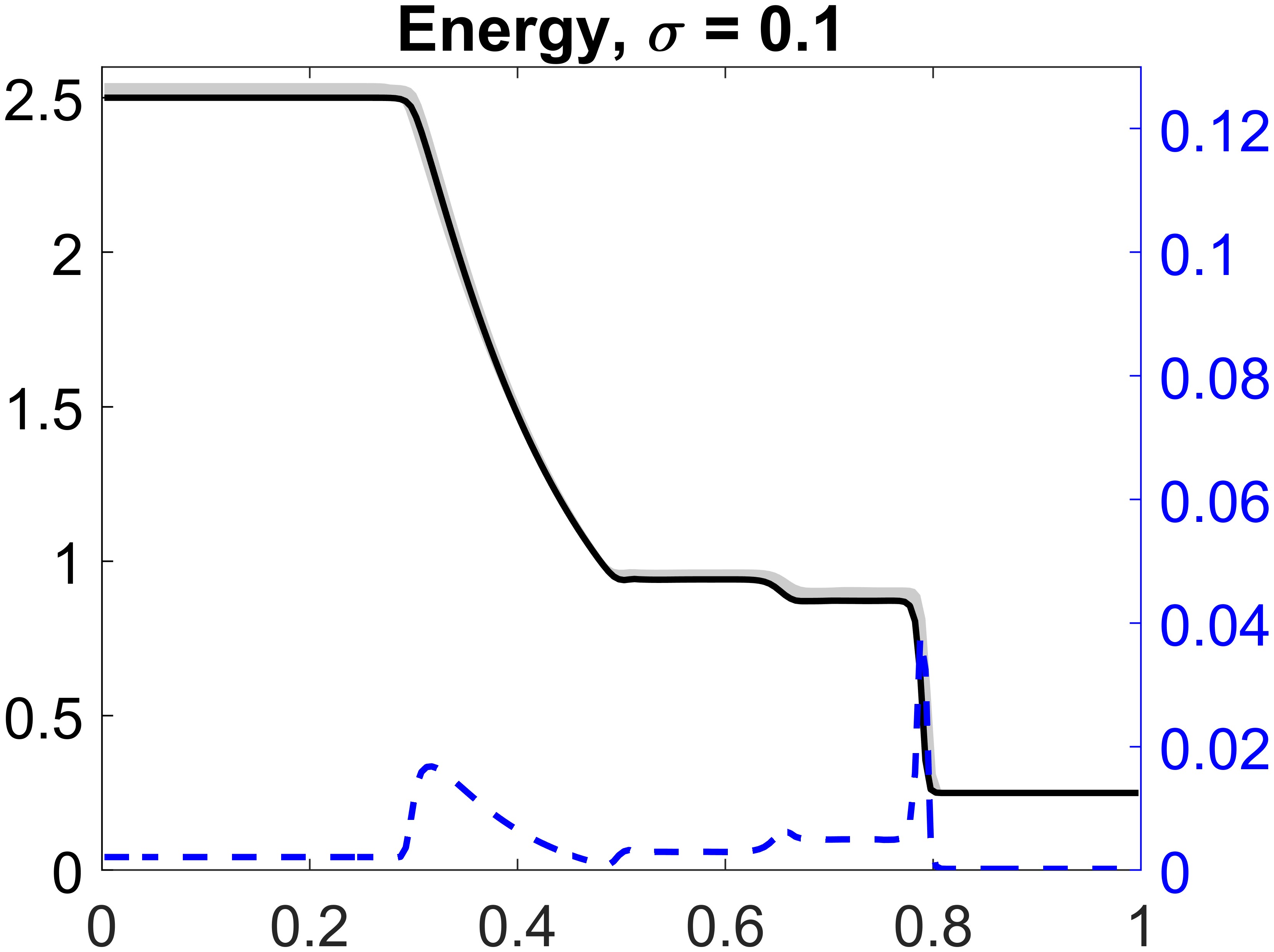}}
\caption{\sf Example 1, Test 1: {Same as in \fref{fig52}, but for $\xi\sim{\cal N}(0,1/36)$}.\label{fig52a}}
\end{figure}

\smallskip
\noindent {\bf Test 2.} Next, we consider the original Sod shock tube problem \eref{5.3} with the uncertainty introduced in the adiabatic
constant $\gamma$, which is now $\gamma(\xi)=1.4+0.1\xi$, {$\xi\sim{\cal U}(-1,1)$}. It is easy to see that the perturbation of
$\gamma$ affects the initial total energy only, and thus, one can expect that this perturbation will have the largest influence on the total
energy. Indeed, this is true, as one can see in \fref{fig53}, where we plot the mean, $95\%$ quantile, and standard deviation of $\rho$,
$\rho u$, and $E$ computed on a uniform mesh with $\dx=1/200$ and $\dxi=1/50$.
\begin{figure}[ht!]
\centerline{\includegraphics[width=0.32\textwidth]{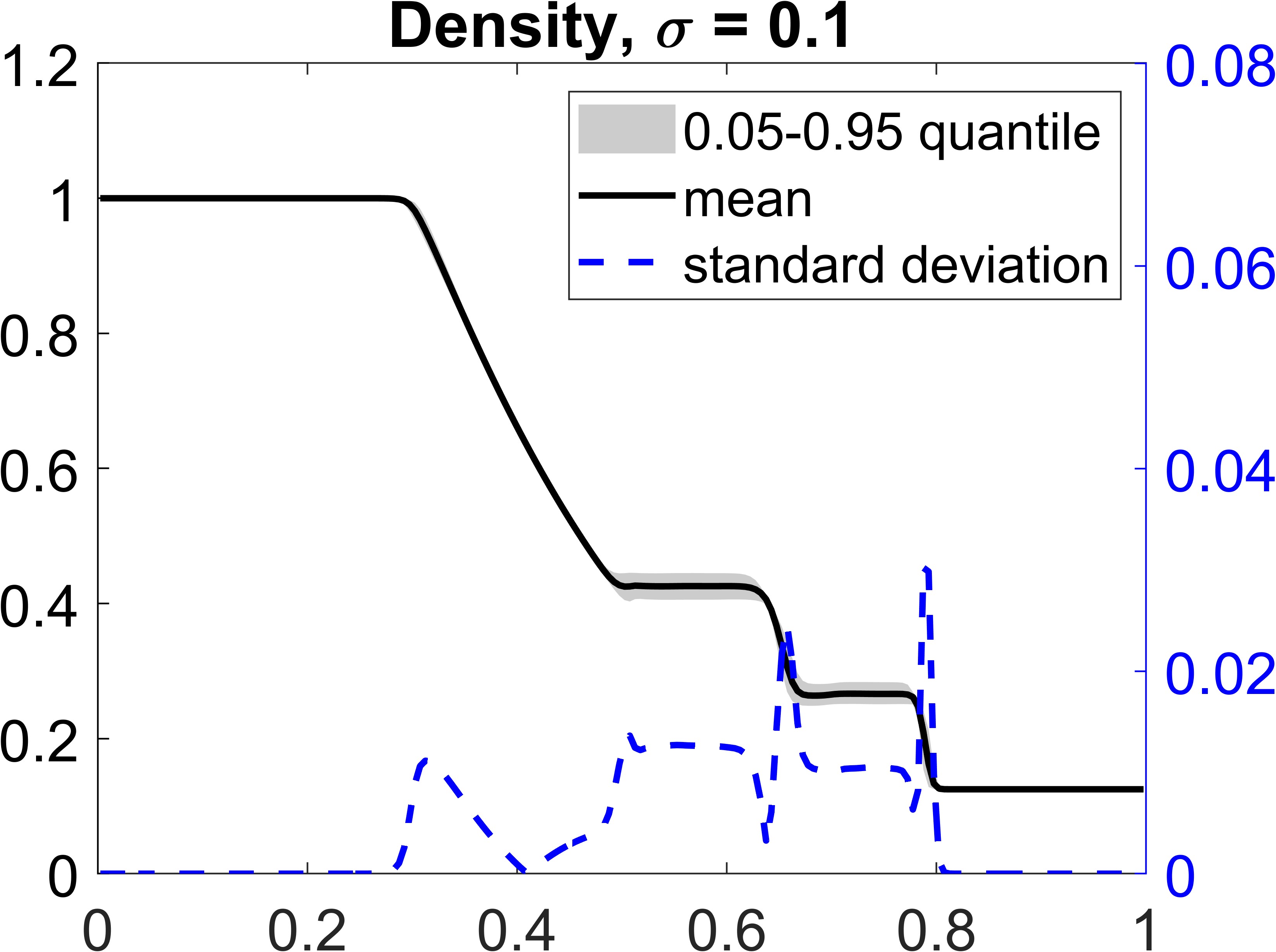}\hspace*{0.5cm}\includegraphics[width=0.32\textwidth]{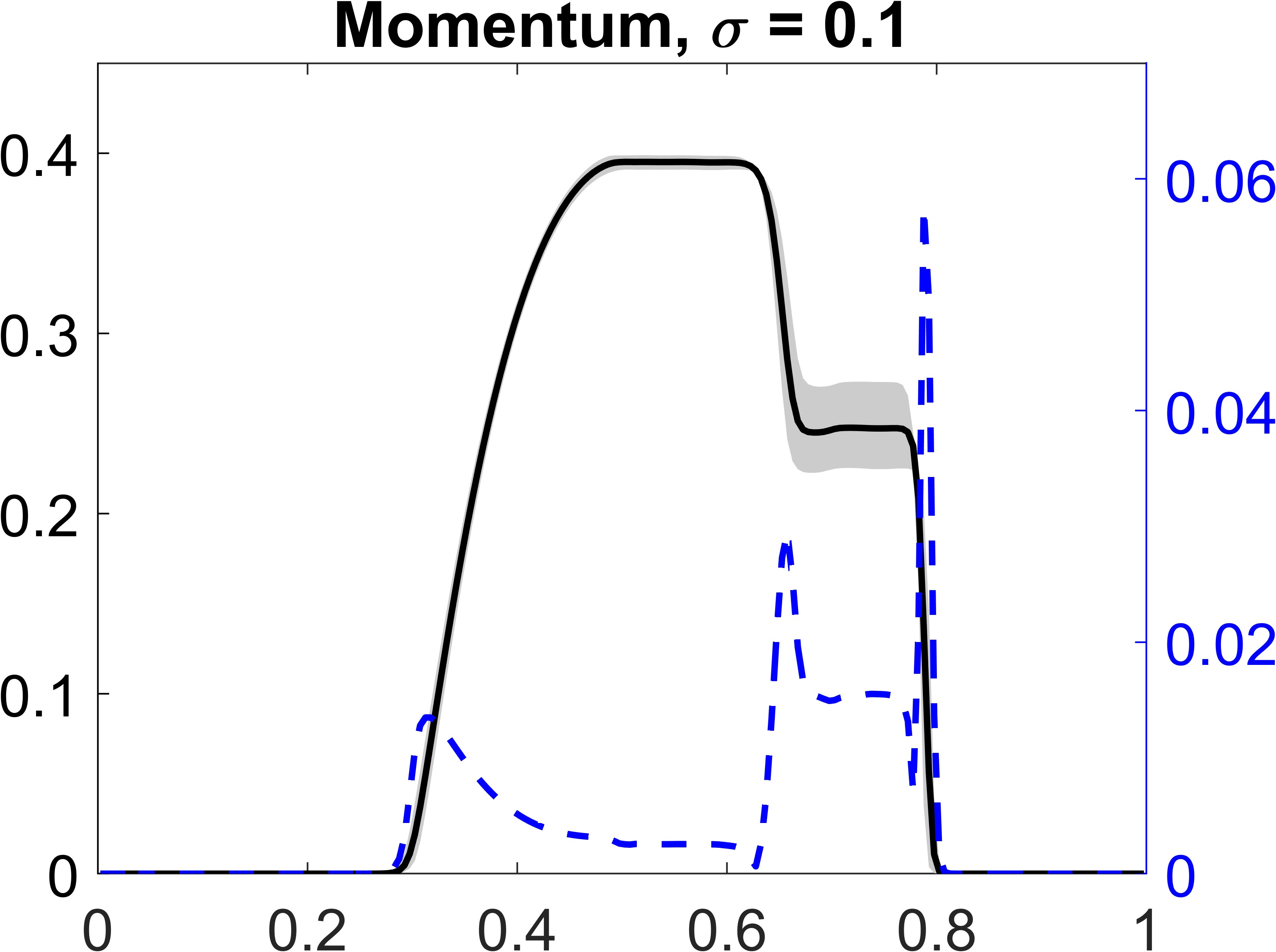}}
\vskip8pt
\centerline{\includegraphics[width=0.34\textwidth]{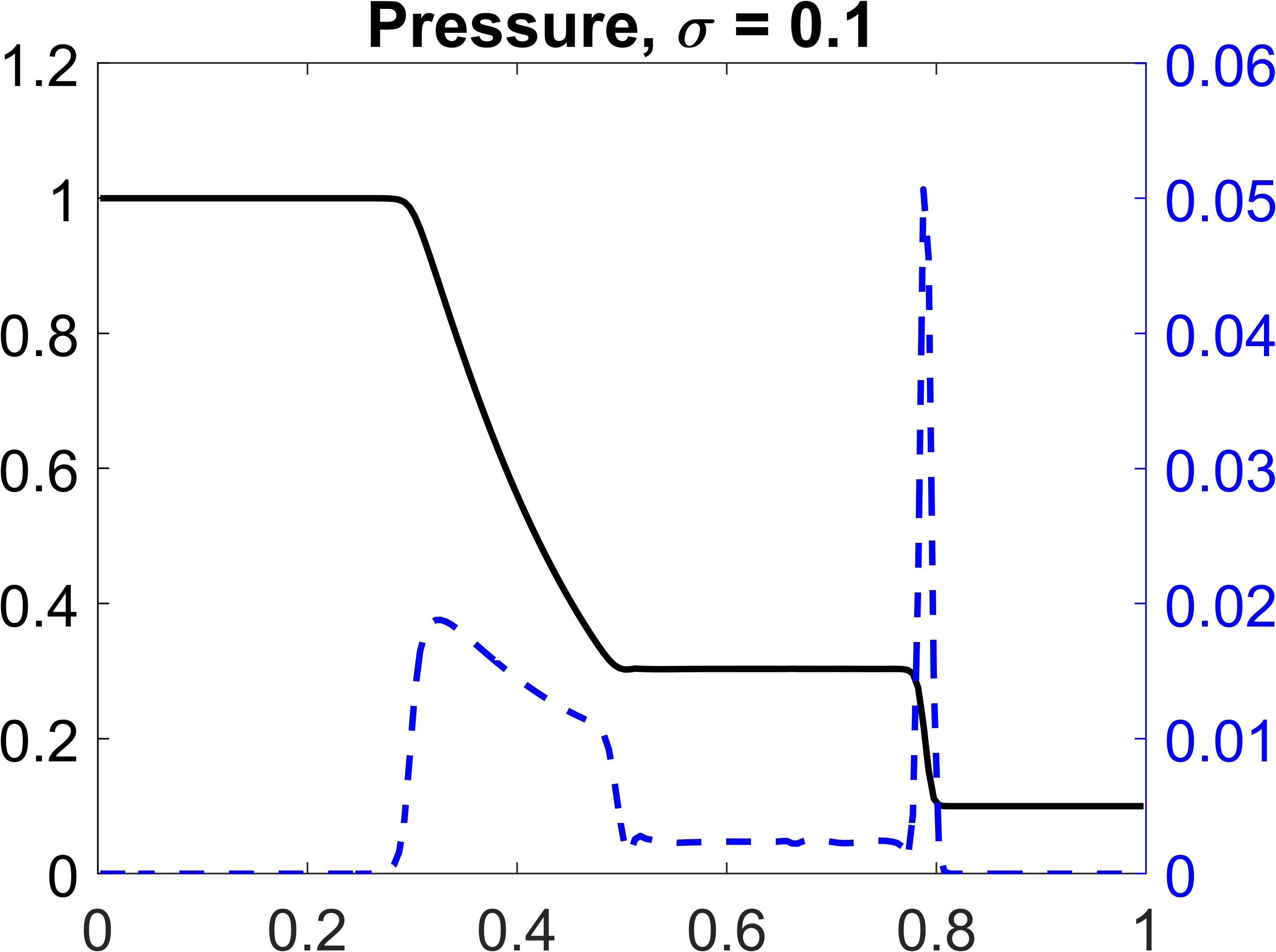}\hspace*{0.5cm}\includegraphics[width=0.32\textwidth]{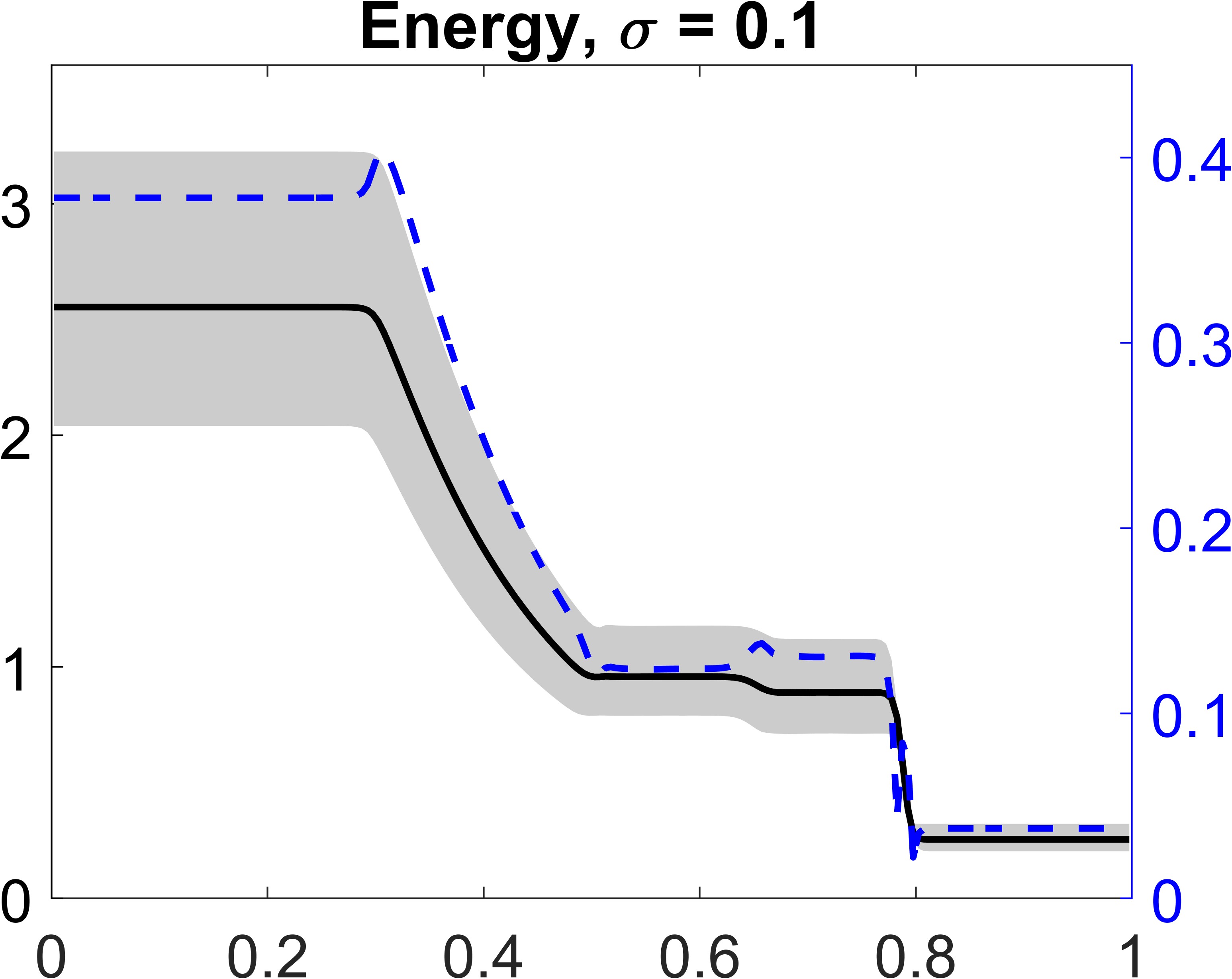}}
\caption{\sf Example 1, Test 2: Mean, 95\%-quantile, and standard deviation of $\rho$, $\rho u$, $p$, and $E$.\label{fig53}}
\end{figure}

\paragraph{Example 2 ($d=1$ and $s=2$).} The second example corresponds to the case of one spatial variable $x$ and two random variables
{$\xi\sim{\cal U}(-1,1)$ and $\eta\sim{\cal U}(-1,1)$, that is, $\nu(\xi,\eta)=\frac{1}{4}$}. In this case, the studied system is
given by \eref{2.1}, \eref{5.1}, \eref{5.2} but with all of the quantities depending not only on $x$, $\xi$, and $t$ but also on $\eta$.

We consider the same Sod shock tube problem \eref{5.3} studied in Example 1, but with the uncertainties in both the ICs and adiabatic
constant:
\begin{equation*}
(\rho(x,0,\xi,\eta),u(x,0,\xi,\eta),p(x,0,\xi,\eta))=\left\{
\begin{aligned}
&(1+0.1\xi,0,1),&&x<0.5,\\
&(0.125,0,0.1),&&x>0.5,
\end{aligned}
\right.\qquad\gamma(\eta)=1.4+0.1\eta.
\end{equation*}
We compute the solution on a uniform mesh with $\dx=1/200$ and $\dxi=\deta=1/50$. The obtained results are shown in \fref{fig54}, where one
can see the mean and standard deviation for $\rho$, $\rho u$, $p$, and $E$. These results confirm that the proposed numerical method can
handle the case of multidimensional random variables.
\begin{figure}[ht!]
\centerline{\includegraphics[width=0.32\textwidth]{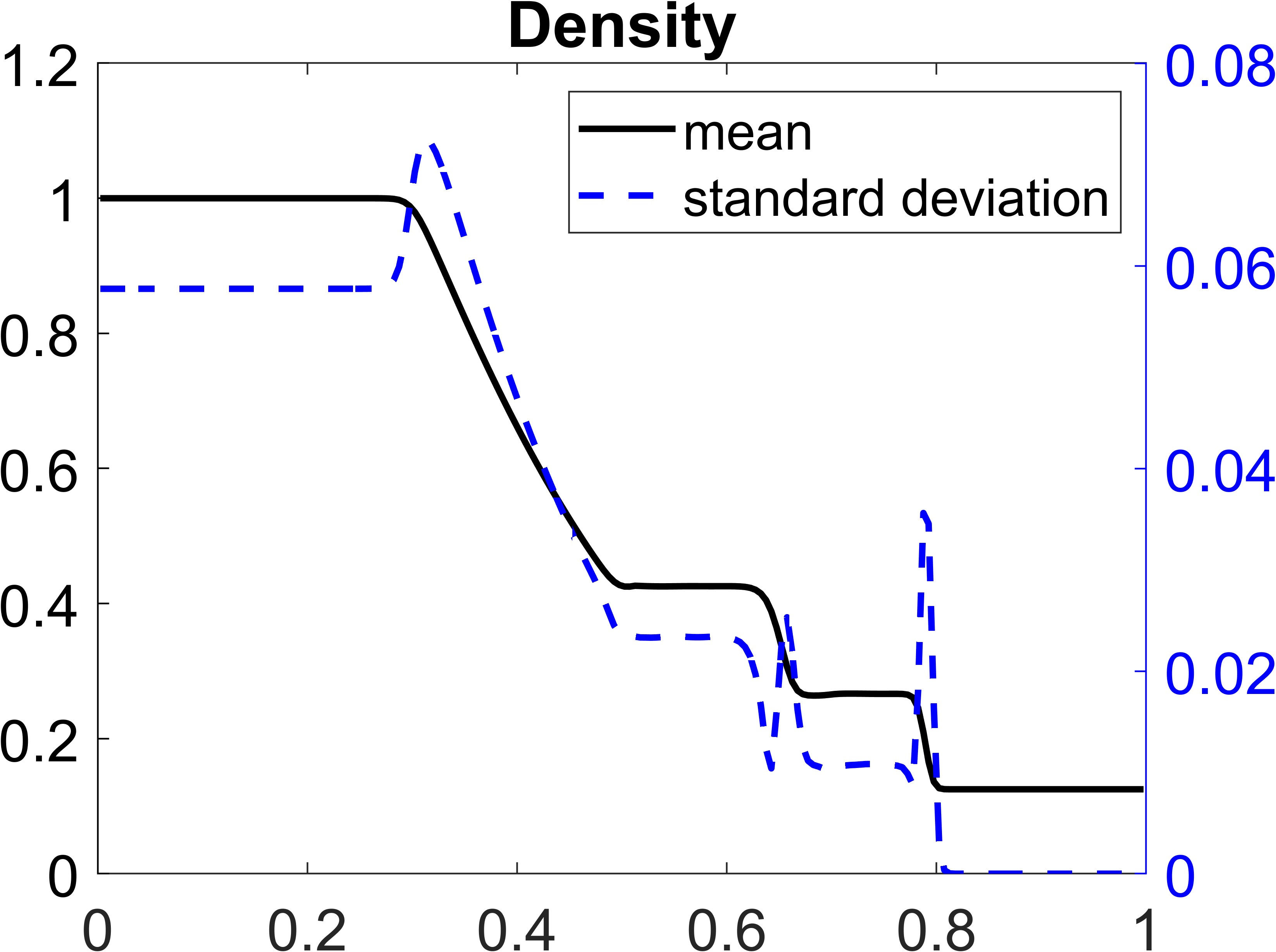}\hspace*{0.5cm}\includegraphics[width=0.32\textwidth]{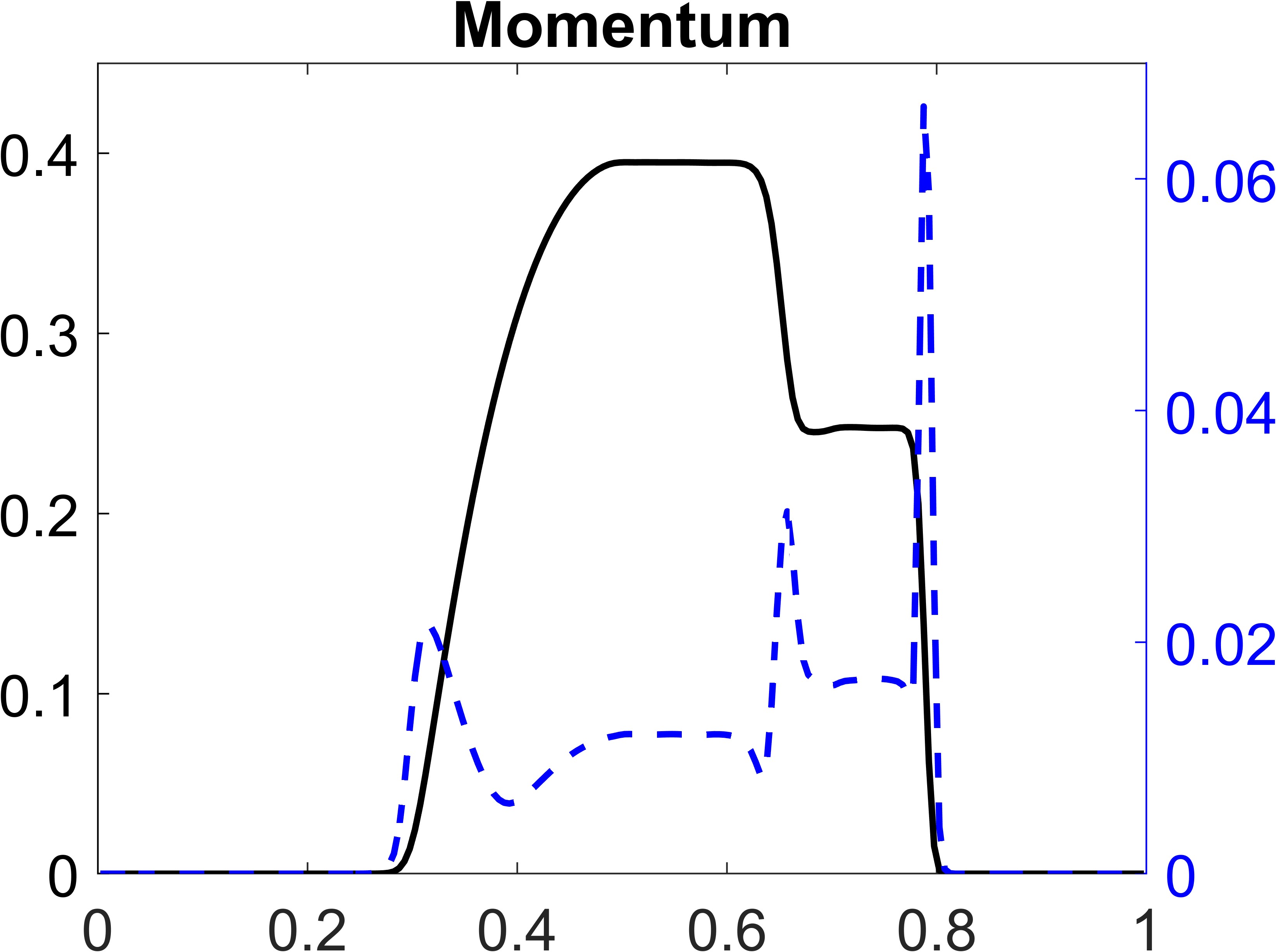}}
\vskip8pt
\centerline{\includegraphics[width=0.32\textwidth]{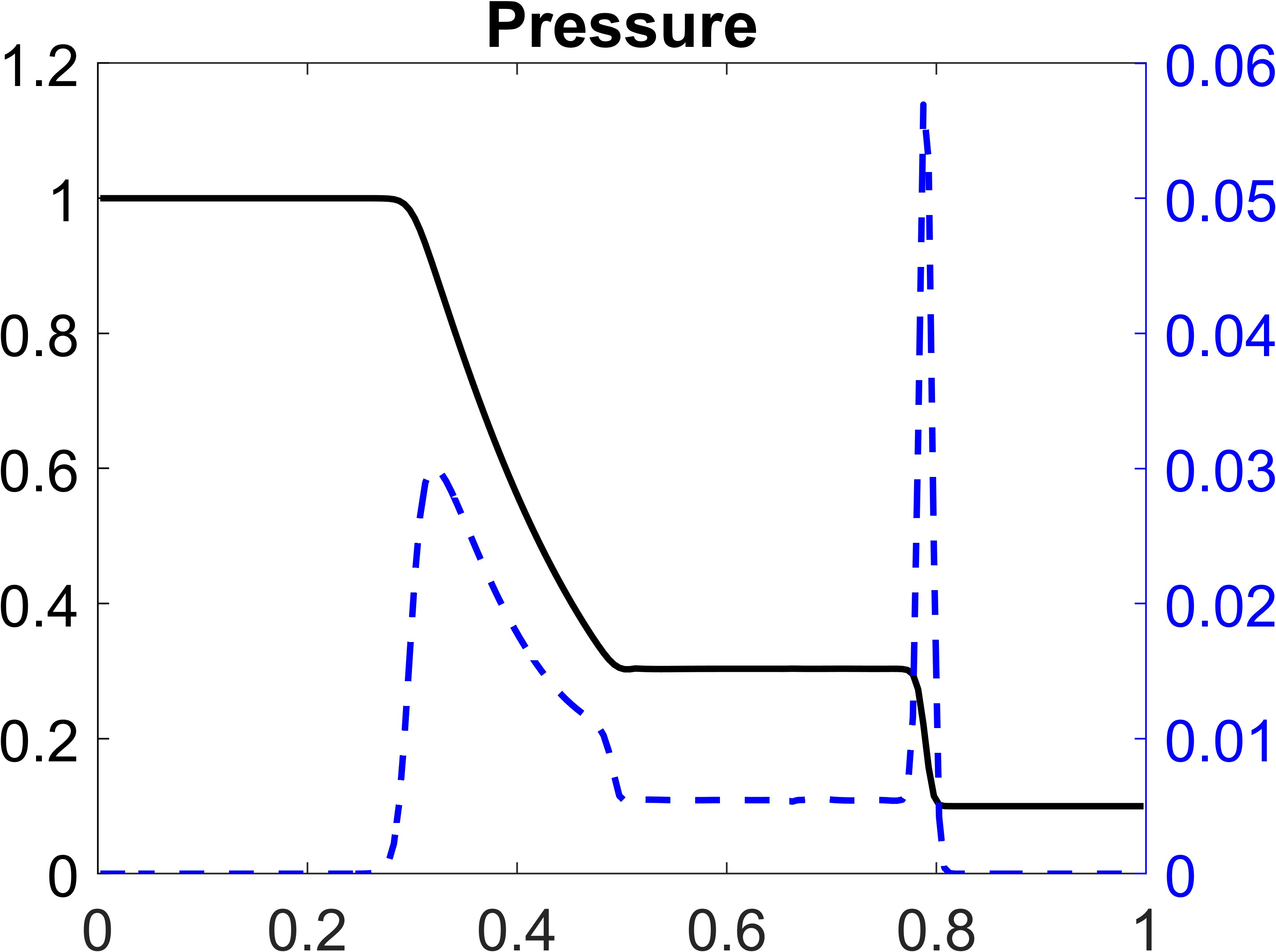}\hspace*{0.5cm}\includegraphics[width=0.32\textwidth]{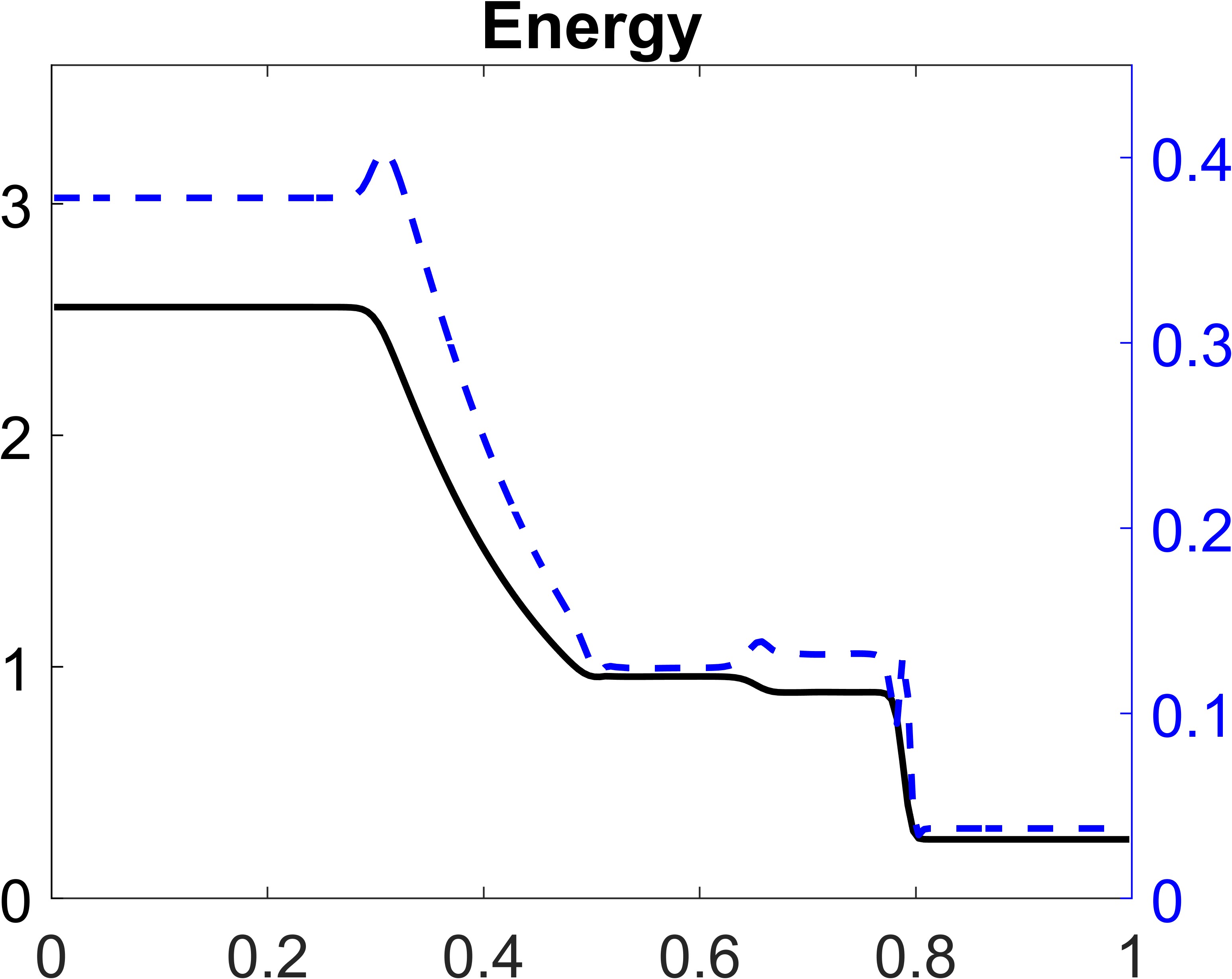}}
\caption{\sf Example 2: Mean and standard deviation of $\rho$, $\rho u$, $p$, and $E$.\label{fig54}}
\end{figure}

\paragraph{Example 3 ($d=2$ and $s=1$).} In the last numerical example, we consider the case of two spatial variables $x$ and $y$ and one
random variable {$\xi\sim{\cal U}(-1,1)$}, for which the Euler equations are governed by \eref{3.8} with
\begin{equation*}
\begin{aligned}
&\bm U=(\rho,\rho u,\rho v,E)^\top,\quad\bm F(\bm U)=\big(\rho u,\rho u^2+p,\rho uv,u(E+p)\big)^\top,\\
&\bm G(\bm U)=(\rho v,\rho uv,\rho u^2+p,v(E+p))^\top,\quad\bm S\equiv\bm0,
\end{aligned}
\end{equation*}
supplemented with the EOS:
\begin{equation*}
p=(\gamma-1)\left[E-\frac{\rho(u^2+v^2)}{2}\right]
\end{equation*}
with $\gamma=1.4$. Here, $u(x,y,t,\xi)$ and $v(x,y,t,\xi)$ are the $x$- and $y$-velocity components, respectively.

For the ICs, we take a two-dimensional (2-D) Riemann problem, which is a perturbed (for nonzero $\sigma$) version of configuration 10 from
\cite{KTrp}:
\begin{equation*}
(\rho,u,v,p)(x,y,0,\xi)=\left\{\begin{aligned}
&(1,0,0.4297(1-\sigma\xi),1),&&x>0.5,~y>0.5,\\
&(0.5(1+\sigma\xi),\sigma\xi,0.6076(1-\sigma\xi),1-\sigma\xi),&&x<0.5,~y>0.5,\\
&(0.2281(1+\sigma\xi),0,-0.6076(1+\sigma\xi),0.3333),&&x<0.5,~y<0.5,\\
&(0.4562,0,-0.4297(1+\sigma\xi),0.3333),&&x>0.5,~y<0.5.
\end{aligned}\right.
\end{equation*}
We first compute the deterministic solution ($\sigma=0$) on the spatial domain $(x,y)\in[0,1]\times[0,1]$ until time $t=0.15$ on a uniform
mesh with $\dx=\dy=1/400$. The resulting density component is shown in \fref{fig5.5} (left). We then use a quite large perturbation with 
$\sigma=0.5$ and repeat the same numerical experiment using $\dxi=1/10$. The mean density obtained is quite different from the
deterministic case; see \fref{fig5.5} (middle). We also present the computed standard deviation in \fref{fig5.5} (right), where one can see
the areas in which the uncertainties influence the solution the most.
\begin{figure}[ht!]
\centerline{\includegraphics[width=0.30\textwidth]{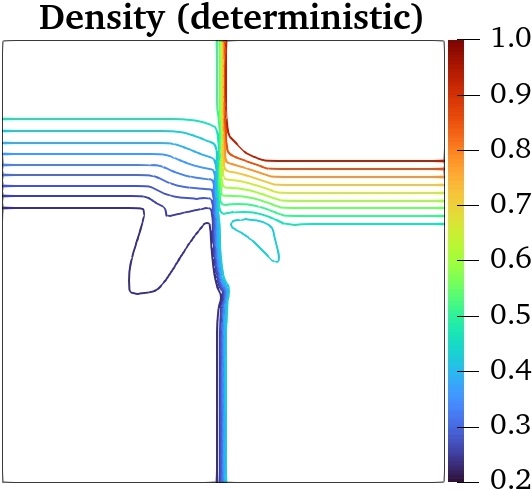}\hspace*{0.5cm}\includegraphics[width=0.30\textwidth]{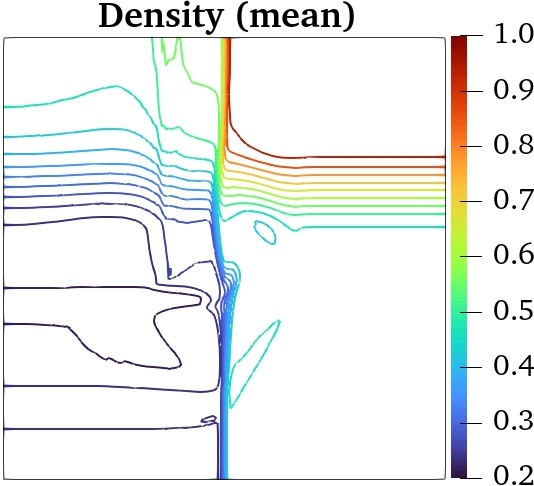}
\hspace*{0.5cm}\includegraphics[width=0.31\textwidth]{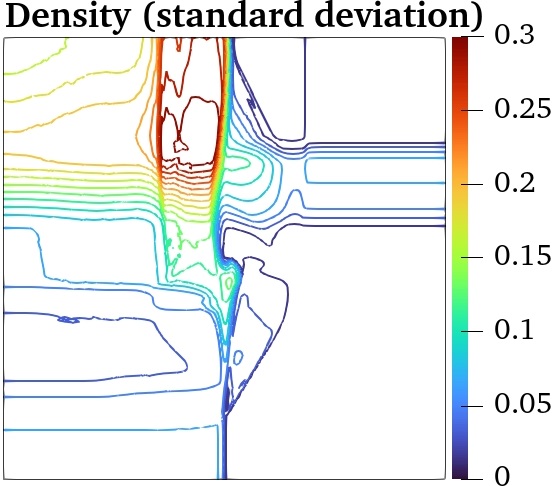}}
\caption{\sf Example 3: Deterministic $\rho$ (left), and mean (middle) and standard deviation (right) of the uncertain $\rho$.
\label{fig5.5}}
\end{figure}

\section{Saint-Venant System of Shallow Water Equations}\label{sec6}
We proceed by considering the Saint-Venant system of shallow water equations. There are several challenges associated with the numerical
solution of the Saint-Venant system. First, it contains a geometric source term, which is balanced by the corresponding flux. This requires
the development of the so-called WB schemes, which maintain this delicate balance at the discrete level. In addition, the water depth may
naturally be very small or even zero, and thus a good numerical method must be able to guarantee the non-negativity of the computed water
depth. The level of complexity increases even further with the presence of uncertainties as several additional techniques (compared to those
presented in \S\ref{sec2}) need to be developed. In what follows, we adopt the approach introduced in \cite{KPshw} and extend it to the 1-D
and 2-D Saint-Venant systems with uncertainties.

\subsection{Case $d=1$ and $s=1$}\label{sec61}
We begin with the 1-D case, in which the Saint-Venant system is given by \eref{2.1} with
\begin{equation}
\bm U=(h,hu)^\top,\quad\bm F(\bm U)=\Big(hu,hu^2+\frac{g}{2}h^2\Big)^\top,\quad\bm S=(0,-ghZ_x)^\top,
\label{6.1}
\end{equation}
where $h(x,t,\xi)$ is the water depth, $u(x,t,\xi)$ is the velocity, $Z(x,\xi)$ is a time-independent bottom topography, and $g$ is the
constant acceleration due to gravity.

The system \eref{2.1}, \eref{6.1} admits several steady states, among which the simplest ones are the so-called ``lake-at-rest'' equilibria:
\begin{equation}
u(x,\xi)\equiv0,\quad w(x,\xi):=h(x,\xi)+Z(x,\xi)=\widehat{w}(\xi),
\label{6.2}
\end{equation}
where $\widehat w$ is an arbitrary function of $\xi$ and the equilibrium variable $w$ represents the water surface. We will say that the
scheme is WB if it is capable of preserving the ``lake-at-rest'' states \eref{6.2} at the discrete level.

We follow the steps described in \S\ref{sec2} to develop a scheme for \eref{2.1}, \eref{6.1}. However, in order to ensure the WB property,
we do not reconstruct the point values of $h$---we instead reconstruct the water surface $w:=h+Z${, whose weighted cell averages
are computed using the following relationship:
$$
\xbar w_{j,\ell}=\,\xbar h_{j,\ell}+\,\xbar Z_{j,\ell},\quad
\,\xbar Z_{j,\ell}=\hf\sum_{i=1}^M\mu_i\,\nu(\xi_{\ell_i})\left[Z_{\jmh,\ell_i}+Z_{\jph,\ell_i}\right].
$$
Here, $Z_{j\pm\hf,\ell_i}:=Z(x_{j\pm\hf},\xi_{\ell_i})$ and the weighted cell average of $Z$ was obtained using the trapezoidal rule in $x$
and the same Gauss-Legendre quadrature as before in $\xi$.
}

We note that some of the point values $w^+_{\jmh,\ell_i}$ and $w^-_{\jph,\ell_i}$, reconstructed using \eref{2.8}--\eref{2.10} followed by
the Ai-WENO-Z interpolation introduced in \S\ref{sec222}, may not satisfy the water depth positivity requirement. We thus enforce the
non-negativity of the point values of $h$ in the simplest way by setting
\begin{equation*}
h^+_{\jmh,\ell_i}=\max\left\{w^+_{\jmh,\ell_i},Z_{\jmh,\ell_i}\right\}-Z_{\jmh,\ell_i},\quad
h^-_{\jph,\ell_i}=\max\left\{w^-_{\jph,\ell_i},Z_{\jph,\ell_i}\right\}-Z_{\jph,\ell_i}.
\end{equation*}

Once the reconstructed point values of $h$ and $hu$ are available, we need to compute the corresponding point values of
$u=\nicefrac{(hu)}{h}$. In order to avoid a division by small numbers, as $h$ may be very small, the computation of $u$ must be
desingularized. We use the following desingularization formula (see \cite{KPshw}):
\begin{equation}
u=\frac{2h(hu)}{h^2+[\max(h,\varepsilon)]^2},
\label{6.3}
\end{equation}
where $\varepsilon$ is a small positive number. For consistency, we then recalculate the corresponding point values of $hu$ by setting
\begin{equation}
(hu):=h\cdot u.
\label{6.4}
\end{equation}
We note that all of the indexes in formulae \eref{6.3} and \eref{6.4} have been omitted for the sake of brevity.
{
\begin{rmk}
It is well-known that when (almost) dry areas are present, one may prefer to reconstract $u$ rather the $hu$. In order to implement this
approach for the studied Saint-Venant system with uncertainties, one would need to first recover the point values of $u$ at the cell
centers out of the given cell averages of $h$ and $hu$. While this can be easily done for first- and second-order schemes, high-order
recovery procedure may be rather cumbersome. We therefore prefer to reconstruct $hu$ and then to use the value of the desingularization
parameter $\varepsilon$ to control the oscillations in $u$ and thus to prevent appearance of artificially small time steps, which may
otherwise significantly affect the efficiency of the resulting scheme.
\end{rmk}
}

Finally, we follow the lines of \cite{KPshw} and design a WB quadrature for the integral on the RHS of \eref{2.7} of the second component
of source term $\bm S$ in \eref{6.1}:
\begin{equation}
\int\limits_{x_\jmh}^{x_\jph}\big(-gh(x,\xi_{\ell_i},t)Z_x(x,\xi_{\ell_i})\big)\,{\rm d}x\approx
-\frac{g}{2}\big(h_{\jmh,\ell_i}^++h_{\jph,\ell_i}^-\big)\big(Z_{\jph,\ell_i}-Z_{\jmh,\ell_i}\big).
\label{6.5}
\end{equation}
The use of the quadrature \eref{6.5} in \eref{2.7} leads to the WB semi-discrete scheme \eref{2.2}, \eref{2.4}--\eref{2.7}, \eref{6.5} as we
prove in the following theorem.
\begin{thm}\label{thm61}
Assume that at a certain time level, the discrete solution is at a ``lake-at-rest'' steady state, that is,
\begin{equation*}
\xbar w_{j,\ell}=\,\xbar h_{j,\ell}+\xbar Z_{j,\ell}\equiv\widehat w_\ell,\quad(\xbar{hu})_{j,\ell}\equiv0,\quad\forall j,\ell.
\end{equation*}
Then, the RHS of the system of ODEs \eref{2.2} vanishes and hence the proposed semi-discrete scheme \eref{2.2}, \eref{2.4}--\eref{2.7},
\eref{6.5} is WB.
\end{thm}
\begin{proof}
We first observe that the minmod reconstruction in the $x$-direction \eref{2.8}--\eref{2.10} ensures
\begin{equation}
w^-_{\jph,\ell}=w^+_{\jph,\ell}=\widehat{w}_\ell,\quad u^-_{\jph,\ell}=u^+_{\jph,\ell}=0,\quad\forall j,\ell.
\label{6.6}
\end{equation}
Then, the fifth-order Ai-WENO-Z interpolation described in \S\ref{sec222} results in
\begin{equation}
\begin{aligned}
&w^-_{\jph,\ell+\kappa}=w^+_{\jph,\ell+\kappa}=w_{\ell+\kappa},\\
&w^-_{\jph,\ell-\kappa}=w^+_{\jph,\ell-\kappa}=w_{\ell-\kappa},
\end{aligned}
\qquad u^-_{\jph,\ell\pm\kappa}=u^+_{\jph,\ell\pm\kappa}=0,\quad\forall j,\ell,
\label{6.7}
\end{equation}
where the values $w_{\ell\pm\kappa}$ are independent of $j$ since for the ``lake-at-rest'' data, the Ai-WENO-Z interpolation is conducted
over the same set of discrete values along $x=x_\jph,\,\forall j$.

Using \eref{6.6} and \eref{6.7} in \eref{2.4}--\eref{2.5} and \eref{2.7}, \eref{6.5}, we obtain the following formula for the CU numerical
fluxes $\bm{{\cal F}}_{\jph,\ell}=({\cal F}^{(1)}_{\jph,\ell},{\cal F}^{(2)}_{\jph,\ell})^\top$:
\begin{equation}
\begin{aligned}
&{\cal F}^{(1)}_{\jph,\ell}=0,\\
&{\cal F}^{(2)}_{\jph,\ell}=\frac{\mu_1g}{2}\big(w_{\ell-\kappa}-Z_{\jph,\ell-\kappa}\big)^2+
\frac{\mu_2g}{2}\big(\widehat w_{\ell}-Z_{\jph,\ell}\big)^2+\frac{\mu_3g}{2}\big(w_{\ell+\kappa}-Z_{\jph,\ell+\kappa}\big)^2,
\end{aligned}
\label{6.8}
\end{equation}
and the source term $\bm S_{j,\ell}=(0,S^{(2)}_{j,\ell})^\top$:
\begin{equation}
\begin{aligned}
S^{(2)}_{j,\ell}=&-\frac{\mu_1g}{2}\big(2w_{\ell-\kappa}-Z_{\jph,\ell-\kappa}-Z_{\jmh,\ell-\kappa}\big)
\big(Z_{\jph,\ell-\kappa}-Z_{\jmh,\ell-\kappa}\big)\\
&-\frac{\mu_2g}{2}\big(2\widehat w_\ell-Z_{\jph,\ell}-Z_{\jmh,\ell}\big)\big(Z_{\jph,\ell}-Z_{\jmh,\ell}\big)\\
&-\frac{\mu_3g}{2}\big(2w_{\ell+\kappa}-Z_{\jph,\ell+\kappa}-Z_{\jmh,\ell+\kappa}\big)
\big(Z_{\jph,\ell+\kappa}-Z_{\jmh,\ell+\kappa}\big).
\end{aligned}
\label{6.9}
\end{equation}
Substituting \eref{6.8} and \eref{6.9} into \eref{2.2} and using a simplifying the obtained expressions, yields
\begin{equation*}
\frac{{\rm d}}{{\rm d}t}\,\xbar{\bm U}_{j,\ell}(t)=\bm0,
\end{equation*}
which completes the proof of the theorem.
\end{proof}

\subsection{Case $d=1$ and $s=2$}\label{sec62}
We now consider the Saint-Venant system \eref{2.1}, \eref{6.1}, where all of the quantities depend on two independent random variables $\xi$
and $\eta$ in addition to the dependence on the physical variables $x$ and $t$. In this case, the ``lake-at-rest'' equilibria are given by
\begin{equation*}
u(x,\xi,\eta)\equiv0,\quad w(x,\xi,\eta):=h(x,\xi,\eta)+Z(x,\xi,\eta)=\widehat w(\xi,\eta),
\end{equation*}
where $\widehat w$ is now an arbitrary function of both $\xi$ and $\eta$.

As in \S\ref{sec61}, we reconstruct $w$ and $hu$ {and the weighted cell averages of $w$ are calculated using the following
relationship:
$$
\xbar w_{j,\ell,m}=\,\xbar h_{j,\ell,m}+\,\xbar Z_{j,\ell,m},\quad
\,\xbar Z_{j,\ell,m}=\hf\sum_{i,r=1}^M\mu_i\,\mu_r\,\nu(\xi_{\ell_i},\eta_{m_r})\left[Z_{\jmh,\ell_i}+Z_{\jph,\ell_i}\right].
$$
Here, $Z_{j\pm\hf,\ell_i,m_r}:=Z(x_{j\pm\hf},\xi_{\ell_i},\eta_{m_r})$ and the weighted cell average of $Z$ was obtained using the
trapezoidal rule in $x$ and the same Gauss-Legendre quadrature as before in $\xi$ and $\eta$.} We then obtain the non-negative point values
of $h$ by setting
\begin{equation*}
\begin{aligned}
&h^+_{\jmh,\ell_i,m_r}=\max\left\{w^+_{\jmh,\ell_i,m_r},Z_{\jmh,\ell_i,m_r}\right\}-Z_{\jmh,\ell_i,m_r},\\
&h^-_{\jph,\ell_i,m_r}=\max\left\{w^-_{\jph,\ell_i,m_r},Z_{\jph,\ell_i,m_r}\right\}-Z_{\jph,\ell_i,m_r},
\end{aligned}
\end{equation*}
where we desingularize the computation of the point values of $u$ using \eref{6.3} and \eref{6.4}. Finally, the WB quadrature \eref{6.5} now
reads as
\begin{equation}
\begin{aligned}
\int\limits_{x_\jmh}^{x_\jph}\big(&-gh(x,\xi_{\ell_i},\eta_{m_r}t)Z_x(x,\xi_{\ell_i},\eta_{m_r})\big)\,{\rm d}x\\
\approx&-\frac{g}{2}\big(h_{\jmh,\ell_i,m_r}^++h_{\jph,\ell_i,m_r}^-\big)\big(Z_{\jph,\ell_i,m_r}-Z_{\jmh,\ell_i,m_r}\big).
\end{aligned}
\label{6.10}
\end{equation}
The use of the quadrature \eref{6.10} in \eref{3.6} leads to the WB semi-discrete scheme \eref{3.2}--\eref{3.6}, \eref{6.10}, as
stated in the following Theorem \ref{thm62}, whose proof is completely analogous to the proof of Theorem \ref{thm61}.
\begin{thm}\label{thm62}
Assume that at a certain time level, the discrete solution is at a ``lake-at-rest'' steady state, that is,
$$
\xbar w_{j,\ell,m}=\xbar h_{j,\ell,m}+\xbar Z_{j,\ell,m}\equiv\widehat w_{\ell,m},\quad(\xbar{hu})_{j,\ell,m}\equiv0,\quad\forall j,\ell,m.
$$
Then, the RHS of the system of ODEs \eref{3.2} vanishes and hence the proposed semi-discrete scheme \eref{3.2}--\eref{3.6}, \eref{6.10} is
WB.
\end{thm}

\subsection{Case $d=2$ and $s=1$}\label{sec63}
In this section, we consider a 2-D physical domain and one random variable $\xi$. The shallow water equations are given by \eref{3.8} with
\begin{equation}
\begin{aligned}
&\bm U=(h,hu,hv)^\top,\quad\bm F(\bm U)=\Big(hu,hu^2+\frac{g}{2}h^2,huv\Big)^\top,\\
&\bm G(\bm U)=\Big(hv,huv,hv^2+\frac{g}{2}h^2\Big)^\top,\quad\bm S=(0,-ghZ_x,-ghZ_y)^\top,
\end{aligned}
\label{6.11}
\end{equation}
where $h(x,y,\xi,t)$ is the water depth, $u(x,y,\xi,t)$ and $v(x,y,\xi,t)$ are the $x$- and $y$-velocities, $Z(x,y,\xi)$ is the bottom
topography and $g$ is the constant acceleration due to gravity.

The system \eref{1.1}, \eref{6.11} also admits the ``lake-at-rest'' steady states satisfying
\begin{equation*}
u(x,y,\xi)=v(x,y,\xi)\equiv0,\quad w(x,y,\xi)=\widehat w(\xi),
\end{equation*}
where $\widehat{w}$ is an arbitrary function of $\xi$.

As in the 1-D case, in order to develop a WB scheme for \eref{3.8}, \eref{6.11}, we first reconstruct the point values of
$w${, whose weighted cell averages are computed using the following relationship:
$$
\xbar w_{j,k,\ell}=\,\xbar h_{j,k,\ell}+\,\xbar Z_{j,k,\ell},\quad
\,\xbar Z_{j,k,\ell}=\frac{1}{4}\sum_{i=1}^M\mu_i\,\nu(\xi_{\ell_i})
\left[Z_{\jmh,k,\ell_i}+Z_{\jph,k,\ell_i}+Z_{j,\kmh,\ell_i}+Z_{j,\kph,\ell_i}\right].
$$
Here, $Z_{j\pm\hf,k,\ell_i}:=Z(x_{j\pm\hf},y_k,\xi_{\ell_i})$, $Z_{j,k\pm\hf,\ell_i}:=Z(x_j,y_{k\pm\hf},\xi_{\ell_i})$, and the weighted
cell average of $Z$ was obtained using the trapezoidal-type rule in $x$ and $y$ and the same Gauss-Legendre quadrature as before in $\xi$.}
We then calculate the point values of $h$ by enforcing their non-negativity through
\begin{equation*}
\begin{aligned}
&h^+_{\jmh,k,\ell_i}=\max\left\{w^+_{\jmh,k,\ell_i},Z_{\jmh,k,\ell_i}\right\}-Z_{\jmh,k,\ell_i},\\
&h^-_{\jph,k,\ell_i}=\max\left\{w^-_{\jph,k,\ell_i},Z_{\jph,k,\ell_i}\right\}-Z_{\jph,k,\ell_i},\\
&h^+_{j,\kmh,\ell_i}=\max\left\{w^+_{j,\kmh,\ell_i},Z_{j,\kmh,\ell_i}\right\}-Z_{j,\kmh,\ell_i},\\
&h^-_{j,\kph,\ell_i}=\max\left\{w^-_{j,\kph,\ell_i},Z_{j,\kph,\ell_i}\right\}-Z_{j,\kph,\ell_i}.
\end{aligned}
\end{equation*}

Once the reconstructed point values of $h, hu$, and $hv$ are available, we compute the corresponding point values of $u$ and $v$ by
implementing the same desingularization procedure as in \eref{6.3}:
\begin{equation}
u=\frac{2h(hu)}{h^2+[\max(h,\varepsilon)]^2},\quad v=\frac{2h(hv)}{h^2+[\max(h,\varepsilon)]^2}
\label{6.12}
\end{equation}
where $\varepsilon$ is a small positive number. For consistency, we then recalculate the corresponding point values of $hu$ and $hv$ by
setting
\begin{equation}
(hu):=h\cdot u,\quad(hv):=h\cdot v.
\label{6.13}
\end{equation}
As before, the indexes in formulae \eref{6.12} and \eref{6.13} have been omitted for the sake of brevity.

We complete the construction of the scheme by designing a WB quadrature for the integral on the RHS of \eref{3.13} of the second and third
components of source term $\bm S$ in \eref{6.11}. As in the 1-D case, we follow the lines of \cite{KPshw} and obtain
\begin{equation}
\begin{aligned}
\int\limits_{y_\kmh}^{y_\kph}\int\limits_{x_\jmh}^{x_\jph}\big(&-gh(x,y,\xi_{\ell_i},t)Z_x(x,y,\xi_{\ell_i})\big)\,{\rm d}x\,{\rm d}y\\
\approx&-\frac{g\dy}{2}\big(h_{\jmh,k,\ell_i}^++h_{\jph,k,\ell_i}^-\big)\big(Z_{\jph,k,\ell_i}-Z_{\jmh,k,\ell_i}\big),\\
\int\limits_{y_\kmh}^{y_\kph}\int\limits_{x_\jmh}^{x_\jph}\big(&-gh(x,y,\xi_{\ell_i},t)Z_y(x,y,\xi_{\ell_i})\big)\,{\rm d}x\,{\rm d}y\\
\approx&-\frac{g\dx}{2}\big(h_{j,\kmh,\ell_i}^++h_{j,\kph,\ell_i}^-\big)\big(Z_{j,\kph,\ell_i}-Z_{j,\kmh,\ell_i}\big).
\end{aligned}
\label{6.14}
\end{equation}
The use of the quadrature \eref{6.14} in \eref{3.13} leads to the WB semi-discrete scheme \eref{3.9}--\eref{3.13}, \eref{6.14} as stated in
the Theorem \ref{thm63}, whose proof is similar to the proof of Theorem \ref{thm61} and thus omitted for the sake of brevity.
\begin{thm}\label{thm63}
Assume that at a certain time level, the discrete solution is at a ``lake-at-rest'' steady state, that is,
\begin{equation*}
\xbar w_{j,k,\ell}=\xbar h_{j,k,\ell}+\xbar Z_{j,k,\ell}\equiv\widehat w_\ell,\quad(\xbar{hu})_{j,k,\ell}=(\xbar{hv})_{j,k,\ell}\equiv0,
\quad\forall j,k,\ell.
\end{equation*}
Then, the RHS of the system of ODEs \eref{3.9} vanishes and hence the proposed semi-discrete scheme \eref{3.9}--\eref{3.13}, \eref{6.14} is
WB.
\end{thm}

\subsection{Numerical Examples}\label{sec64}
In this section, we illustrate the performance of the proposed numerical method on several examples for the Saint-Venant system. In all
of the examples, we take $g=1$, except for Examples 7 and 8, in which we take $g=2$ and $g=9.812$, respectively. We have used the minmod
parameter either $\theta=1$ (Examples 4, 7, and 8) or $\theta=1.3$ (Examples 5, 6, and 9). {The desingularization parameter
$\varepsilon$ is set to $10^{-6}$ in Examples 4, 6, and 9, $\varepsilon=10^{-3}$ in Example 5, $\varepsilon=10^{-5}$ in Example 7, and
$\varepsilon=5\times10^{-4}$ in Example 8.}

\paragraph{Example 4 (Accuracy Test, $d=s=1$).} We begin with experimentally verifying the fifth order of accuracy in $\xi$ of the proposed
finite-volume method. To this end, we consider the following smooth ICs with uncertainty in the water surface:
\begin{equation*}
w(x,0,\xi)=1+0.1\tanh(4\xi)+0.01\sin(2\pi x),\quad u(x,0,\xi)\equiv0.1,\quad{\xi\sim{\cal U}(-1,1)}.
\end{equation*}
We take a flat bottom topography $Z(x,\xi)\equiv0$ and impose periodic BCs on the spatial domain $x\in[0,1]$.

We compute the numerical solution until time $t=0.01$ (at which the solution is still smooth) on a sequence of uniform meshes with
$\dxi=2/16$, $2/32$, $2/64$, and $2/128$. We then estimate the $L^1$-errors and experimental convergence rates using the following Runge
formulae, which are based on the solutions computed on the three consecutive uniform grids with the mesh sizes $\dxi$, $2\dxi$, and $4\dxi$
and denoted by $(\cdot)^{\dxi}$, $(\cdot)^{2\dxi}$, and $(\cdot)^{4\dxi}$, respectively:
\begin{equation*}
{\rm Error}(\dxi)\approx\frac{\delta^2_{12}}{|\delta_{12}-\delta_{24}|},\quad
{\rm Rate}(\dxi)\approx\log_2\left(\frac{\delta_{24}}{\delta_{12}}\right).
\end{equation*}
Here, $\delta_{12}:=\|(\cdot)^{\dxi}-(\cdot)^{2\dxi}\|_{L^1}$ and $\delta_{24}:=\|(\cdot)^{2\dxi}-(\cdot)^{4\dxi}\|_{L^1}$. Since spatial
discretization has only the second order of accuracy and temporal discretization has the third order, the following restrictions are applied
when the grid refinement is performed:
\begin{equation*}
\dx\propto(\dxi)^{\frac{5}{2}},\quad\dt\le\frac{(\dx)^{\frac{3}{2}}}{2\max\limits_{j,\ell}\left\{\max\big(a_{\jph,\ell}^+,-a_{\jph,\ell}^-\big)\right\}}.
\end{equation*}
In particular, we take the following sequences of spatial ($\dx=1/200$, $1/2400$, $1/28800$, and $1/345600$), and temporary
($\dt=1.38\times10^{-4}$, $3.33\times10^{-6}$, $8.01\times10^{-8}$, and $1.93\times10^{-9}$) mesh sizes for $\dxi=2/16$, $2/32$, $2/64$, and
$2/128$, respectively.

The computed $L^1$-errors and corresponding convergence rates for water surface and discharge are reported in Table \ref{tab61}, where one
can clearly see that the fifth order of accuracy in $\xi$ has been achieved.
\begin{table}[ht!]
\centering
\begin{tabular}{|c|c|c|c|c|}
\hline
\multirow{2}{*}{$\dxi$}&\multicolumn{2}{c|}{$w$}&\multicolumn{2}{c|}{$hu$}\\\cline{2-5}
&Error&Rate&Error&Rate\\\hline
$2/64$&$3.76\times10^{-9}$&6.34&$1.95\times10^{-9}$&5.25\\
$2/128$&$5.79\times10^{-11}$&6.18&$4.89\times10^{-11}$&5.28\\
\hline
\end{tabular}
\caption{\sf Example 4: The $L^1$-errors and experimental convergence rates for water surface $w$ and dicharge $hu$.\label{tab61}}
\end{table}

\paragraph{Example 5 (Dam Break over Random Bottom Topography, $d=s=1$).} In this example proposed in \cite[\S5.1]{DEN21a}, we consider the
deterministic ICs
\begin{equation*}
w(x,0,\xi)=\left\{\begin{aligned}&1,&&x<0,\\&0.5,&&x>0,\end{aligned}\right.\qquad u(x,0,\xi)\equiv0,
\end{equation*}
and free BCs, prescribed in the spatial domain $x\in[-1,1]$, while introducing the uncertainty via the following random bottom topography:
\begin{equation*}
Z(x,\xi)=0.125\xi+\left\{\begin{aligned}
&0.125\big[\cos(5\pi x)+2\big],&&|x|<0.2,\\
&0.125,&&\mbox{otherwise},
\end{aligned}\right.\quad{\xi\sim{\cal U}(-1,1)}.
\end{equation*}

We compute the numerical solution until time $t=0.8$ on a uniform grid with $\dx=1/800$ and $\dxi=1/50$. The obtained water surface $w$ and
discharge $hu$ are shown in Figure \ref{fig61}, where one can observe that the numerical solution computed by the proposed finite-volume
method is substantially less oscillatory than the gPC-SG solutions reported in \cite{DEN21a}.
\begin{figure}[ht!]
\centerline{\includegraphics[width=0.37\textwidth]{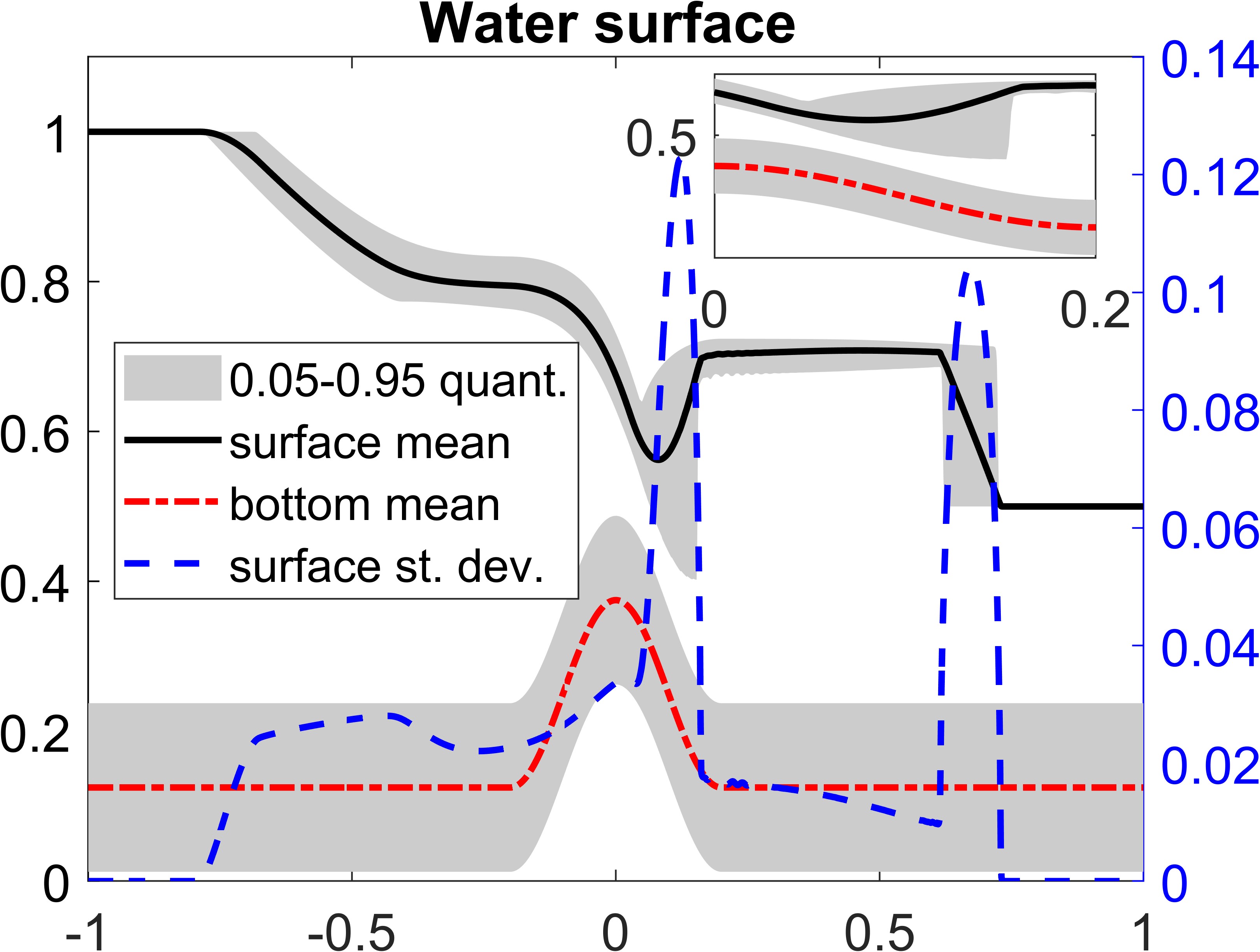}\hspace*{0.5cm}\includegraphics[width=0.37\textwidth]{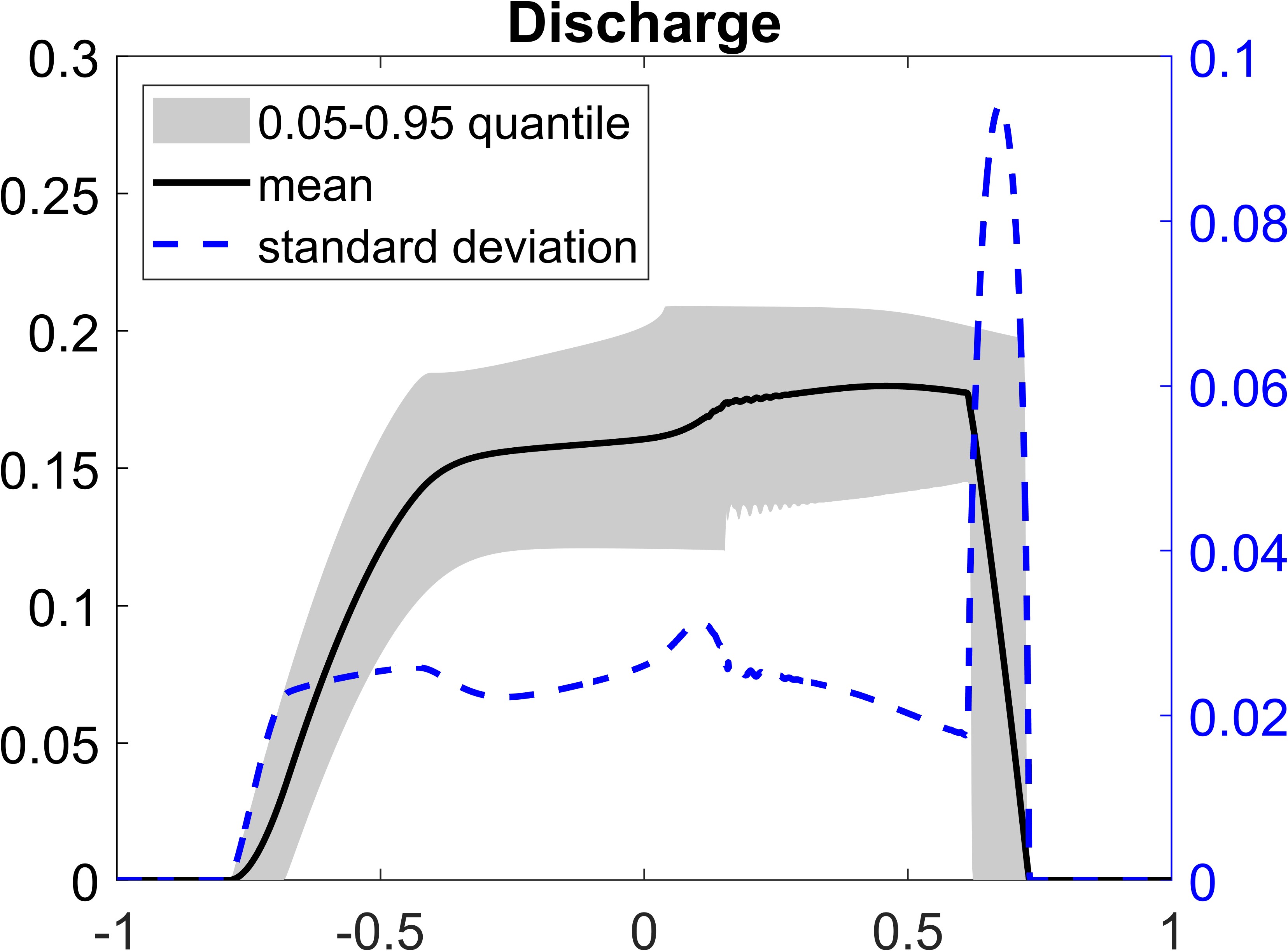}}
\caption{\sf Example 5: Mean, 95\%-quantile, and standard deviation of $w$ (left) and $hu$ (right).\label{fig61}}
\end{figure}

\paragraph*{Example 6 (Random Water Surface, $d=s=1$).} This example is taken from \cite[\S5.2]{DEN21a}. We now consider the deterministic
bottom topography
\begin{equation*}
Z(x,\xi)=\left\{\begin{aligned}
&10(x-0.3),&&0.3\le x\le0.4,\\
&1-0.0025\sin^2(25\pi x),&&0.4\le x\le0.6,\\
&-10(x-0.7),&&0.6\le x\le0.7,\\
&0,&&\mbox{otherwise},
\end{aligned}\right.
\end{equation*}
and the ICs with a randomly perturbed water surface
\begin{equation*}
w(x,0,\xi)=\left\{\begin{aligned}&1.001+0.001\xi,&&0.1<x<0.2,\\&1,&&\mbox{otherwise},\end{aligned}\right.\qquad u(x,0,\xi)\equiv0,\quad
{\xi\sim{\cal U}(-1,1)}.
\end{equation*}
The spatial domain is $x\in[-1,1]$ and free BCs are implemented at both ends of the interval. The water surface $w$ and discharge $hu$,
computed on a uniform grid with $\dx=1/800$ and $\dxi=1/50$ at time $t=1$ are plotted in Figure \ref{fig62}. The obtained results are
similar to those obtained in \cite{DEN21a} using a gPC-SG method, but our results are visibly sharper. We believe that the filter used in
\cite{DEN21a} causes additional smearing so that the main advantage of the gPC-SG method---spectral accuracy in the random space---is lost.
\begin{figure}[ht!]
\centerline{\includegraphics[width=0.39\textwidth]{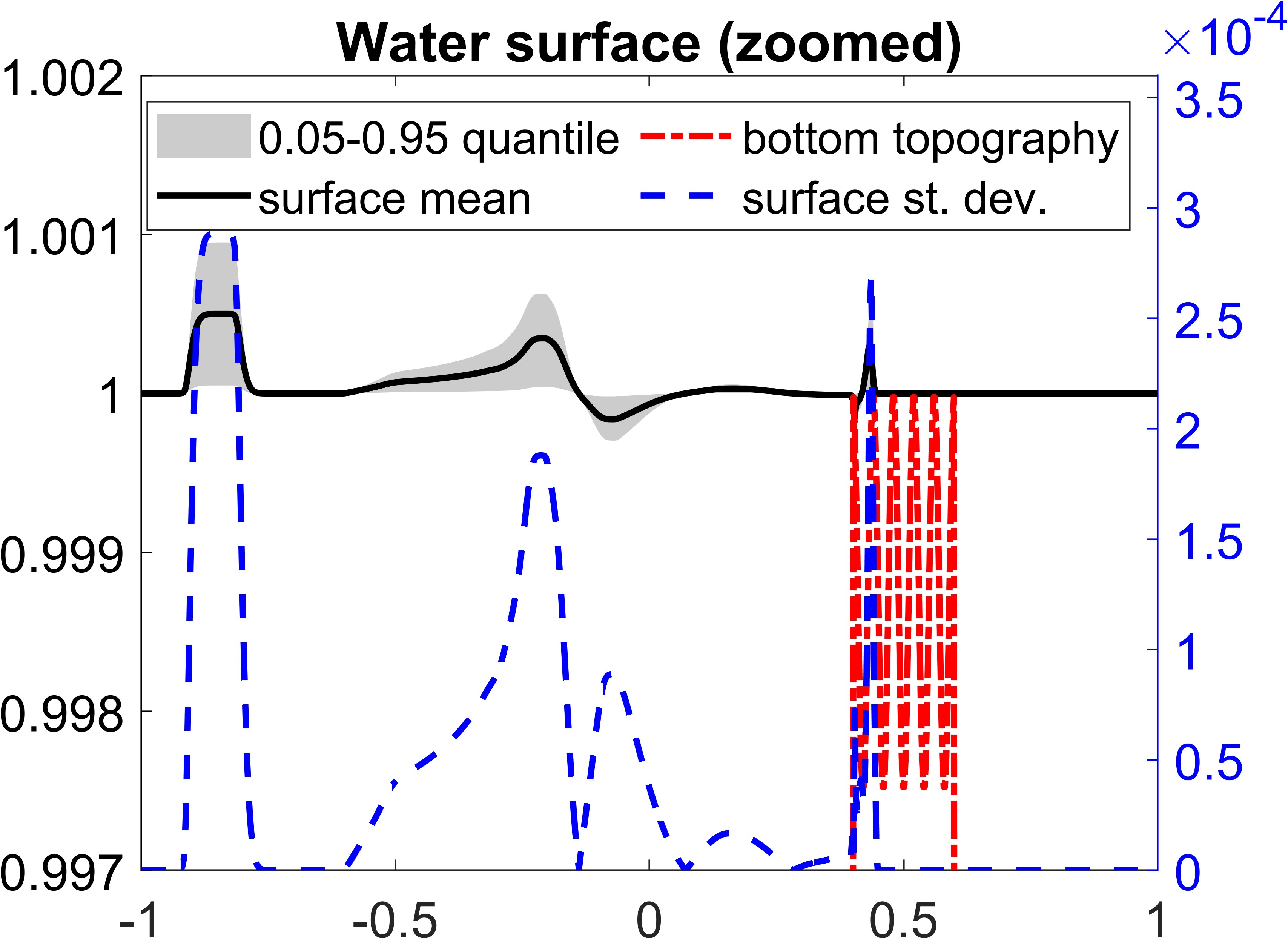}\hspace*{0.5cm}\includegraphics[width=0.37\textwidth]{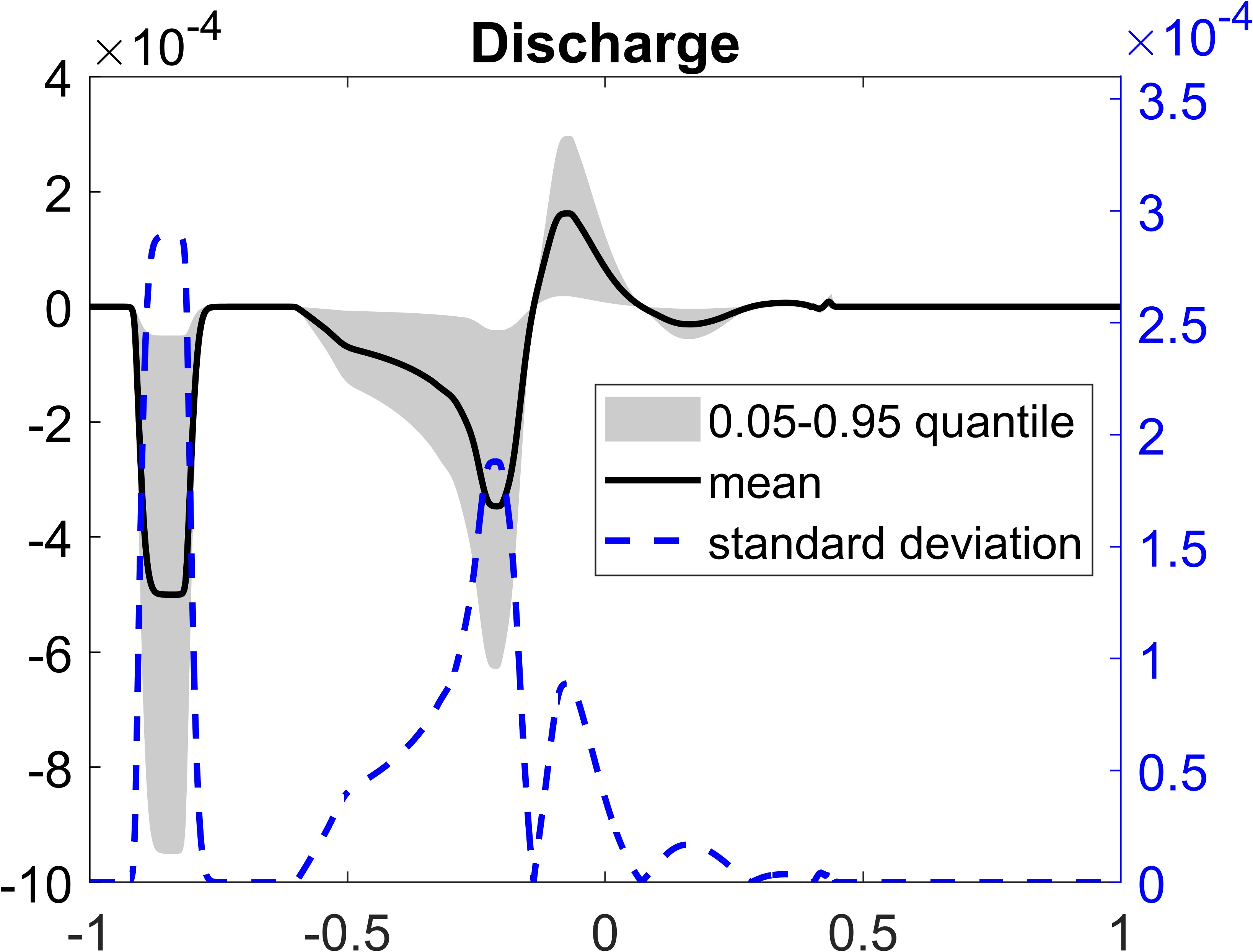}}
\caption{\sf Example 6: Mean, 95\%-quantile, and standard deviation of $w$ (left) and $hu$ (right).\label{fig62}}
\end{figure}

{
\paragraph{Example 7 (Random Discontinuous Bottom, $d=1$, $s=1$).} In this example taken from \cite[\S5.3]{DEN21a}, the deterministic ICs
for the water surface and discharge,
\begin{equation*}
w(x,0,\xi)=\left\{\begin{aligned}&5,&&x\le0.5,\\&1.6,&&\mbox{otherwise},\end{aligned}\right.\quad
u(x,\xi,0)=\left\{\begin{aligned}&1,&&x\le0.5,\\&-2,&&\mbox{otherwise},\end{aligned}\right.
\end{equation*}
and a stochastic discontinuous bottom,
\begin{equation*}
Z(x,\xi)=\left\{\begin{aligned}&1.5+0.1\xi,&&x\le0.5,\\&1.1+0.1\xi,&&\mbox{otherwise},\end{aligned}\right.
\end{equation*}
are prescribed in the spatial domain $x\in[0,1]$ with free BCs implemented at both ends of the interval. Here, $\xi$ is a random variable
with the PDF
\begin{equation*}
\nu(\xi)=\frac{(1-\xi)^3(1+\xi)}{32\,{\cal B}(2,4)},
\end{equation*}
which corresponds to the Beta distribution over $[-1,1]$ and ${\cal B}$ stands for the Beta function.

We compute the solution on a uniform mesh with $\dx=1/401$ and $\dxi=1/50$ until the final time $T=0.15$. The obtained water surface and
discharge are plotted in \fref{fig62a}. While our results are similar to those reported in \cite[\S5.3]{DEN21a}, the over- and undershoots
seen in the neighborhood of the bottom discontinuity in \fref{fig62a} are smaller despite the fact that the results in
\cite[Figure 5]{DEN21a} were filtered to prevent large oscillations.
}
\begin{figure}[ht!]
\centerline{\includegraphics[width=0.37\textwidth]{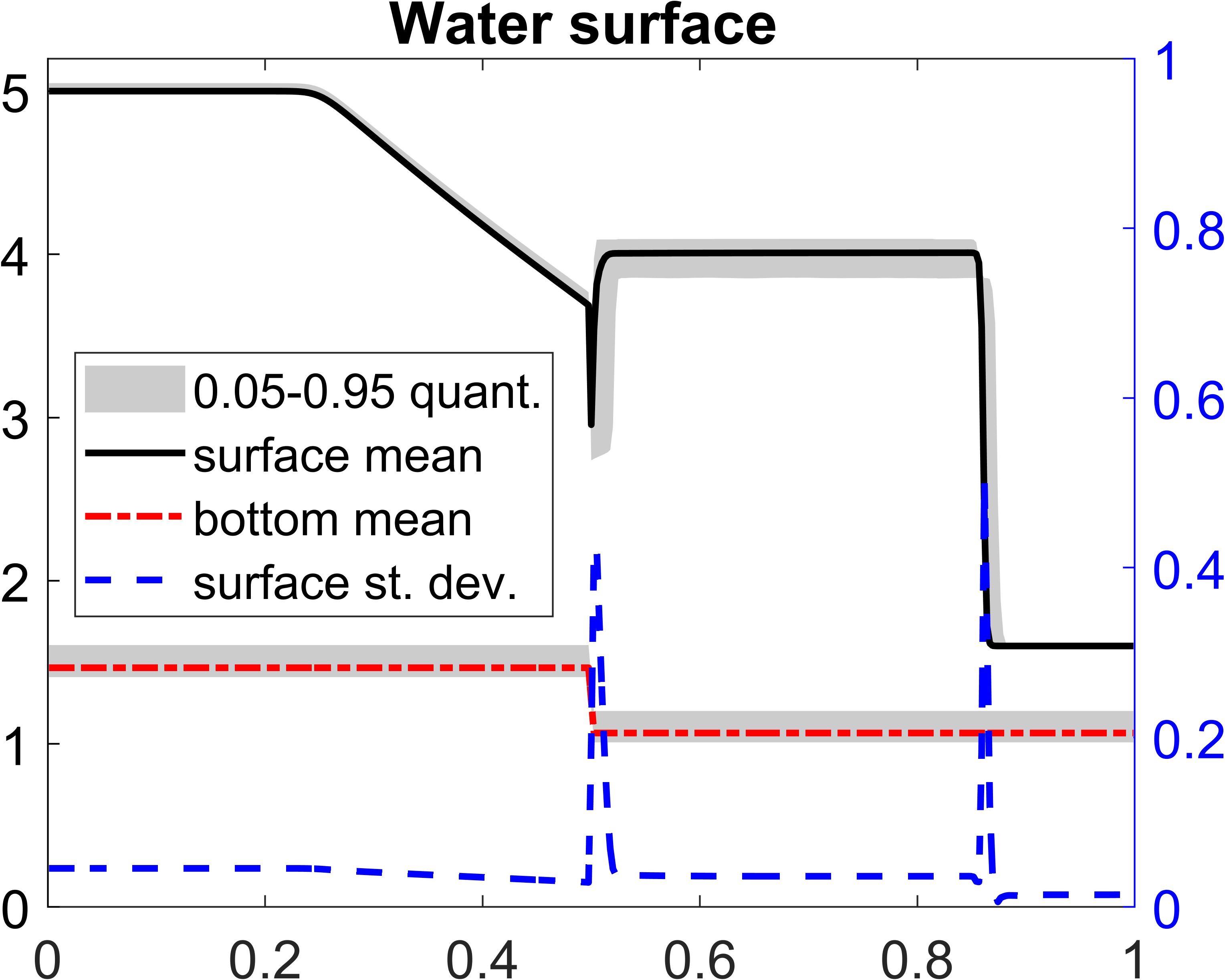}\hspace*{0.5cm}
            \includegraphics[width=0.37\textwidth]{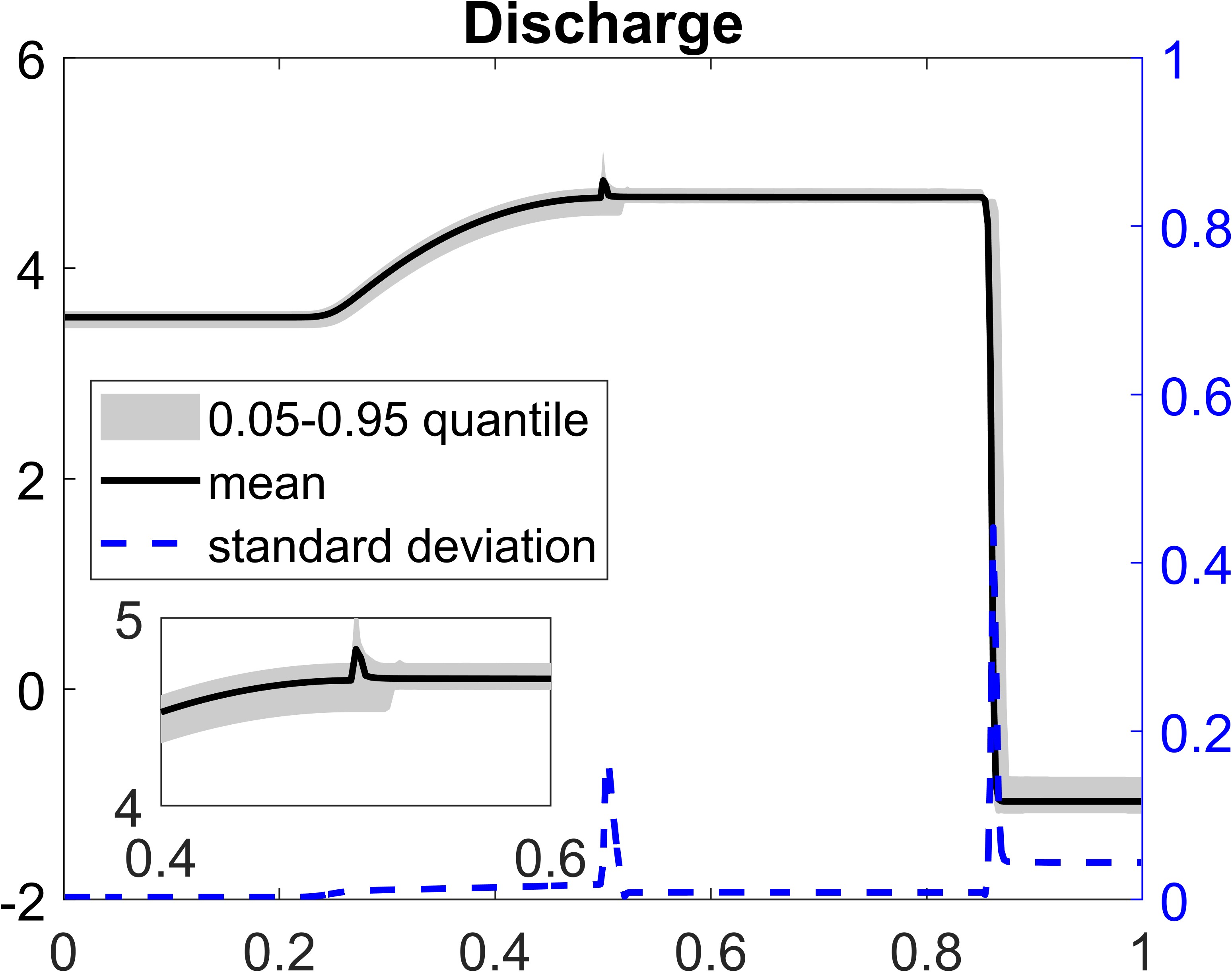}}
\caption{\sf{Example 7: Mean, 95\%-quantile, and standard deviation of $w$ (left) and $hu$ (right).}\label{fig62a}}
\end{figure}

\paragraph{Example 8 (Random Water Surface and Discharge, $d=1$, $s=2$).} We take random ICs for perturbed water surface and discharge,
\begin{equation*}
\begin{aligned}
&w(x,\xi,\eta,0)=\max\big\{0.9,Z(x,\xi,\eta)\big\}+\left\{
\begin{aligned}
&0.1(1+0.1\xi),&&-1.6<x<-1.4,\\&0,&&\text{otherwise},
\end{aligned}
\right.\\
&u(x,\xi,\eta,0)=0.01\eta,\quad{\xi\sim{\cal U}(-1,1),\quad\eta\sim{\cal U}(-1,1)},
\end{aligned}
\end{equation*}
and the following deterministic bottom topography:
\begin{equation*}
Z(x,\xi,\eta)=\left\{\begin{aligned}&(1-x^2),&&-1<x<1,\\&0,&&\text{otherwise},\end{aligned}\right.
\end{equation*}
all considered in the spatial domain is $x\in[-4,2]$ with free BCs implemented at both ends of the interval.

In this example, the ICs give rise to two waves, which emerge in the area $x\in[-1.6,-1.4]$ and then propagate in the opposite directions.
The right-moving wave reaches the initially dry area and then moves over the ``island'', leading to the occurrence of wetting/drying
processes. This makes the studied problem numerically challenging. However, the proposed finite-volume method is capable of robustly
capturing the solution as shown in Figure \ref{fig63}, where time snapshots of the mean and standard deviation of the water surface,
computed on a uniform mesh with $\dx=1/100$ and $\dxi=\deta=1/10$, are plotted.
\begin{figure}[ht!]
\centerline{\includegraphics[width=0.32\textwidth]{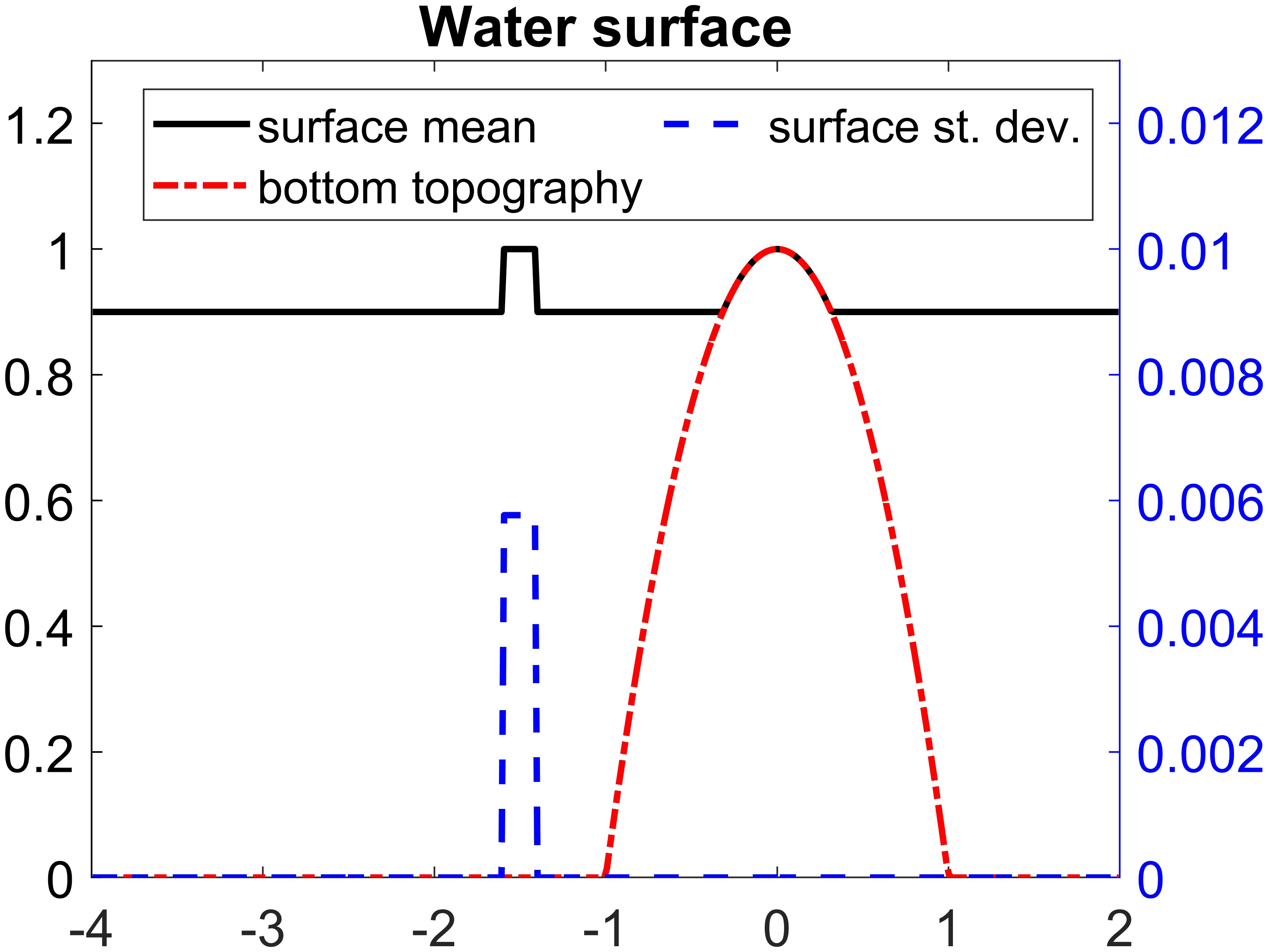}\hspace*{0.5cm}\includegraphics[width=0.32\textwidth]{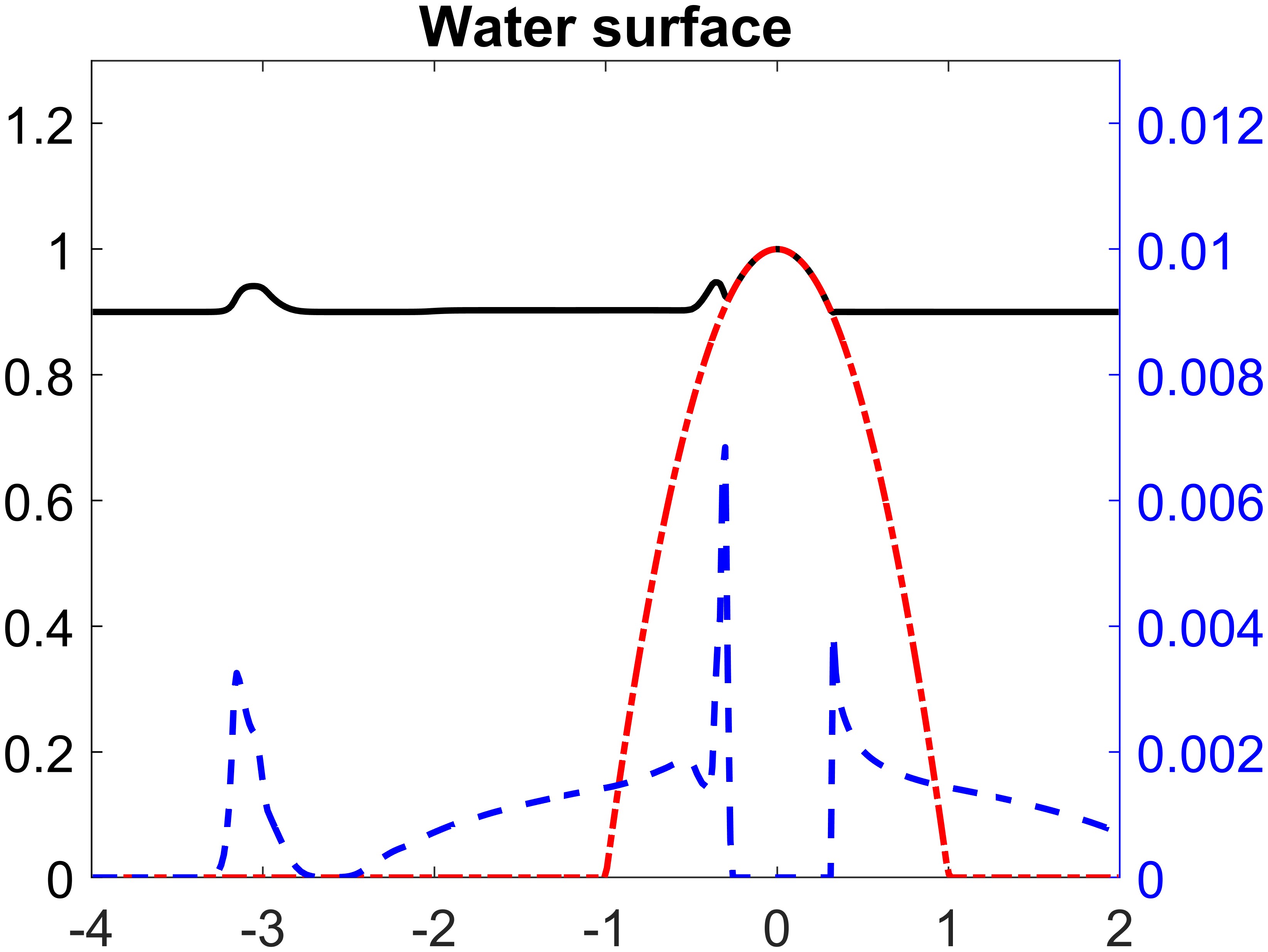}}
\vskip 8pt
\centerline{\includegraphics[width=0.32\textwidth]{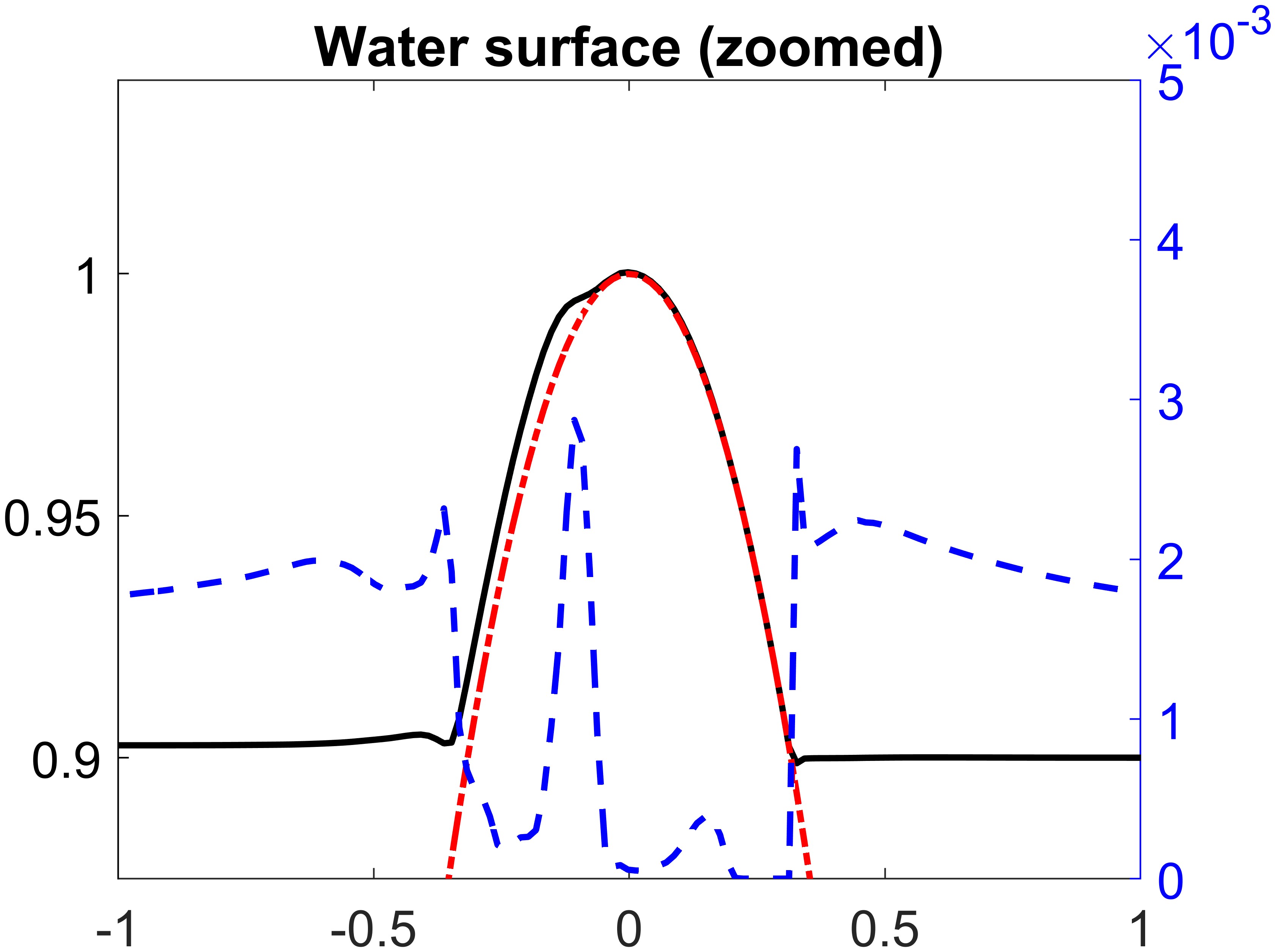}\hspace*{0.5cm}
            \includegraphics[width=0.32\textwidth]{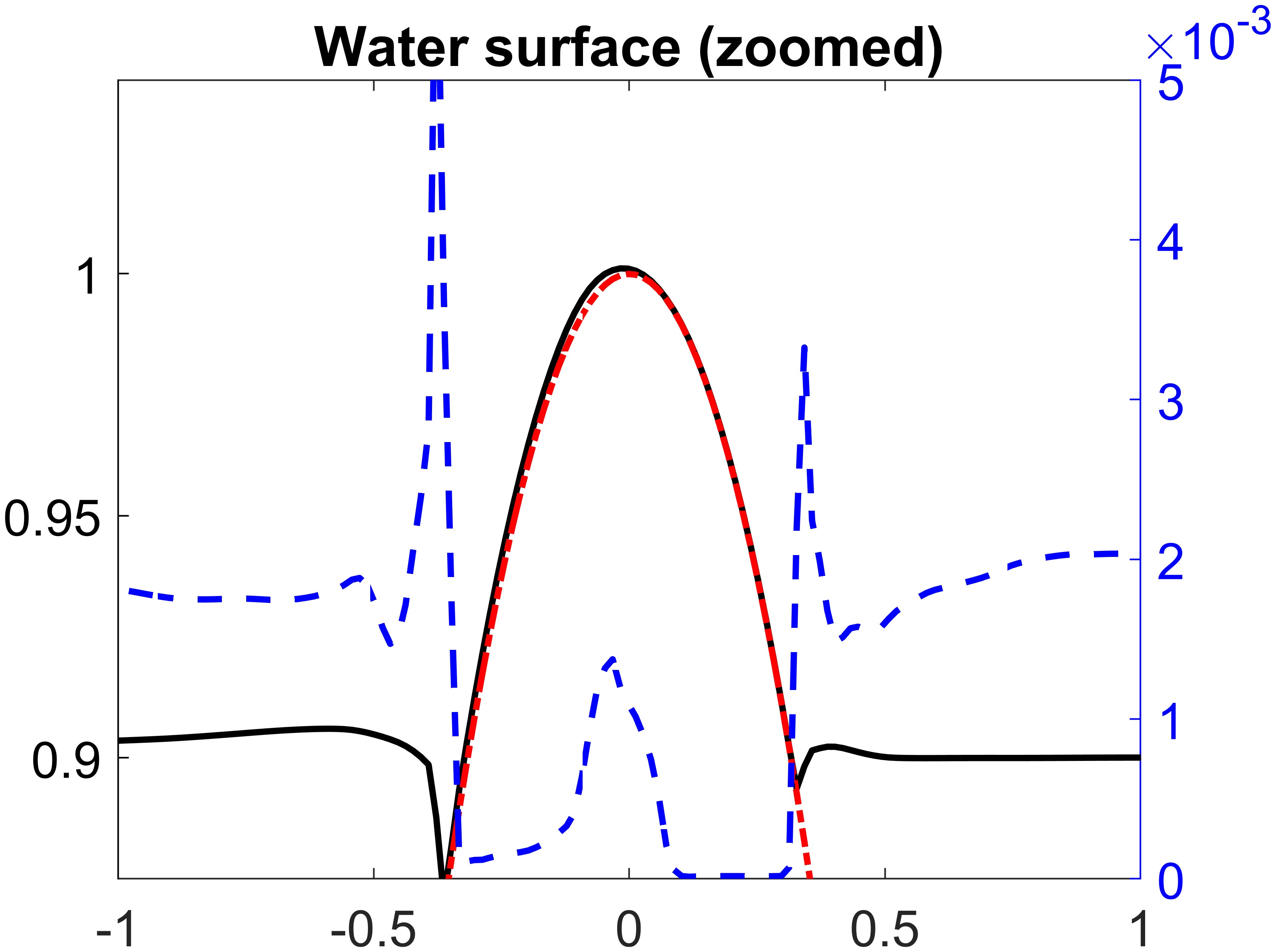}}
\caption{\sf Example 8: Mean and standard deviation of $w$ at $t=0$ (top left), 0.5 (top right), 0.75 (bottom left), and 1 (bottom right).
\label{fig63}}
\end{figure}

\paragraph{Example 9 (Random Bottom Topography, $d=2$, $s=1$).} In the final example, we consider the deterministic ICs
\begin{equation*}
w(x,y,\xi,0)=\left\{\begin{aligned}&\max\big\{1.01,Z(x,y)\big\},&&0.05<x<0.15,\\&\max\big\{1,Z(x,y)\big\},&&\mbox{otherwise}\end{aligned}
\right.\qquad u(x,y,\xi,0)\equiv0,
\end{equation*}
and the following random bottom topography:
\begin{equation*}
Z(x,y,\xi)=0.8e^{-5(x-0.9)^2-50(y-0.5)^2}+0.1(\xi+1),\quad{\xi\sim{\cal U}(-1,1)},
\end{equation*}
in the spatial domain $(x,y)\in[0,2]\times[0,1]$ with free BCs prescribed at all sides of its boundary.

In Figure \ref{fig64}, we plot the mean and standard deviation of the water surface computed on a uniform mesh with $\dx=\dy=1/200$ and
$\dxi=1/10$ at time $t=1.2$. As in Example 3, the standard deviation identifies the areas in which the uncertainties the solution the most.
\begin{figure}[ht!]
\centerline{\includegraphics[width=0.32\textwidth]{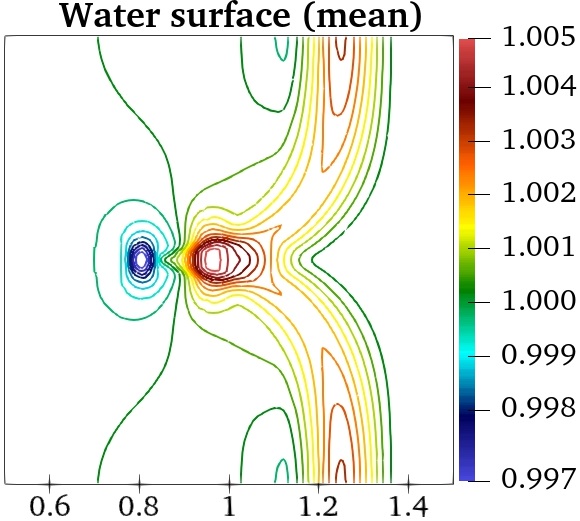}\hspace*{1.0cm}\includegraphics[width=0.32\textwidth]{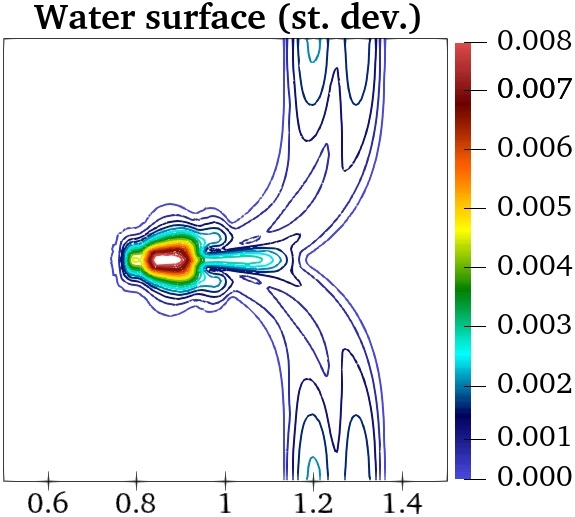}}
\caption{\sf Example 9: Mean (left) and standard deviation (right) of $w$.\label{fig64}}
\end{figure}

\section{Conclusions}\label{sec7}
In this paper, we have introduced new well-balanced and positivity preserving finite-volume methods for nonlinear hyperbolic systems of PDEs
with uncertainties. These methods are based on finite-volume approximation in both the spatial and random directions. In order to achieve a
high order of accuracy, we have applied second-order minmod reconstruction in physical space and fifth-order Ai-WENO-Z interpolation in
random directions. The lack of boundary conditions in the random space requires the development of a special high-order interpolation
technique near the boundary. This has been achieved by designing a new one-sided Ai-WENO-Z interpolation. Even though, our numerical
experiments did not show any sensitivity towards this choice, we note that it may possibly lead to instabilities if, for instance, strong
discontinuities travel to a random boundary. We also note that the proposed high-order finite-volume method is well suited for hyperbolic
systems with low or medium dimensional uncertainties. In order to overcome the curse of dimensionality, we suggest to apply GPU
parallelization for high-dimensional random spaces. Both the problem of boundary conditions and curse of dimensionality in random space
constitute interesting topics for future research.

\subsubsection*{Conflict of Interest Statement}
On behalf of all authors, the corresponding author states that there is no conflict of interest.

\begin{acknowledgment}
The work of A. Chertock was supported in part by NSF grant DMS-2208438. The work of M. Herty was supported in part by the DFG (German
Research Foundation) through 20021702/GRK2326, 333849990/IRTG-2379, HE5386/ 18-1, 19-2, 22-1, 23-1, and under Germany’s Excellence Strategy
EXC-2023 Internet of Production 390621612. The work of A. Kurganov was supported in part by NSFC grant 12171226 and the fund of the
Guangdong Provincial Key Laboratory of Computational Science and Material Design (No. 2019B030301001). The work of
M. Luk\'a\v{c}ov\'a-Medvi{\softd}ov\'a has been funded by the DFG under the SFB/TRR 146 Multiscale Simulation Methods for Soft Matter
Systems. M. Luk\'a\v{c}ov\'a-Medvi{\softd}ov\'a gratefully acknowledges the support of the Gutenberg Research College of University Mainz
and the Mainz Institute of Multiscale Modeling. A. Chertock and A. S. Iskhakov also acknowledge the support of the LeRoy B. Martin, Jr.
Distinguished Professorship Foundation. M. Herty and M. Luk\'a\v{c}ov\'a-Medvi{\softd}ov\'a also acknowledge the support of the DFG through
the project 525853336 funded within the Focused Programme SPP 2410 ``Hypebrolic Balance Laws: Complexity, Scales and Randomness''.
\end{acknowledgment}

\bibliographystyle{siamnodash}
\bibliography{ref}

\end{document}